\documentclass[12pt]{article}
\usepackage{mathtools,amssymb,amsfonts,amsthm,url,xspace,graphicx}
\usepackage{xifthen}
\usepackage[usenames,dvipsnames]{xcolor}
\usepackage[pdftitle={The Green's function on the double cover of the grid and application to the uniform spanning tree trunk},
            pdfauthor={Richard W. Kenyon and David B. Wilson},
            colorlinks=true,linkcolor=NavyBlue,urlcolor=RoyalBlue,citecolor=PineGreen,
            bookmarks=true,bookmarksopen=true,bookmarksopenlevel=2,unicode=true]{hyperref}
\usepackage{pdfcomment}
\usepackage[outline]{contour}
\usepackage{tikz}
\usepackage{pgfplots}
\usetikzlibrary{arrows.meta}
\tikzset{>={Latex[width=2.2mm,length=2.5mm]}}
\usepackage{microtype}
\usepackage{textgreek}
\usepackage[all]{hypcap}

\def\@rst #1 #2other{#1}
\newcommand\MR[1]{\relax\ifhmode\unskip\spacefactor3000 \space\fi
  \MRhref{\expandafter\@rst #1 other}{#1}}
\newcommand{\MRhref}[2]{\href{http://www.ams.org/mathscinet-getitem?mr=#1}{MR#1}}

\newcommand{\sref}[1]{\hyperref[#1]{\S~\ref*{#1}}}
\newcommand{\aref}[1]{\hyperref[#1]{Appendix~\ref*{#1}}}
\newcommand{\lref}[1]{\hyperref[#1]{Lemma~\ref*{#1}}}
\newcommand{\tref}[1]{\hyperref[#1]{Theorem~\ref*{#1}}}
\newcommand{\cref}[1]{\hyperref[#1]{Corollary~\ref*{#1}}}
\newcommand{\fref}[1]{\hyperref[#1]{Figure~\ref*{#1}}}
\newcommand{\pref}[1]{\hyperref[#1]{Proposition~\ref*{#1}}}

\def\clap#1{\hbox to 0pt{\hss#1\hss}}

\newcommand{\old}[1]{}

\makeatletter
\newif\if@borderstar
   \def\bordermatrix{\@ifnextchar*{%
       \@borderstartrue\@bordermatrix@i}{\@borderstarfalse\@bordermatrix@i*}%
   }
   \def\@bordermatrix@i*{\@ifnextchar[{\@bordermatrix@ii}{\@bordermatrix@ii[()]}}
   \def\@bordermatrix@ii[#1]#2{%
   \begingroup
     \m@th\@tempdima8.75\p@\setbox\z@\vbox{%
       \def\cr{\crcr\noalign{\kern 2\p@\global\let\cr\endline }}%
       \ialign {$##$\hfil\kern 2\p@\kern\@tempdima & \thinspace %
       \hfil $##$\hfil && \quad\hfil $##$\hfil\crcr\omit\strut %
       \hfil\crcr\noalign{\kern -\baselineskip}#2\crcr\omit %
       \strut\cr}}%
     \setbox\tw@\vbox{\unvcopy\z@\global\setbox\@ne\lastbox}%
     \setbox\tw@\hbox{\unhbox\@ne\unskip\global\setbox\@ne\lastbox}%
     \setbox\tw@\hbox{%
       $\kern\wd\@ne\kern -\@tempdima\left\@firstoftwo#1%
         \if@borderstar\kern2pt\else\kern -\wd\@ne\fi%
       \global\setbox\@ne\vbox{\box\@ne\if@borderstar\else\kern 2\p@\fi}%
       \vcenter{\if@borderstar\else\kern -\ht\@ne\fi%
         \unvbox\z@\kern-\if@borderstar2\fi\baselineskip}%
         \if@borderstar\kern-2\@tempdima\kern2\p@\else\,\fi\right\@secondoftwo#1 $%
     }\null \;\vbox{\kern\ht\@ne\box\tw@}%
   \endgroup
   }
\makeatother

\newtheorem{thm}{Theorem}
\numberwithin{figure}{section}
\numberwithin{thm}{section}
\numberwithin{equation}{section}
\newtheorem{theorem}[thm]{Theorem}

\newtheorem{lemma}[thm]{Lemma}

\newtheorem{corollary}[thm]{Corollary}
\theoremstyle{definition}\newtheorem{remark}[thm]{Remark}
\theoremstyle{definition}
\newcommand{\Z}{\mathbb{Z}}
\newcommand{\Q}{\mathbb{Q}}
\newcommand{\R}{\mathbb{R}}
\renewcommand{\C}{\mathbb{C}}

\renewcommand{\G}{\mathcal{G}}
\renewcommand{\H}{\mathcal{H}}

\renewcommand{\t}{{\text t}}

\newcommand{\be}{\begin{equation}}
\newcommand{\ee}{\end{equation}}

\newcommand{\sign}{\operatorname{sign}}

\newcommand{\restrictboxwidth}[2]{%
  \begingroup
    \sbox0{\ignorespaces#1\unskip}%
    \leavevmode
    \ifdim\wd0>#2
      \hbox to #2{%
        \hss\resizebox{#2}{!}{\copy0 }\hss
      }%
    \else
      \copy0 %
    \fi
  \endgroup
}
\newcommand{\rsbox}[1]{\restrictboxwidth{$#1$}{24pt}}
\newcommand{\rnode}[2]{\node at (axis cs:#1) {\rsbox{#2}};}

\title{\vspace{-40pt}The Green's function on the double cover of the grid \\ and application to the uniform spanning tree trunk}
\author{\href{http://www.math.brown.edu/~rkenyon/}{Richard W. Kenyon}\thanks{Brown University, Providence, RI 02912, USA; richard\_kenyon@brown.edu.} \and \href{http://dbwilson.com}{David B. Wilson}\thanks{University of Washington, Seattle, WA 98195, USA; dbwilson@uw.edu.}}
\date{}
\begin{document}

\maketitle
\vspace{-16pt}
\begin{abstract}
  We compute the Green's function on the double cover of $\Z^2$, branched over a vertex or a face.
  We use this result to compute the local statistics of the ``trunk'' of the
  uniform spanning tree on the square lattice, i.e., the limiting
  probabilities of cylinder events conditional on the path connecting
  far away points passing through a specified edge.  We also show how
  to compute the local statistics of large-scale triple points of the
  uniform spanning tree, where the trunk branches.  The method reduces
  the problem to a dimer system with isolated monomers, and we compute
  the inverse Kasteleyn matrix using the Green's function on the double
  cover of the square lattice.  For the trunk, the probabilities of
  cylinder events are in $\Q[\sqrt{2}]$, while for the triple points
  the probabilities are in $\Q[1/\pi]$.
\end{abstract}
\vspace{-36pt}
\renewcommand{\contentsname}{\empty}
\tableofcontents

\section{Introduction}

The Green's function, or inverse Laplacian,
on a graph is of fundamental importance in the study of
random walks, potential theory, and statistical mechanics models on graphs.
In its simplest probabilistic interpretation, the Green's function with
Dirichlet boundary conditions $G(v_1,v_2)$ is the expected
time spent at a vertex $v_2$ before hitting the boundary, of a simple random walk started at $v_1$ run in continuous time.

The Green's function on $\Z^2$ is of particular interest
because that is a natural setting of simple random walk (SRW) and other statistical
mechanics models where the Laplacian plays a role, like the uniform spanning tree (UST),
the dimer model, and the Ising model.

We study here the Green's function and a closely related operator, the inverse
Kasteleyn matrix, on the double cover of~$\Z^2$ branched over a face or a vertex.
These operators are important in a number of different situations:

\begin{enumerate}
\item In properties of the loop-erased random walk, or the uniform spanning tree ``centered on the trunk''.
\item In dimer covers of $\Z^2$ with holes at specified locations (see \cite{Fisher-Stephenson}).
\item Simple random walk (SRW) on $\Z^2$ with boundary consisting of a ray $\{(x,0)\;:\;x\le 0\}$ (see \cite{BMS1, BMS2})
\item SRW on $\Z^2$ with boundary consisting of a diagonal ray $\{(x,x)\;:\;|x|\le 0\}$.
\item In the study of how the SRW winds around a face of $\Z^2$ \cite{Budd}.
\item In the critical Ising model on $\Z^2$ with a ``disorder insertion'' (see \cite{CCK:ising}).
\end{enumerate}

We compute an exact expression for the Green's functions
for the double branched covers of~$\Z^2$, branched over a vertex or a face.
As shown below, these are closely related to the Green's function for
SRW on $\Z^2$ with boundary consisting of a ray $\{(x,0)\;:\;x\le 0\}$ or the diagonal ray $\{(x,x)\;:\;|x|\le 0\}$.

As an application, we show that the spanning tree edge probabilities near the
two-ended loop-erased random walk on
$\Z^2$ (the trunk of the uniform spanning tree) form a \textit{determinantal process}, with an explicit kernel.
The local edge probabilities are shown in Figure~\ref{neartrunkprobs}.
For example we show that the probability that a vertex on the trunk has degree $2$, $3$ or $4$ is respectively
$\frac12$, $\sqrt{2}-1$ and $\frac32-\sqrt{2}$.
We also prove a surprising ``geometric runs'' property: the probability that the trunk goes straight at least $k$ times is $(\sqrt{2}-1)^k$.
This is a finite (in~$k$) version of an earlier asymptotic (in $k$) result \cite{kenyon:det-lap} (also proved in \cite[ex.~5.5]{MR2030095}).

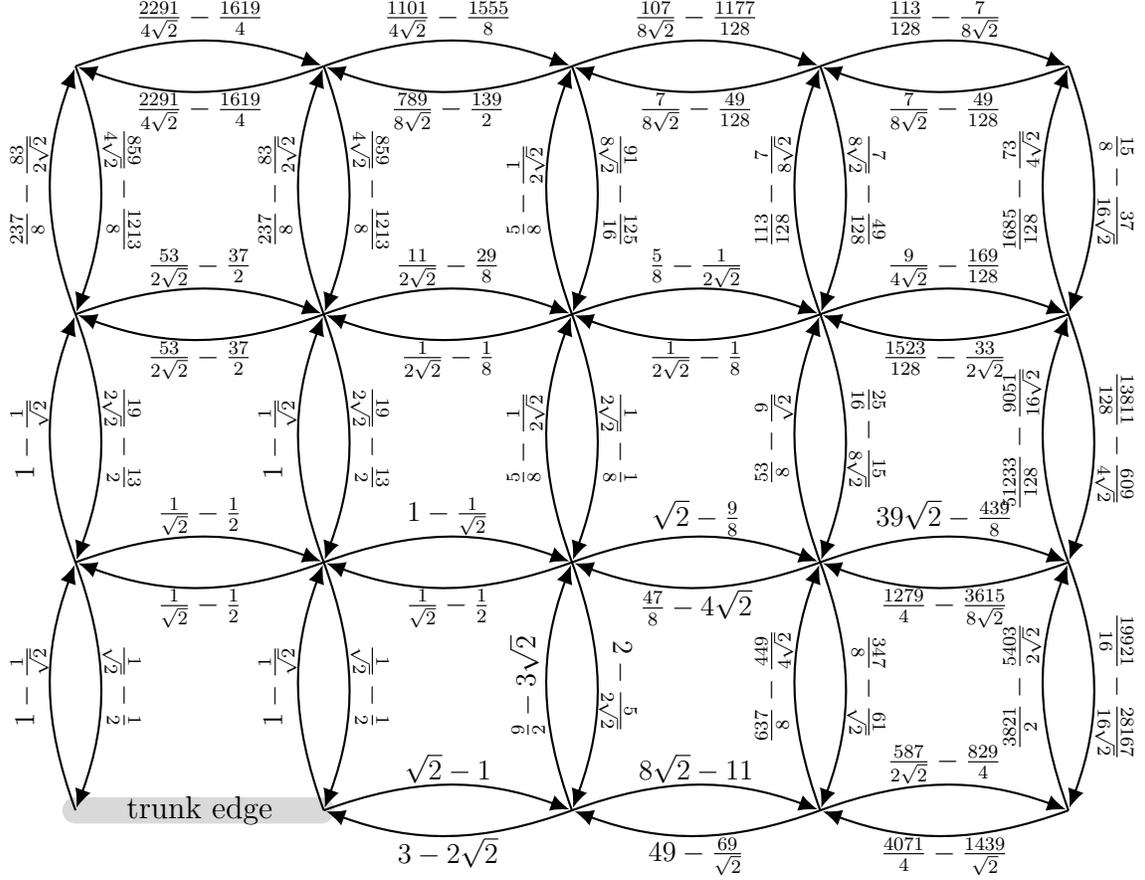
\begin{figure}[t]
\centerline{\begin{tikzpicture}[
  scale=3.3,
  ->,
  bend angle=20,
  shorten >=1pt,
  auto,
  node distance=2cm,
  thick,
  every state/.style={fill=red,draw=none,text=white},
]
\newcommand{\edgeR}[2]{(#1) edge [bend left] node [scale=0.9,inner sep=0.5pt] {$#2$} ++(+1,0)}
\newcommand{\edgeL}[2]{(#1) edge [bend left] node [scale=0.9,inner sep=0.5pt, anchor=north] {$#2$} ++(-1,0)}
\newcommand{\edgeU}[2]{(#1) edge [bend left] node [scale=0.9,inner sep=0.5pt] {$#2$} ++(0,+1)}
\newcommand{\edgeD}[2]{(#1) edge [bend left] node [scale=0.9,inner sep=0.5pt] {$#2$} ++(0,-1)}
\draw [gray!30,thick,-,line width=10pt,line cap=round] (-1,0) -- (0,0);
\node at (-0.5,0) {trunk edge};
\path[every node/.style={sloped,anchor=south,auto=false}]
\edgeL{0,1}{\frac{1}{\sqrt{2}}-\frac{1}{2}}
\edgeL{0,2}{\frac{53}{2 \sqrt{2}}-\frac{37}{2}}
\edgeL{0,3}{\frac{2291}{4 \sqrt{2}}-\frac{1619}{4}}
\edgeL{1,0}{3-2 \sqrt{2}}
\edgeL{1,1}{\frac{1}{\sqrt{2}}-\frac{1}{2}}
\edgeL{1,2}{\frac{1}{2 \sqrt{2}}-\frac{1}{8}}
\edgeL{1,3}{\frac{789}{8 \sqrt{2}}-\frac{139}{2}}
\edgeL{2,0}{49-\frac{69}{\sqrt{2}}}
\edgeL{2,1}{\frac{47}{8}-4 \sqrt{2}}
\edgeL{2,2}{\frac{1}{2 \sqrt{2}}-\frac{1}{8}}
\edgeL{2,3}{\frac{7}{8 \sqrt{2}}-\frac{49}{128}}
\edgeL{3,0}{\frac{4071}{4}-\frac{1439}{\sqrt{2}}}
\edgeL{3,1}{\frac{1279}{4}-\frac{3615}{8 \sqrt{2}}}
\edgeL{3,2}{\frac{1523}{128}-\frac{33}{2 \sqrt{2}}}
\edgeL{3,3}{\frac{7}{8 \sqrt{2}}-\frac{49}{128}}
\edgeR{-1,1}{\frac{1}{\sqrt{2}}-\frac{1}{2}}
\edgeR{-1,2}{\frac{53}{2 \sqrt{2}}-\frac{37}{2}}
\edgeR{-1,3}{\frac{2291}{4 \sqrt{2}}-\frac{1619}{4}}
\edgeR{0,0}{\sqrt{2}-1}
\edgeR{0,1}{1-\frac{1}{\sqrt{2}}}
\edgeR{0,2}{\frac{11}{2 \sqrt{2}}-\frac{29}{8}}
\edgeR{0,3}{\frac{1101}{4 \sqrt{2}}-\frac{1555}{8}}
\edgeR{1,0}{8 \sqrt{2}-11}
\edgeR{1,1}{\sqrt{2}-\frac{9}{8}}
\edgeR{1,2}{\frac{5}{8}-\frac{1}{2 \sqrt{2}}}
\edgeR{1,3}{\frac{107}{8 \sqrt{2}}-\frac{1177}{128}}
\edgeR{2,0}{\frac{587}{2 \sqrt{2}}-\frac{829}{4}}
\edgeR{2,1}{39 \sqrt{2}-\frac{439}{8}}
\edgeR{2,2}{\frac{9}{4 \sqrt{2}}-\frac{169}{128}}
\edgeR{2,3}{\frac{113}{128}-\frac{7}{8 \sqrt{2}}}
\edgeU{-1,0}{1-\frac{1}{\sqrt{2}}}
\edgeU{-1,1}{1-\frac{1}{\sqrt{2}}}
\edgeU{-1,2}{\frac{237}{8}-\frac{83}{2 \sqrt{2}}}
\edgeU{0,0}{1-\frac{1}{\sqrt{2}}}
\edgeU{0,1}{1-\frac{1}{\sqrt{2}}}
\edgeU{0,2}{\frac{237}{8}-\frac{83}{2 \sqrt{2}}}
\edgeU{1,0}{\frac{9}{2}-3 \sqrt{2}}
\edgeU{1,1}{\frac{5}{8}-\frac{1}{2 \sqrt{2}}}
\edgeU{1,2}{\frac{5}{8}-\frac{1}{2 \sqrt{2}}}
\edgeU{2,0}{\frac{637}{8}-\frac{449}{4 \sqrt{2}}}
\edgeU{2,1}{\frac{53}{8}-\frac{9}{\sqrt{2}}}
\edgeU{2,2}{\frac{113}{128}-\frac{7}{8 \sqrt{2}}}
\edgeU{3,0}{\frac{3821}{2}-\frac{5403}{2 \sqrt{2}}}
\edgeU{3,1}{\frac{51233}{128}-\frac{9051}{16 \sqrt{2}}}
\edgeU{3,2}{\frac{1685}{128}-\frac{73}{4 \sqrt{2}}}
\edgeD{-1,1}{\frac{1}{\sqrt{2}}-\frac{1}{2}}
\edgeD{-1,2}{\frac{19}{2 \sqrt{2}}-\frac{13}{2}}
\edgeD{-1,3}{\frac{859}{4 \sqrt{2}}-\frac{1213}{8}}
\edgeD{0,1}{\frac{1}{\sqrt{2}}-\frac{1}{2}}
\edgeD{0,2}{\frac{19}{2 \sqrt{2}}-\frac{13}{2}}
\edgeD{0,3}{\frac{859}{4 \sqrt{2}}-\frac{1213}{8}}
\edgeD{1,1}{2-\frac{5}{2 \sqrt{2}}}
\edgeD{1,2}{\frac{1}{2 \sqrt{2}}-\frac{1}{8}}
\edgeD{1,3}{\frac{91}{8 \sqrt{2}}-\frac{125}{16}}
\edgeD{2,1}{\frac{347}{8}-\frac{61}{\sqrt{2}}}
\edgeD{2,2}{\frac{25}{16}-\frac{15}{8 \sqrt{2}}}
\edgeD{2,3}{\frac{7}{8 \sqrt{2}}-\frac{49}{128}}
\edgeD{3,1}{\frac{19921}{16}-\frac{28167}{16 \sqrt{2}}}
\edgeD{3,2}{\frac{13811}{128}-\frac{609}{4 \sqrt{2}}}
\edgeD{3,3}{\frac{15}{8}-\frac{37}{16 \sqrt{2}}}
;
\end{tikzpicture}}
\caption{Uniform spanning tree edge probabilities near an edge conditioned to be on the trunk.
From each vertex there is a unique directed path to infinity that avoids the
edge conditioned to be on the trunk; here we give the directed edge probabilities.}
\label{neartrunkprobs}
\end{figure}

For the spanning tree trunk on the triangular lattice we also prove a
geometric runs property, where the probability that the trunk
continues straight $k$ times is $(2-\sqrt{3})^k$; however on this
lattice, other local statistics such as degree probabilities seem to
require more ideas to compute.

Lawler \cite{lawler:trunk-measure} also constructs the measure near the trunk on $\Z^2$ (and $\Z^3$), but without giving
explicit values, or showing that it is determinantal for $\Z^2$.

As another application we compute the dimer edge probabilities for the
dimer model on~$\Z^2$ with a fixed monomer at the origin (with
``flat'' boundary conditions).  Again these dimer probabilities are
a determinantal process with an explicit kernel.  See Figure~\ref{nearholeprobs}.

A curious consequence of the calculation is that the edge probabilities of the above types (and indeed the relevant Green's function) take values in $\Q[\sqrt{2}]$;
recall that the values of the Green's function on $\Z^2$ are in $\Q\oplus\frac1\pi\Q$.

In the case that the origin is a triple point of the UST (there are
three disjoint branches from the origin to $\infty$), the edge process
is again determinantal; we compute its kernel explicitly.  (This tripod
computation uses the usual Green's function rather than the branched
double cover Green's function.)

\begin{figure}[t!]
\centerline{\includegraphics[width=5in]{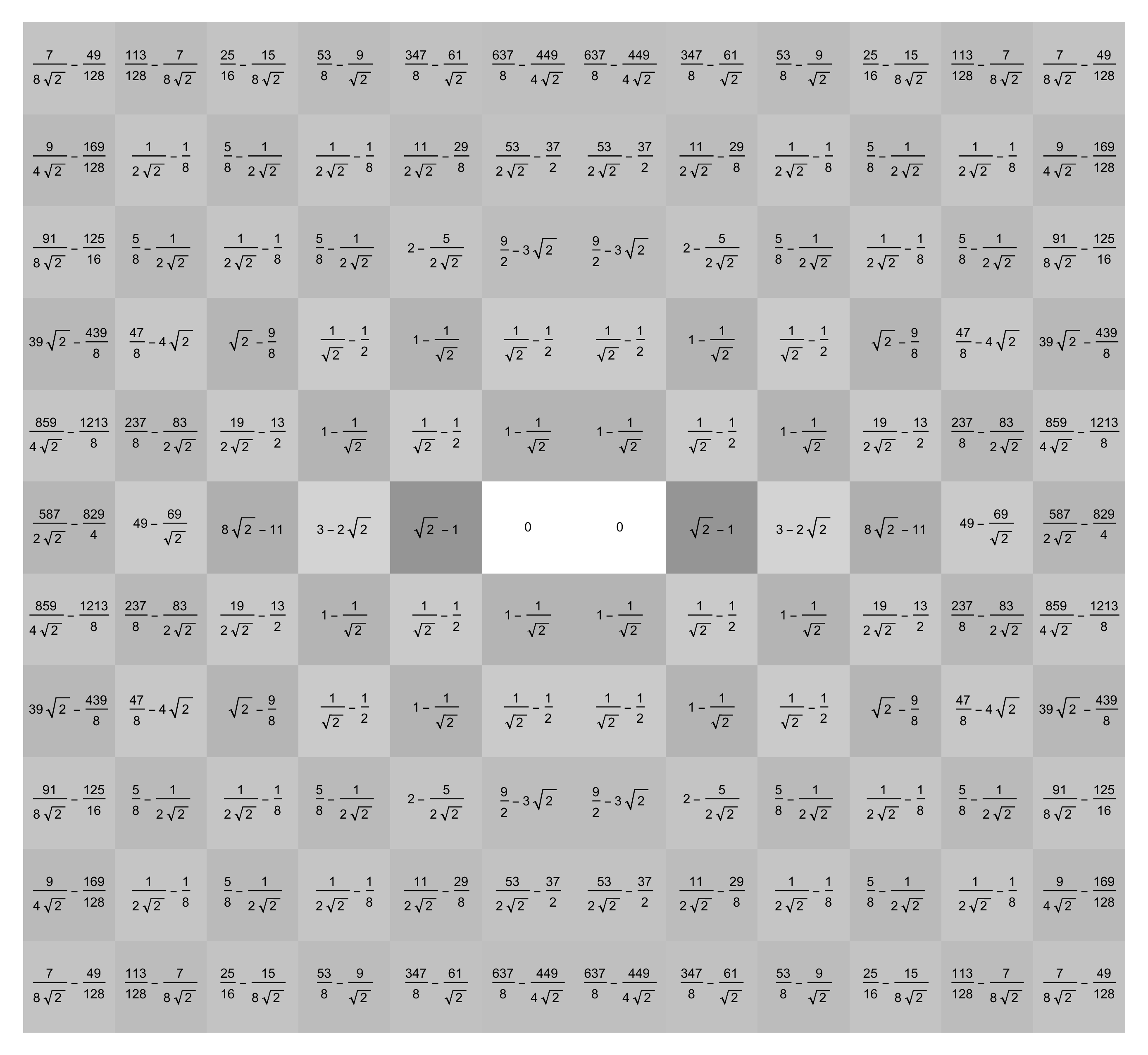}}
\caption{Horizontal dimer probabilities near the monomer, with the exact values written, and shading proportional to these values.}
\label{nearholeprobs}
\end{figure}

In Section~\ref{background} we give background information about the Laplacian, the Green's function, their
electrical interpretations, and their relation with spanning trees. In  Section~\ref{GDsection}
we compute the Green's function on the \emph{slit plane}, that is with Dirichlet boundary conditions on the negative diagonal.
This Green's function is related to the Green's function on the double branched cover of the plane in
Section~\ref{branchedcoversection}. In Section~\ref{holesection} we discuss the Kasteleyn matrix on
regions with a hole (a monomer) at the origin, and how this relates to the UST trunk measure.
In Section~\ref{trunksection} we discuss several properties of the UST trunk measure.
In Section~\ref{triplesection} we prove that, conditional on the origin being a triple point of the whole-plane UST,
the conditioned UST is again determinantal, and we give an explicit kernel.
Our results are extended to the triangular lattice in Section~\ref{triangularsection}.
\medskip

\noindent{\bf Acknowledgements:}
The work of the first author was supported by the NSF grants DMS-1713033 and the Simons foundation grant 327929.
The work of the second author was begun while at Microsoft Research.

\section{Background}\label{background}

\subsection{Laplacian and Green's function}

The Laplacian on a graph $\G$ is the operator $\Delta:\R^\G\to\R^\G$ defined by
\[\Delta f(v) = \sum_{v'\sim v} f(v)-f(v').\]
If $B$ is a subset of vertices of $\G$, called boundary vertices,
we define the \emph{Laplacian with Dirichlet boundary conditions on $B$},
defined on functions which are zero on $B$, by the same formula
but restricted to $v\in\G\setminus B$ (but $v'$ still varies over all neighbors of $v$ in $\G$).
This is an operator from $\R^{\G\setminus B}$ to itself.

For a finite connected graph, if $B$ is nonempty, $\Delta$ is invertible
and we define $G$, the Green's function with Dirichlet boundary conditions at $B$,
to be its inverse.

If $B$ is empty, $\Delta$ has a kernel consisting of the constant functions.
In this case we can define $G$ on the orthocomplement of the constant functions (functions whose sum is zero),
and $G$ is only defined up to an additive constant. Often in this case this constant
is chosen as a function of the first vertex so that $G(v,v)=0$.

\subsection{Electrical interpretation}

When $B$ is nonempty, the Green's function $G(u,v)$ with Dirichlet boundary conditions at $B$ has the following
interpretation in terms of a resistor network.
Consider $\G$ to be a resistor network with a unit resistance on each edge. Boundary
vertices are held at potential $0$.
Then $G(u,v)$ is the potential at $v$ when one unit of current enters the circuit at $u$.
Equivalently, hold the vertex $u$ at potential $f(u)$, where $f(u)$ is the effective resistance between $u$ and $B$.
Then one unit of current will enter the circuit at $u$ and leave through $B$;
the potential values at the vertices~$v$ (including $v=u$) are $G(u,v)$.
The current exiting the circuit at a boundary vertex $b\in B$ has the interpretation as the probability that a random walk started at $u$ will
first reach $B$ at $b$.
For a continuous-time random walk started at $u$, the voltage $G(u,v)$ equals the expected time that the walk spends at $v$ before reaching $B$.

\subsection{Planar graphs and duals}

If a graph $\G$ is planar with Dirichlet boundary on its outer face, then associated to the Green's function $G(u,v)$
is a \emph{dual Green's function\/} $G^*$ on the dual graph. It is the conjugate harmonic function for $G(u,v)$,
that is, if $e=vv'$ is an edge of $\G$ then
\[G^*(f_0,f)-G^*(f_0,f') = G(u,v)-G(u,v')\] where $f,f'$ are the faces left and right of $e$ when traversed from $v$ to $v'$.
Note however that $G^*(f_0,f)$ as a function of $f$ is multivalued around $u$, that is, on a path winding counterclockwise around $u$ it increases by $1$.

In the electrical interpretation, $G^*(f_0,f)-G^*(f_0,f')$ is the signed amount of current flowing between faces $f$ and $f'$.
We sometimes refer to $G^*$ as the \emph{dual voltage}.

\subsection{Green's function on \texorpdfstring{$\Z^2$}{Z\texttwosuperior}}

On $\Z^2$ we define the Green's function as a limit of the Green's function $G_n$ on
the graphs $\G_n = \Z^2\cap[-n,n]\times[-n,n]$ with Dirichlet boundary conditions
as follows.
\be\label{Greenlimitdef}G(v,v') = \lim_{n\to\infty} [G_n(v,v')-G_n(v,v)].\ee
$G_n(v,v')$ itself diverges as $n\to\infty$ since random walk on $\Z^2$ is recurrent.
See \cite{Spitzer}.  Since $G(v,v')<0$ it is sometimes convenient to use its negative which is known as the \textit{potential kernel}.

The Green's function $G((x_1,y_1),(x_2,y_2))$ on $\Z^2$
only depends on $(x_2,y_2)-(x_1,y_1)=(x,y)$ and has the formula
\be\label{Greendoubleint}
G((x,y)):=G((0,0),(x,y))=\frac1{(2\pi)^2}\oint\!\!\!\oint \frac{z^x w^y-1}{4-z-1/z-w-1/w}
\frac{dz}{iz}\frac{dw}{iw}.
\ee

Values of $G$ can be extracted by contour integration.
The diagonal values are $-G((x,x)) = \frac{1}{\pi} \sum_{k=1}^x 1/k$ for $x\geq 0$,
and the remaining values $G((x,y))$ can be deduced by symmetry and harmonicity \cite{mccrea-whipple};
see Figure~\ref{Green}.

\begin{figure}[htbp]
\begin{equation*}
\begin{tikzpicture}[scale=1.2,baseline=0.0cm-2.5pt]
\newcommand{\vnode}[2]{\node at (#1) {$\scriptstyle #2$};}
\foreach \y in {-3,...,3} {\draw[dotted] (-3.4,\y)--(3.4,\y);}
\foreach \x in {-3,...,3} {\draw[dotted] (\x,-3.4)--(\x,3.4);}
\vnode{-3,-3}{\frac{23}{15\pi}}
\vnode{-3,-2}{\frac{1}{4}{+}\frac{2}{3\pi}}
\vnode{-3,-1}{\frac{23}{3\pi}{-}2}
\vnode{-3,0}{\frac{17}{4}{-}\frac{12}{\pi}}
\vnode{-3,1}{\frac{23}{3\pi}{-}2}
\vnode{-3,2}{\frac{1}{4}{+}\frac{2}{3\pi}}
\vnode{-3,3}{\frac{23}{15\pi}}
\vnode{-2,-3}{\frac{1}{4}{+}\frac{2}{3\pi}}
\vnode{-2,-2}{\frac{4}{3\pi}}
\vnode{-2,-1}{\frac{2}{\pi}{-}\frac{1}{4}}
\vnode{-2,0}{1{-}\frac{2}{\pi}}
\vnode{-2,1}{\frac{2}{\pi}{-}\frac{1}{4}}
\vnode{-2,2}{\frac{4}{3\pi}}
\vnode{-2,3}{\frac{1}{4}{+}\frac{2}{3\pi}}
\vnode{-1,-3}{\frac{23}{3\pi}{-}2}
\vnode{-1,-2}{\frac{2}{\pi}{-}\frac{1}{4}}
\vnode{-1,-1}{\frac{1}{\pi}}
\vnode{-1,0}{\frac{1}{4}}
\vnode{-1,1}{\frac{1}{\pi}}
\vnode{-1,2}{\frac{2}{\pi}{-}\frac{1}{4}}
\vnode{-1,3}{\frac{23}{3\pi}{-}2}
\vnode{0,-3}{\frac{17}{4}{-}\frac{12}{\pi}}
\vnode{0,-2}{1{-}\frac{2}{\pi}}
\vnode{0,-1}{\frac{1}{4}}
\vnode{0,0}{0}
\vnode{0,1}{\frac{1}{4}}
\vnode{0,2}{1{-}\frac{2}{\pi}}
\vnode{0,3}{\frac{17}{4}{-}\frac{12}{\pi}}
\vnode{1,-3}{\frac{23}{3\pi}{-}2}
\vnode{1,-2}{\frac{2}{\pi}{-}\frac{1}{4}}
\vnode{1,-1}{\frac{1}{\pi}}
\vnode{1,0}{\frac{1}{4}}
\vnode{1,1}{\frac{1}{\pi}}
\vnode{1,2}{\frac{2}{\pi}{-}\frac{1}{4}}
\vnode{1,3}{\frac{23}{3\pi}{-}2}
\vnode{2,-3}{\frac{1}{4}{+}\frac{2}{3\pi}}
\vnode{2,-2}{\frac{4}{3\pi}}
\vnode{2,-1}{\frac{2}{\pi}{-}\frac{1}{4}}
\vnode{2,0}{1{-}\frac{2}{\pi}}
\vnode{2,1}{\frac{2}{\pi}{-}\frac{1}{4}}
\vnode{2,2}{\frac{4}{3\pi}}
\vnode{2,3}{\frac{1}{4}{+}\frac{2}{3\pi}}
\vnode{3,-3}{\frac{23}{15\pi}}
\vnode{3,-2}{\frac{1}{4}{+}\frac{2}{3\pi}}
\vnode{3,-1}{\frac{23}{3\pi}{-}2}
\vnode{3,0}{\frac{17}{4}{-}\frac{12}{\pi}}
\vnode{3,1}{\frac{23}{3\pi}{-}2}
\vnode{3,2}{\frac{1}{4}{+}\frac{2}{3\pi}}
\vnode{3,3}{\frac{23}{15\pi}}
\end{tikzpicture}
\end{equation*}
\caption{The potential kernel $A((x,y))=-G((x,y))$ of the square lattice.}\label{Green}
\end{figure}
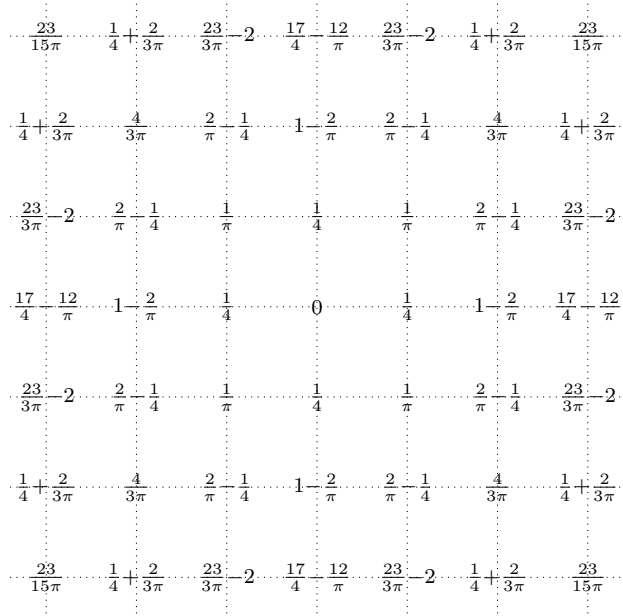

\subsection{Uniform spanning trees}

The Dirichlet Laplacian and Greens functions on a graph $\G$ with boundary $B$ are closely related to uniformly random spanning trees.
$\det \Delta$ counts the number of spanning trees, where the boundary is contracted to a single vertex.  For two edges $vw$ and $xy$,
the \textit{transfer impedance\/} between them is defined to be $T_{vw,xy}=G(v,x)-G(v,y)-G(w,x)+G(w,y)$.  The probability that $k$ edges $e_1,\dots,e_k$
are in a uniform spanning tree is given by the $k\times k$ determinant
\begin{equation} \label{eq:determinantal}
\det \big[T_{e_i,e_j}\big]_{i=1,\dots,k}^{j=1,\dots,k}
\end{equation}
\cite{BP}, so the edge process of the UST is said to be \textit{determinantal\/} with kernel $T$.
There are other natural determinantal processes;
more generally a measure $\mu$ on subsets of a set $S$ is said to be determinantal with kernel $T$
if for any finite collection of items $e_1,\dots,e_k\in S$,
the event that $e_1,\dots,e_k$ are contained in a $\mu$-random subset occurs with probability given by \eqref{eq:determinantal}.

The uniform spanning tree is also closely related to random walks.  The UST path connecting two vertices $u$ and $v$ is distributed according to a \textit{loop-erased random walk\/} (LERW) from $u$ to $v$ \cite{Pemantle}.  Loop-erased random walk from $u$ to $v$ is a process that was first studied by Lawler, and is formed from simple random walk from $u$ to $v$ by erasing loops as they are formed (see \cite{Lawler-Limic:book}).

Pemantle constructed the uniform spanning tree on $\Z^{d}$ by taking the UST on large
boxes and showing that the finite-graph UST measures converge \cite{Pemantle}.
This limiting measure is supported on spanning trees of $\Z^d$ precisely when $d\leq 4$;
it is supported on spanning forests (graphs with no cycles) with infinitely many trees when $d>4$ \cite{Pemantle}.
Almost surely every vertex of the UST has a unique path to $\infty$ when $d>1$ \cite{BLPS}.
The local statistics of the UST on $\Z^d$ can be computed from the Green's function on $\Z^d$.

For the UST on an $n\times n$ box of $\Z^2$, a path within the tree that starts at a vertex $v$ and travels distance order $n$ from $v$ is called an arm at $v$.
It is easy to see that there are vertices with 2 or 3 arms.  Almost always (with probability tending to $1$ as $n\to\infty$)
there are no vertices with 4 arms \cite{Schramm:SLE}.

\section{Green's function on the slit plane}\label{GDsection}

\begin{figure}[t]
\centerline{\begin{tikzpicture}[scale=0.75]
\foreach \y in {-4,...,4} \draw [gray!30!white,thick] (-4.7,\y) -- (4.7,\y);
\foreach \x in {-4,...,4} \draw [gray!30!white,thick] (\x,-4.7) -- (\x,4.7);
\draw [red,thick] (-0.5,-0.5) -- (-0.5,-1);
\foreach \x in {-3,...,-1} {
  \draw [red,thick] (\x+0.5,\x) arc (0:-90:0.5) arc (90:180:0.5);
};
\draw [red,thick] (-3.5,-4) arc (0:-90:0.5) arc (90:145:0.5);
\draw (0,0) node [scale=0.8] {$(0,0)$};
\end{tikzpicture}}
\caption{The square lattice with a zipper (in red) along the dual path $\gamma$ just below the negative diagonal.  \label{zfig}}
\end{figure}
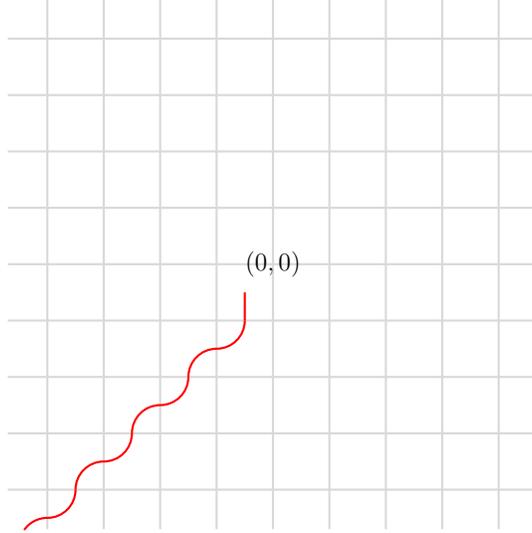
Let $D=\{(k,k)\;:\;k\le-1\}$.
We calculate the Green's function with Dirichlet boundary conditions $G_D$ on $\Z^2\setminus D$, which is zero on $D$.
The function $G_D$ is closely related to another function, the Green's function $G_Z$ on $\Z^2$ with a
\emph{zipper of minus signs\/} starting at the face centered at $(-\tfrac12,-\tfrac12)$, defined as follows.
On $\Z^2$ take a dual path $\gamma$ from $(-\tfrac12,-\tfrac12)$ to $\infty$ contained in the region $x>y$,
for example crossing the edges
with one vertex on $D$ and lying below $D$, see Figure~\ref{zfig}.
Change edge conductances
on the edges crossing $\gamma$ to $-1$. Let $\Delta_Z$ be the Laplacian for $\Z^2$ with these new conductances
and $G_Z$ its inverse (this inverse is essentially the
antisymmetric Green's function on the double cover, see Section~\ref{branchedcoversection} below).
Even though $G_Z((0,0),v)$ depends on the choice of $\gamma$, moving $\gamma$ past a vertex $v$
has the effect of changing the sign of $G_Z((0,0),v)$, but no other values of $G_Z((0,0),\cdot)$ change.
In particular if we move $\gamma$ across all vertices of $D$, a symmetry argument shows that
$G_Z((0,0),v)=0$ when $v\in D$, and thus
we have $G_Z((0,0),v) = G_D((0,0),v)$ for all $v$.
(For a more general formula relating $G_Z$ and $G_D$, see equations~\eqref{eq:GAS-GZ} and \eqref{eq:GD0-GAS}.)

Since $G_D((0,0),v)$ is symmetric about the diagonal, it is convenient to work in the half plane $x\le y$.
We change coordinates, rotating by $-45^\circ$ and scaling so that vertices are at $(x,y)\in\Z^2$ with $x+y$ even and $y\geq0$, and dual vertices are at
$(x,y)\in\Z^2$ with $x+y$ odd and $y\geq 0$. It is also convenient to scale current by $2$, so that
\[G_H(x,y) = 2G_D\left((0,0),\big(\frac{x-y}2,\frac{x+y}2\big)\right) = 2G_Z\left((0,0),\big(\frac{x-y}2,\frac{x+y}2\big)\right).\]
The function $G_H$ is shown in Figure~\ref{half-z2};
in the rest of this section we give an algorithm for computing $G_H$ and a formula for its generating function.
Once this is done, we shall see how to compute the values $G_Z(u,v)$ (Corollarly~\ref{cor:GZ}) and $G_D(u,v)$ (equations~\eqref{eq:GAS-GZ} and \eqref{eq:GD0-GAS})
at arbitrary pairs of vertices in terms of $G_H$.

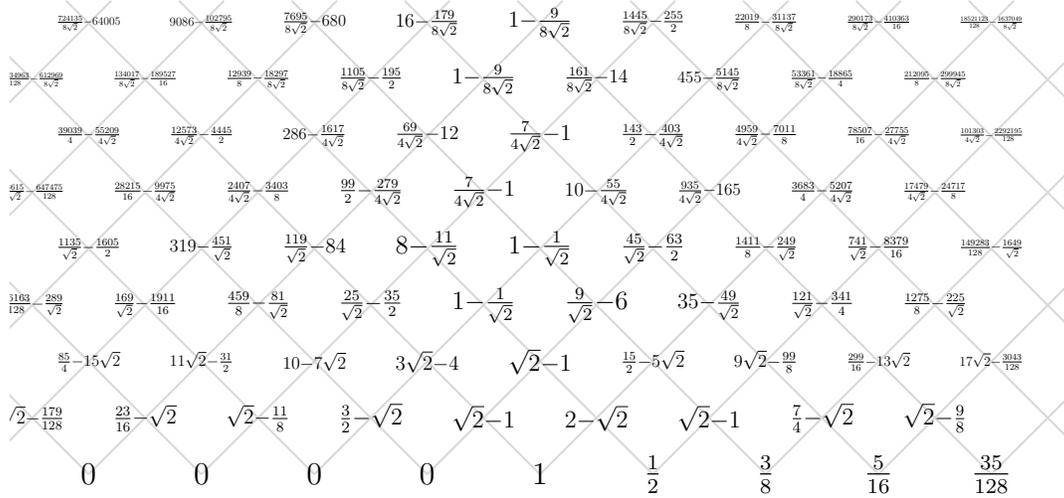
\begin{figure}[h!]
\centerline{\begin{tikzpicture}
\begin{axis}[xmin=-4.7,xmax=4.7,ymin=-.5,ymax=4.2,x={(1.5cm,0)}, y={(0,1.5cm)}, axis lines=none]
\foreach \j in {-9,...,9} {\edef\temp{\noexpand\addplot[domain=0:4.2,gray!30!white,thick]({\j-x},x);\noexpand\addplot[domain=0:4.2,gray!30!white,thick]({\j+x},x);}\temp} ;
\rnode{-1,0}{0}
\rnode{-1.5,0.5}{\frac{3}{2}{-}\sqrt{2}}
\rnode{-0.5,0.5}{\sqrt{2}{-}1}
\rnode{-2,0}{0}
\rnode{-1,1}{3 \sqrt{2}{-}4}
\rnode{0,0}{1}
\rnode{-2,1}{10{-}7 \sqrt{2}}
\rnode{0,1}{\sqrt{2}{-}1}
\rnode{-2.5,0.5}{\sqrt{2}{-}\frac{11}{8}}
\rnode{-1.5,1.5}{\frac{25}{\sqrt{2}}{-}\frac{35}{2}}
\rnode{-0.5,1.5}{1{-}\frac{1}{\sqrt{2}}}
\rnode{0.5,0.5}{2{-}\sqrt{2}}
\rnode{-3,0}{0}
\rnode{-1,2}{8{-}\frac{11}{\sqrt{2}}}
\rnode{1,0}{\frac{1}{2}}
\rnode{-2.5,1.5}{\frac{459}{8}{-}\frac{81}{\sqrt{2}}}
\rnode{0.5,1.5}{\frac{9}{\sqrt{2}}{-}6}
\rnode{-3,1}{11 \sqrt{2}{-}\frac{31}{2}}
\rnode{-2,2}{\frac{119}{\sqrt{2}}{-}84}
\rnode{0,2}{1{-}\frac{1}{\sqrt{2}}}
\rnode{1,1}{\frac{15}{2}{-}5 \sqrt{2}}
\rnode{-3.5,0.5}{\frac{23}{16}{-}\sqrt{2}}
\rnode{-1.5,2.5}{\frac{99}{2}{-}\frac{279}{4 \sqrt{2}}}
\rnode{-0.5,2.5}{\frac{7}{4 \sqrt{2}}{-}1}
\rnode{1.5,0.5}{\sqrt{2}{-}1}
\rnode{-3,2}{319{-}\frac{451}{\sqrt{2}}}
\rnode{1,2}{\frac{45}{\sqrt{2}}{-}\frac{63}{2}}
\rnode{-3.5,1.5}{\frac{169}{\sqrt{2}}{-}\frac{1911}{16}}
\rnode{-2.5,2.5}{\frac{2407}{4 \sqrt{2}}{-}\frac{3403}{8}}
\rnode{0.5,2.5}{10{-}\frac{55}{4 \sqrt{2}}}
\rnode{1.5,1.5}{35{-}\frac{49}{\sqrt{2}}}
\rnode{-4,0}{0}
\rnode{-1,3}{\frac{69}{4 \sqrt{2}}{-}12}
\rnode{2,0}{\frac{3}{8}}
\rnode{-4,1}{\frac{85}{4}{-}15 \sqrt{2}}
\rnode{-2,3}{286{-}\frac{1617}{4 \sqrt{2}}}
\rnode{0,3}{\frac{7}{4 \sqrt{2}}{-}1}
\rnode{2,1}{9 \sqrt{2}{-}\frac{99}{8}}
\rnode{-3.5,2.5}{\frac{28215}{16}{-}\frac{9975}{4 \sqrt{2}}}
\rnode{-4.5,0.5}{\sqrt{2}{-}\frac{179}{128}}
\rnode{-1.5,3.5}{\frac{1105}{8 \sqrt{2}}{-}\frac{195}{2}}
\rnode{-0.5,3.5}{1{-}\frac{9}{8 \sqrt{2}}}
\rnode{1.5,2.5}{\frac{935}{4 \sqrt{2}}{-}165}
\rnode{2.5,0.5}{\frac{7}{4}{-}\sqrt{2}}
\rnode{-4,2}{\frac{1135}{\sqrt{2}}{-}\frac{1605}{2}}
\rnode{-3,3}{\frac{12573}{4 \sqrt{2}}{-}\frac{4445}{2}}
\rnode{1,3}{\frac{143}{2}{-}\frac{403}{4 \sqrt{2}}}
\rnode{2,2}{\frac{1411}{8}{-}\frac{249}{\sqrt{2}}}
\rnode{-4.5,1.5}{\frac{26163}{128}{-}\frac{289}{\sqrt{2}}}
\rnode{-2.5,3.5}{\frac{12939}{8}{-}\frac{18297}{8 \sqrt{2}}}
\rnode{0.5,3.5}{\frac{161}{8 \sqrt{2}}{-}14}
\rnode{2.5,1.5}{\frac{121}{\sqrt{2}}{-}\frac{341}{4}}
\rnode{-1,4}{16{-}\frac{179}{8 \sqrt{2}}}
\rnode{3,0}{\frac{5}{16}}
\rnode{-2,4}{\frac{7695}{8 \sqrt{2}}{-}680}
\rnode{0,4}{1{-}\frac{9}{8 \sqrt{2}}}
\rnode{3,1}{\frac{299}{16}{-}13 \sqrt{2}}
\rnode{-4,3}{\frac{39039}{4}{-}\frac{55209}{4 \sqrt{2}}}
\rnode{2,3}{\frac{4959}{4 \sqrt{2}}{-}\frac{7011}{8}}
\rnode{-4.5,2.5}{\frac{28615}{4 \sqrt{2}}{-}\frac{647475}{128}}
\rnode{-3.5,3.5}{\frac{134017}{8 \sqrt{2}}{-}\frac{189527}{16}}
\rnode{1.5,3.5}{455{-}\frac{5145}{8 \sqrt{2}}}
\rnode{2.5,2.5}{\frac{3683}{4}{-}\frac{5207}{4 \sqrt{2}}}
\rnode{-3,4}{9086{-}\frac{102795}{8 \sqrt{2}}}
\rnode{1,4}{\frac{1445}{8 \sqrt{2}}{-}\frac{255}{2}}
\rnode{3,2}{\frac{741}{\sqrt{2}}{-}\frac{8379}{16}}
\rnode{3.5,0.5}{\sqrt{2}{-}\frac{9}{8}}
\rnode{3.5,1.5}{\frac{1275}{8}{-}\frac{225}{\sqrt{2}}}
\rnode{-4.5,3.5}{\frac{6934963}{128}{-}\frac{612969}{8 \sqrt{2}}}
\rnode{2.5,3.5}{\frac{53361}{8 \sqrt{2}}{-}\frac{18865}{4}}
\rnode{-4,4}{\frac{724135}{8 \sqrt{2}}{-}64005}
\rnode{2,4}{\frac{22019}{8}{-}\frac{31137}{8 \sqrt{2}}}
\rnode{3,3}{\frac{78507}{16}{-}\frac{27755}{4 \sqrt{2}}}
\rnode{4,0}{\frac{35}{128}}
\rnode{4,1}{17 \sqrt{2}{-}\frac{3043}{128}}
\rnode{3.5,2.5}{\frac{17479}{4 \sqrt{2}}{-}\frac{24717}{8}}
\rnode{4,2}{\frac{149283}{128}{-}\frac{1649}{\sqrt{2}}}
\rnode{3,4}{\frac{290173}{8 \sqrt{2}}{-}\frac{410363}{16}}
\rnode{3.5,3.5}{\frac{212095}{8}{-}\frac{299945}{8 \sqrt{2}}}
\rnode{4,3}{\frac{101303}{4 \sqrt{2}}{-}\frac{2292195}{128}}
\rnode{4,4}{\frac{18521123}{128}{-}\frac{1637049}{8 \sqrt{2}}}
\end{axis}
\end{tikzpicture}}
\caption{Values of $G_H(x,y)$: equivalently, voltages in the half plane with mixed Dirichlet--Neumann boundary conditions, and $1$ unit of current inserted at the origin.}
\label{half-z2}
\end{figure}

\subsection{Boundary values}

We compute the boundary voltages and currents of $G_H$.

Let $c_{x}$ denote the current that enters the network at $(x,0)$.  Then $c_0=1$, and for $x<0$ the $-c_x$'s give the exit
probabilities of a simple random walk on the slit plane started at $(0,0)$. (Because we scaled voltages by $2$,
these same $c_x$ are the exit probabilities for the simple random walk on the half-plane with reflecting boundary conditions
on the nonnegative $x$-axis.)
Define the associated generating function
\begin{equation}\label{eq:C-Z2-def}
C(z)=\sum_{k=0}^{-\infty} c_k z^{k}.
\end{equation}
Let
\begin{equation} \label{eq:V-Z2-def}
V(z)=\sum_{k=0}^\infty G_H(2k,0) z^{2k}\,.
\end{equation}
The series of voltage drops for $G_H$ along the $x$-axis is defined on odd integers $\geq -1$, and is $(z-1/z)V(z)$.

In the full plane $\Z^2$, when one unit of current is inserted at $0$ and removed at infinity,
using the diagonal values of the Green's function on $\Z^2$, the
voltage drop from $(2k,0)$ to $(2k+2,0)$ is (for any $k\in\Z$)
\[
\frac{2/\pi}{2k+1}\,.
\]
Let $\delta^+(z)$, $\delta^-(z)$, and $\delta(z)$ be the generating functions for voltage drops on the positive-half, negative-half, and whole axis respectively:
\begin{align*}
\delta^+(z) &= \frac{1}{\pi}\log\frac{1+z}{1-z} \\
\delta^-(z) &= -\frac{1}{\pi}\log\frac{1+z^{-1}}{1-z^{-1}} \\
\delta(z) &= \delta^+(z)+ \delta^-(z) = \frac{1}{\pi} \log(-1) = \pm i
\end{align*}
More precisely, $\delta^+(z)$ is a formal power series in $z$ which converges for $|z|\leq 1$ except at $\pm1$, $\delta^-(z)$ is a formal power series in $z^{-1}$ which converges for $|z|\geq 1$ except at $\pm 1$,
and $\delta(z)$ is a formal Laurent series in $z$ which converges for $|z|=1$ except at $z=\pm 1$.  The value it converges to is
\begin{equation}
\delta(z) = \begin{cases} i & \text{$|z|=1$ and $\Im z>0$} \\ -i & \text{$|z|=1$ and $\Im z<0$} \end{cases}
\end{equation}

For any resistor network, when one unit of current is inserted at a
vertex and removed at another vertex, each edge carries at most one
unit of current.  Suppose we apply currents to the network at
infinitely many vertices.  If the applied currents are absolutely
summable, then the resulting current in any edge of the network
converges unconditionally.  Hence the resulting potential function
makes sense (up to a global constant), and is harmonic off the
vertices to which current is applied.

Because of the probabilistic interpretation of the coefficients of
$C(z)$, its coefficients are absolutely summable, so we may apply
currents according to $C(z)$ to obtain a well-defined set of voltages
which is harmonic off the nonpositive axis.  Recall that $\delta(z)$,
when interpreted as a formal Laurent series, is the generating
function for the voltage drops along the axis when one unit of current
is inserted at the origin.  When currents are inserted according to
$C(z)$, by linearity of voltages, the resulting voltage drops are
$C(z) \delta(z)$.  The voltage drops can also be written $(z-1/z)V(z)$
so we obtain
\begin{equation}\label{eq:Cdelta=DeltaV}
C(z) \delta(z) = (z-1/z) V(z)\,.
\end{equation}

Define
\[
C_*(z) = \sqrt{1-1/z^2}\quad\quad\text{and}\quad\quad D_*(z) = \frac{z-1/z}{\sqrt{1-z^2}},
\]
where $C_*(z)$ has a branch cut from $1$ to $-1$ in the unit disk and is $1$ at $\infty$,
and $D_*(z)$ has a branch cut from $1$ to $-1$ outside the unit disk and the $\sqrt{1-z^2}$ term is $1$ at $z=0$.
Then $D_*(z) = \pm \sqrt{1/z^2-1}$, so $D_*(z)/C_*(z)=\pm i$, where the choice is $+i$ in the region containing
$z=i$ and $-i$ in the region containing $z=-i$,
\[ C_*(z) \delta(z) = D_*(z)\quad\quad\quad\text{when $|z|=1$ and $z\neq\pm 1$}. \]

Observe that $C_*(z)$ is a series with only nonpositive powers of $z$, and $C_*(1)=0$, and except for the
constant term, all remaining coefficients are negative, so the coefficients are absolutely summable.
The equation $C_*(z)\delta(z)=D_*(z)$ implies $D_*(z)$ is the series of resulting voltage drops.
Since $z D_*(z)$ is a series with only nonnegative powers of $z$, the resulting voltages are $0$ on the
negative axis.  If we apply the currents $C-C_*$ (which we can do since the coefficients of $C$ and $C_*$ are
both absolutely summable), currents are only applied on the negative axis
(since the constant terms of $C$ and $C_*$ are the same), and the
resulting voltages are all $0$ on the negative axis.
Since the resulting voltages are harmonic off the negative axis, they are zero everywhere.
From this we deduce that $(z-1/z)V(z)-D_*(z)=0$ and $C(z)-C_*(z)=0$,
i.e.,
\begin{equation} \label{eq:C,V-Z2}
C(z) = \sqrt{1-1/z^2}\quad\quad\text{and}\quad\quad V(z) = \frac{1}{\sqrt{1-z^2}}\,.
\end{equation}

\begin{remark}
  The form of equation~\eqref{eq:Cdelta=DeltaV} suggests that one
  might solve for $C(z)$ and $V(z)$ using Wiener--Hopf factorization,
  in which a function $h(z)$ is factored into $h(z) = f_+(z) f_-(z)$,
  where $f_+(z)$ is analytic in a neighborhood of the unit disk, and
  $f_-(z)$ is analytic in a neighborhood of the complement of the unit
  disk.  When $h(z)$ is analytic in an annulus containing the unit
  circle, this factorization is unique up to a constant factor, which
  would allow one to ``guess and verify'' the factors $f_+$ and $f_-$.
  For our application, $h(z)$ is non-analytic on the unit circle, and
  we were unable to find a version of the unique factorization theorem
  applicable in this setting; indeed, when $f_+$ and $f_-$ are
  allowed to have non-analyticities on the unit circle, the
  factorization is not unique up to constants.  For this reason we
  took a more ``bare hands'' approach to solving for the factors,
  where we made use of the extra information not encoded in
  \eqref{eq:Cdelta=DeltaV}, that is, information we have about the
  coefficients in the series expansions for $C(z)$ and $V(z)$.
\end{remark}

\subsection{The next row}

Given all the values of $G_H(x,0)$ together with $G_H(1,1)$,
harmonicity on the positive $x$-axis allows us to deduce
all the remaining values $G_H(x,1)$.
For $y\geq 0$ let
\[g_y^+(z) = \sum_{x\geq 0} G_H(x,y) z^x\,.\]
Then $g_0^+(z)=V(z)=1/\sqrt{1-z^2}$.  Since $G_H$ is harmonic on the positive axis, we have
\[
2 g_0^+(z) - (z+1/z)g_1^+(z) = 2 G_H(0,0) - G_H(1,1)
\]
Thus
\[ g_1^+(z) = \frac{2/\sqrt{1-z^2} - 2 +G_H(1,1)}{z+1/z}\]
Since $\lim_{x\to\infty} G_H(x,1) = 0$, this implicitly determines $G_H(1,1)$.
If we guess a value $\tilde{G}_H(1,1)$ for $G_H(1,1)$ and solve for the rest of the first row using harmonicity, the resulting values $\tilde{G}_H(x,1)$ would
converge to a sequence alternating between $\pm(\tilde{G}_H(1,1)-G_H(1,1))$, and $\sum_x \tilde{G}_H(x,1) i^x$ would diverge if $\tilde{G}_H(1,1)\neq G_H(1,1)$.
From the above equation we see that $g_1^+(z)$ would not diverge as $z\to i$ only if $G_H(1,1) = 2-\sqrt{2}$, so this must be the correct value of $G_H(1,1)$.
Since $2 G_H(0,0)-G_H(1,1)-G_H(-1,1)=1$, we have
\begin{equation} \label{GH(-1,1)}
G_H(-1,1) = \sqrt{2}-1 \,.
\end{equation}
In principle the same method can be used to find the other values of $G_H$ row by row, but we will instead
derive a generating function for the values along the $y$-axis.  For that calculation it is convenient
to already have the value of $G_H(-1,1)$.

\subsection{Self-duality in the half-plane}\label{sdhp}

Consider the function $G_H(v)$ of Figure~\ref{half-z2}.
Let $G^*_H$ denote the corresponding dual voltages.
Since no current flows into the network along the positive $x$-axis, the dual voltages on
the faces on the positive axis are all equal, so we take them to be $0$ there.
One unit of current flows into the network at $0$, so $G^*_H(-1,0)=-1$.
Since the voltages on the negative real axis are $0$, the dual currents at $(-2k-1,0)$ for $k>1$
are zero.
The harmonic conjugate $G^*_H$ therefore
satisfies (up to a scale factor) the same boundary conditions as $G_H$ but reflected, where
$(x,y)$ gets mapped to $(-x-1,y)$.  By taking the dual twice, we see that the scale factor must be $-1$.
Thus
\[
G^*_H(x,y) = -G_H(-x-1,y)\,.
\]

By combining the definition of the dual voltages with the reflection-self-duality, we have
\begin{align*}
G_H(x,y-1)-G_H(x-1,y) &= G^*_H(x,y)-G^*_H(x-1,y-1) \\ &= -G_H(-x-1,y) + G_H(-x,y-1)
\end{align*}
and
\begin{align*}
G_H(x,y)-G_H(x-1,y-1) &= G^*_H(x-1,y)-G^*_H(x,y-1) \\ &= -G_H(-x,y) + G_H(-x-1,y-1)
\end{align*}
which in the case $x=0$ and $y>0$ gives
\begin{equation} \label{eq:GH-off-vertical}
G_H(0,y) = G_H(-1,y-1).
\end{equation}

\subsection{Moving the zipper}
\label{sec:move-zipper}

To compute $G^Z(u,v)$ for general $u$ and $v$, we can use the above computation for $G^Z((0,0),v)$ and introduce new zipper edges.  A similar approach was used in \cite{KW4}.
Let $\Delta^Z$ denote the Laplacian with the zipper, let $p$ and $q$ be adjacent vertices,
and let $\Delta^{Z'}$ be the Laplacian with the zipper and one extra zipper edge $(p,q)$.
Then for a function $f$,
\[
(\Delta^{Z'} f)(v) = (\Delta^Z f)(v) + 2 f(p) 1_{v=q} + 2 f(q) 1_{v=p}
\]

Define
\[\label{newG}
h(v) = G^Z_{u,v} + a G^Z_{p,v} + b G^Z_{q,v}
\]
Then
\[
(\Delta^{Z'} h)(v) = 1_{v=u} + (a+2 G^Z_{u,q}+ 2 a G^Z_{p,q}+2 b G^Z_{q,q}) 1_{v=p} + (b+2 G^Z_{u,p}+2 a G^Z_{p,p} + 2 b G^Z_{q,p}) 1_{v=q}
\]
We choose the edge $(p,q)$ to be incident to the face containing the endpoint of the zipper,
and we choose $a$ and $b$ so that the last two terms are zero:
\begin{align*}
a+2 G^Z_{u,q}+ 2 a G^Z_{p,q}+2 b G^Z_{q,q} &= 0\\
b+2 G^Z_{u,p}+2 a G^Z_{p,p} + 2 b G^Z_{q,p} &= 0\,.
\end{align*}
For the square lattice with zipper $\gamma$ of Figure~\ref{zfig} let $p=(0,0)$ and $q=(-1,0)$.
Then $G^Z_{p,p}=G^Z_{q,q}=1/2$, and $G^Z_{p,q}=G^Z_{q,p}=1/\sqrt{2}-1/2$, which gives
\begin{align*}
\sqrt{2}a + 2 G^Z_{u,q} + b &= 0\\
\sqrt{2}b + 2 G^Z_{u,p} + a &= 0
\end{align*}
and hence
\begin{align*}
a &= 2 G^Z_{u,p} - 2\sqrt{2} G^Z_{u,q} \\
b &= 2 G^Z_{u,q} - 2\sqrt{2} G^Z_{u,p} \,.
\end{align*}

\subsection{Vertical generating function}

To obtain the ``vertical generating function'' $\sum_{k\ge0} G_H(0,2k)w^k$
we work with the $G_Z^\circlearrowleft$, which is the zipper Green's function $G_Z$ on the
whole lattice $\Z^2$ but with rotated and dilated coordinates.
We let $u=(0,0)$ and $v=(2k,0)$ and adjoin the zipper edges $(p,q)$ with $p=(0,0)$ and $q=(-1,1)$.
The zipper then starts in the
face at $(0,1)$, and we deform the zipper so that it goes up vertically from there.  Then
\[G_{Z'}^\circlearrowleft((0,0),(2k,0)) = G_Z^\circlearrowleft((0,0),(0,2k))\,.\]
But by \eqref{newG}
\[ G_{Z'}^\circlearrowleft((0,0),v) = G_Z^\circlearrowleft((0,0),v) (1+a) + b G_Z^\circlearrowleft((-1,1),v) \]
where
\[ a = 2 G_Z^\circlearrowleft((0,0),(0,0)) - 2\sqrt{2} G_Z^\circlearrowleft((0,0),(-1,1)) = 1 - (2-\sqrt{2}) = \sqrt{2}-1\]
and
\[b = 2 G_Z^\circlearrowleft((0,0),(-1,1)) - 2\sqrt{2} G_Z^\circlearrowleft((0,0),(0,0)) = \sqrt{2}-1 - \sqrt{2} = -1 \]
But by deforming the zipper, we see that
$G_Z^\circlearrowleft((-1,1),(2k,0)) = G_Z^\circlearrowleft((0,0),(-1,2k+1))$.
Thus
\[
G_Z^\circlearrowleft((0,0),(0,2k)) = \sqrt{2} G_Z^\circlearrowleft((0,0),(2k,0)) - G_Z^\circlearrowleft((0,0),(-1,2k+1))
\]
For $k\geq 0$, by the reflection-duality, $G_Z^\circlearrowleft((0,0),(-1,2k+1)) = G_Z^\circlearrowleft((0,0),(0,2k+2))$,
so
\begin{equation} \label{eq:GH-vertical-recurrence}
G_H(0,2k)+G_H(0,2k+2) = \sqrt{2} G_H(2k,0)\quad\quad\quad k\geq 0\,.
\end{equation}
This gives a recurrence for the vertical values, which we can encode in the generating function
\begin{equation}
\sum_{y\geq 2} G_H(0,y) w^y = \frac{\frac{\sqrt{2}}{\sqrt{1-w^2}}-1}{1+w^2} \,.
\end{equation}

\begin{theorem} \label{thm:GH-alg}
We have an algorithm for computing $G_H(x,y)$,
which uses the values for $G_H$ given by equations \eqref{eq:V-Z2-def}, \eqref{eq:C,V-Z2}, \eqref{eq:GH-vertical-recurrence}, and \eqref{eq:GH-off-vertical}
together with the recurrence relation given by the harmonicity of $G_H$.
As a consequence of this algorithm we see that all values of $G_H(x,y)$ are in $\Q[\sqrt{2}]$, in fact they're dyadic rationals with $\sqrt{2}$ adjoined.
\end{theorem}

\begin{corollary} \label{cor:GZ}
  We have an algorithm for computing $G_Z$: $G_Z(0,v)$ is determined
  by $G_H$, and for $u\neq 0$, $G_Z(u,v)$ can be recursively computed
  by moving the endpoint of the zipper as described in
  Section~\ref{sec:move-zipper}.
\end{corollary}

\subsection{Generating function in quadrants}\label{GFquads}

Let $G_N(z,w)$ and $G_W(z,w)$ be the formal power series for $G_H$ in the first and second quadrants:
\begin{align*}
G_N(w,z) &= \sum_{\substack{x\geq0\\y\geq 0}} G_H(x,y) z^x w^y &
G_W(w,z) &= \sum_{\substack{x<0\\y\geq 0}} G_H(x,y) z^x w^y .
\end{align*}
Recall that $V(z) = G_N(z,0) = 1/\sqrt{1-z^2}$ is the generating function of voltages along the $x$-axis.
Let $B(w) = G_N(0,w)-1$ be the generating function of voltages along the positive $y$-axis.
From the reflection duality \eqref{eq:GH-off-vertical}, we had $\sum_y G_H(-1,y) w^y = B(w)/w$.

Since the Green's function in the quadrant is harmonic except along the boundary, we can compute
\begin{multline*}
(4-(w+1/w)(z+1/z)) G_N(z,w) = \\ (2-(z+1/z)/w) V(z) - w/z - (w+1/w)/z B(w) + (1+1/w^2) B(w) + 2 - G_H(-1,1) - G_H(1,1)
\end{multline*}
which we may rewrite as
\begin{equation}
G_N(z,w) = \sum_{\substack{x\geq 0\\y\geq 0}} G_H(x,y) z^x w^y = \frac{\frac{\sqrt{2}(1-w/z)}{\sqrt{1-w^2}}+\frac{2-z/w-1/zw}{\sqrt{1-z^2}}}{4-w/z-z/w-z w-1/zw}\,.
\end{equation}

By multiplying both numerator and denominator by $w z$, we can extract all values in the quadrant via ordinary series expansion.
Similarly,
\begin{equation}
G_W(z,w) = \sum_{\substack{x<0\\y\geq 0}} G_H(x,y) z^x w^y = \frac{\frac{\sqrt{2}(w/z-1)}{\sqrt{1-w^2}}+\sqrt{1-1/z^2}}{4-w/z-z/w-z w-1/zw}
\end{equation}
is a power series in $1/z$ and $w$.

Notice that $G_N(z,w)+G_W(z,w)$ is a formal Laurent series in $z$ and power series in $w$, in which the $\sqrt{2}$ terms cancel.
The $\sqrt{2}$ terms can be recovered from $G_N+G_W$ even though they do not appear explicitly, using the fact that the coefficients
converge to $0$ at infinity, but it is easier to work with $G_N$ and $G_W$ than $G_N+G_W$.

Since the voltages are between $0$ and $1$, the generating function $G_N(z,w)$
is convergent for  $|z|<1$ and $|w|<1$, while $G_W(z,w)$ is convergent for $|z|>1$ and $|w|<1$.
There is a whole curve of $(z,w)$ values along which the denominator $4-w/z-z/w-z w-1/zw$ vanishes,
but the numerator also vanishes along that curve.

\section{Green's function on the double cover of \texorpdfstring{$\Z^2$}{Z\texttwosuperior}}\label{branchedcoversection}

\begin{figure}[htbp]
\centerline{\begin{tikzpicture}[scale=0.9]
\newcommand{\bdclen}{2.19737}
\begin{axis}[xmin=-\bdclen,xmax=\bdclen,ymin=-\bdclen,ymax=\bdclen,x={(3.5cm,0)}, y={(0,3.5cm)}, axis lines=none]
\draw [white,fill=blue!40] (axis cs:-1,0) circle [radius=4];
\draw [white,fill=red!40] (axis cs:+1,0) circle [radius=4];
\addplot [domain=-2:2,gray!30!white,thick] (x,0);
\addplot [domain=-2:2,gray!30!white,thick] (0,x);
\addplot [domain=-2:2,gray!30!white,thick] ({x/sqrt(2)},{x/sqrt(2)});
\addplot [domain=-2:2,gray!30!white,thick] ({x/sqrt(2)},{-x/sqrt(2)});
\foreach \th in {0,...,7} {
\foreach \y in {1,...,4} {\edef\temp{\noexpand\addplot [domain=-4:4,samples=50,gray!30!white,thick]({cos(atan(x/\y)/2+\th*45)*(x^2+\y*\y)^0.25},{sin(atan(x/\y)/2+\th*45)*(x^2+\y*\y)^0.25});}\temp} ;
};
\rnode{0,0}{0}
\rnode{0,-1.00000}{0}
\rnode{0,1.00000}{0}
\rnode{0.707107,-0.707107}{\frac{3}{2}{-}\sqrt{2}}
\rnode{-0.707107,0.707107}{\sqrt{2}{-}\frac{3}{2}}
\rnode{0.707107,0.707107}{\frac{3}{2}{-}\sqrt{2}}
\rnode{-0.707107,-0.707107}{\sqrt{2}{-}\frac{3}{2}}
\rnode{-1.00000,0}{1{-}\sqrt{2}}
\rnode{1.00000,0}{\sqrt{2}{-}1}
\rnode{-0.455090,1.09868}{\frac{1}{\sqrt{2}}{-}\frac{3}{4}}
\rnode{0.455090,-1.09868}{\frac{3}{4}{-}\frac{1}{\sqrt{2}}}
\rnode{-0.455090,-1.09868}{\frac{1}{\sqrt{2}}{-}\frac{3}{4}}
\rnode{0.455090,1.09868}{\frac{3}{4}{-}\frac{1}{\sqrt{2}}}
\rnode{-1.09868,0.455090}{\frac{1}{2}{-}\frac{1}{\sqrt{2}}}
\rnode{1.09868,-0.455090}{\frac{1}{\sqrt{2}}{-}\frac{1}{2}}
\rnode{-1.09868,-0.455090}{\frac{1}{2}{-}\frac{1}{\sqrt{2}}}
\rnode{1.09868,0.455090}{\frac{1}{\sqrt{2}}{-}\frac{1}{2}}
\rnode{0,-1.41421}{0}
\rnode{0,1.41421}{0}
\rnode{-1.00000,1.00000}{4 \sqrt{2}{-}\frac{23}{4}}
\rnode{1.00000,-1.00000}{\frac{23}{4}{-}4 \sqrt{2}}
\rnode{-1.00000,-1.00000}{4 \sqrt{2}{-}\frac{23}{4}}
\rnode{1.00000,1.00000}{\frac{23}{4}{-}4 \sqrt{2}}
\rnode{1.41421,0}{3 \sqrt{2}{-}4}
\rnode{-1.41421,0}{4{-}3 \sqrt{2}}
\rnode{0.343561,-1.45535}{\frac{5}{2}{-}\frac{7}{2 \sqrt{2}}}
\rnode{-0.343561,1.45535}{\frac{7}{2 \sqrt{2}}{-}\frac{5}{2}}
\rnode{0.343561,1.45535}{\frac{5}{2}{-}\frac{7}{2 \sqrt{2}}}
\rnode{-0.343561,-1.45535}{\frac{7}{2 \sqrt{2}}{-}\frac{5}{2}}
\rnode{-0.786151,1.27202}{1{-}\frac{3}{2 \sqrt{2}}}
\rnode{0.786151,-1.27202}{\frac{3}{2 \sqrt{2}}{-}1}
\rnode{-0.786151,-1.27202}{1{-}\frac{3}{2 \sqrt{2}}}
\rnode{0.786151,1.27202}{\frac{3}{2 \sqrt{2}}{-}1}
\rnode{1.27202,-0.786151}{\frac{11}{2 \sqrt{2}}{-}\frac{15}{4}}
\rnode{-1.27202,0.786151}{\frac{15}{4}{-}\frac{11}{2 \sqrt{2}}}
\rnode{1.27202,0.786151}{\frac{11}{2 \sqrt{2}}{-}\frac{15}{4}}
\rnode{-1.27202,-0.786151}{\frac{15}{4}{-}\frac{11}{2 \sqrt{2}}}
\rnode{-1.45535,0.343561}{\frac{3}{2 \sqrt{2}}{-}\frac{5}{4}}
\rnode{1.45535,-0.343561}{\frac{5}{4}{-}\frac{3}{2 \sqrt{2}}}
\rnode{-1.45535,-0.343561}{\frac{3}{2 \sqrt{2}}{-}\frac{5}{4}}
\rnode{1.45535,0.343561}{\frac{5}{4}{-}\frac{3}{2 \sqrt{2}}}
\rnode{0.643594,-1.55377}{\frac{1}{2 \sqrt{2}}{-}\frac{5}{16}}
\rnode{-0.643594,1.55377}{\frac{5}{16}{-}\frac{1}{2 \sqrt{2}}}
\rnode{0.643594,1.55377}{\frac{1}{2 \sqrt{2}}{-}\frac{5}{16}}
\rnode{-0.643594,-1.55377}{\frac{5}{16}{-}\frac{1}{2 \sqrt{2}}}
\rnode{1.55377,-0.643594}{\frac{1}{2}{-}\frac{1}{2 \sqrt{2}}}
\rnode{-1.55377,0.643594}{\frac{1}{2 \sqrt{2}}{-}\frac{1}{2}}
\rnode{1.55377,0.643594}{\frac{1}{2}{-}\frac{1}{2 \sqrt{2}}}
\rnode{-1.55377,-0.643594}{\frac{1}{2 \sqrt{2}}{-}\frac{1}{2}}
\rnode{0,-1.73205}{0}
\rnode{0,1.73205}{0}
\rnode{1.22474,-1.22474}{\frac{105}{4}{-}\frac{37}{\sqrt{2}}}
\rnode{-1.22474,1.22474}{\frac{37}{\sqrt{2}}{-}\frac{105}{4}}
\rnode{1.22474,1.22474}{\frac{105}{4}{-}\frac{37}{\sqrt{2}}}
\rnode{-1.22474,-1.22474}{\frac{37}{\sqrt{2}}{-}\frac{105}{4}}
\rnode{1.73205,0}{\frac{25}{\sqrt{2}}{-}\frac{35}{2}}
\rnode{-1.73205,0}{\frac{35}{2}{-}\frac{25}{\sqrt{2}}}
\rnode{0.284849,-1.75532}{\frac{153}{16}{-}\frac{27}{2 \sqrt{2}}}
\rnode{-0.284849,1.75532}{\frac{27}{2 \sqrt{2}}{-}\frac{153}{16}}
\rnode{0.284849,1.75532}{\frac{153}{16}{-}\frac{27}{2 \sqrt{2}}}
\rnode{-0.284849,-1.75532}{\frac{27}{2 \sqrt{2}}{-}\frac{153}{16}}
\rnode{-1.03978,1.44262}{\frac{163}{16}{-}\frac{29}{2 \sqrt{2}}}
\rnode{1.03978,-1.44262}{\frac{29}{2 \sqrt{2}}{-}\frac{163}{16}}
\rnode{-1.03978,-1.44262}{\frac{163}{16}{-}\frac{29}{2 \sqrt{2}}}
\rnode{1.03978,1.44262}{\frac{29}{2 \sqrt{2}}{-}\frac{163}{16}}
\rnode{1.44262,-1.03978}{\frac{59}{2 \sqrt{2}}{-}\frac{83}{4}}
\rnode{-1.44262,1.03978}{\frac{83}{4}{-}\frac{59}{2 \sqrt{2}}}
\rnode{1.44262,1.03978}{\frac{59}{2 \sqrt{2}}{-}\frac{83}{4}}
\rnode{-1.44262,-1.03978}{\frac{83}{4}{-}\frac{59}{2 \sqrt{2}}}
\rnode{1.75532,-0.284849}{9{-}\frac{25}{2 \sqrt{2}}}
\rnode{-1.75532,0.284849}{\frac{25}{2 \sqrt{2}}{-}9}
\rnode{1.75532,0.284849}{9{-}\frac{25}{2 \sqrt{2}}}
\rnode{-1.75532,-0.284849}{\frac{25}{2 \sqrt{2}}{-}9}
\rnode{-0.550251,1.81735}{\frac{33}{8}{-}\frac{47}{8 \sqrt{2}}}
\rnode{0.550251,-1.81735}{\frac{47}{8 \sqrt{2}}{-}\frac{33}{8}}
\rnode{-0.550251,-1.81735}{\frac{33}{8}{-}\frac{47}{8 \sqrt{2}}}
\rnode{0.550251,1.81735}{\frac{47}{8 \sqrt{2}}{-}\frac{33}{8}}
\rnode{0.895977,-1.67415}{\frac{11}{8}{-}\frac{15}{8 \sqrt{2}}}
\rnode{-0.895977,1.67415}{\frac{15}{8 \sqrt{2}}{-}\frac{11}{8}}
\rnode{0.895977,1.67415}{\frac{11}{8}{-}\frac{15}{8 \sqrt{2}}}
\rnode{-0.895977,-1.67415}{\frac{15}{8 \sqrt{2}}{-}\frac{11}{8}}
\rnode{-1.67415,0.895977}{\frac{63}{8 \sqrt{2}}{-}\frac{91}{16}}
\rnode{1.67415,-0.895977}{\frac{91}{16}{-}\frac{63}{8 \sqrt{2}}}
\rnode{-1.67415,-0.895977}{\frac{63}{8 \sqrt{2}}{-}\frac{91}{16}}
\rnode{1.67415,0.895977}{\frac{91}{16}{-}\frac{63}{8 \sqrt{2}}}
\rnode{1.81735,-0.550251}{\frac{15}{8 \sqrt{2}}{-}\frac{19}{16}}
\rnode{-1.81735,0.550251}{\frac{19}{16}{-}\frac{15}{8 \sqrt{2}}}
\rnode{1.81735,0.550251}{\frac{15}{8 \sqrt{2}}{-}\frac{19}{16}}
\rnode{-1.81735,-0.550251}{\frac{19}{16}{-}\frac{15}{8 \sqrt{2}}}
\rnode{0.788239,-1.90298}{\frac{21}{32}{-}\frac{7}{8 \sqrt{2}}}
\rnode{-0.788239,1.90298}{\frac{7}{8 \sqrt{2}}{-}\frac{21}{32}}
\rnode{0.788239,1.90298}{\frac{21}{32}{-}\frac{7}{8 \sqrt{2}}}
\rnode{-0.788239,-1.90298}{\frac{7}{8 \sqrt{2}}{-}\frac{21}{32}}
\rnode{-1.90298,0.788239}{\frac{1}{2}{-}\frac{7}{8 \sqrt{2}}}
\rnode{1.90298,-0.788239}{\frac{7}{8 \sqrt{2}}{-}\frac{1}{2}}
\rnode{-1.90298,-0.788239}{\frac{1}{2}{-}\frac{7}{8 \sqrt{2}}}
\rnode{1.90298,0.788239}{\frac{7}{8 \sqrt{2}}{-}\frac{1}{2}}
\end{axis}
\end{tikzpicture}}
\caption{The antisymmetric Green's function $G_{\Sigma}^A((1,0)^\beta,\cdot)$.  The source is at $(1,0)^\beta$ shown in red, and the sink is at $(1,0)^{\beta\sigma}$, shown in blue.}
\label{GSA10}
\end{figure}
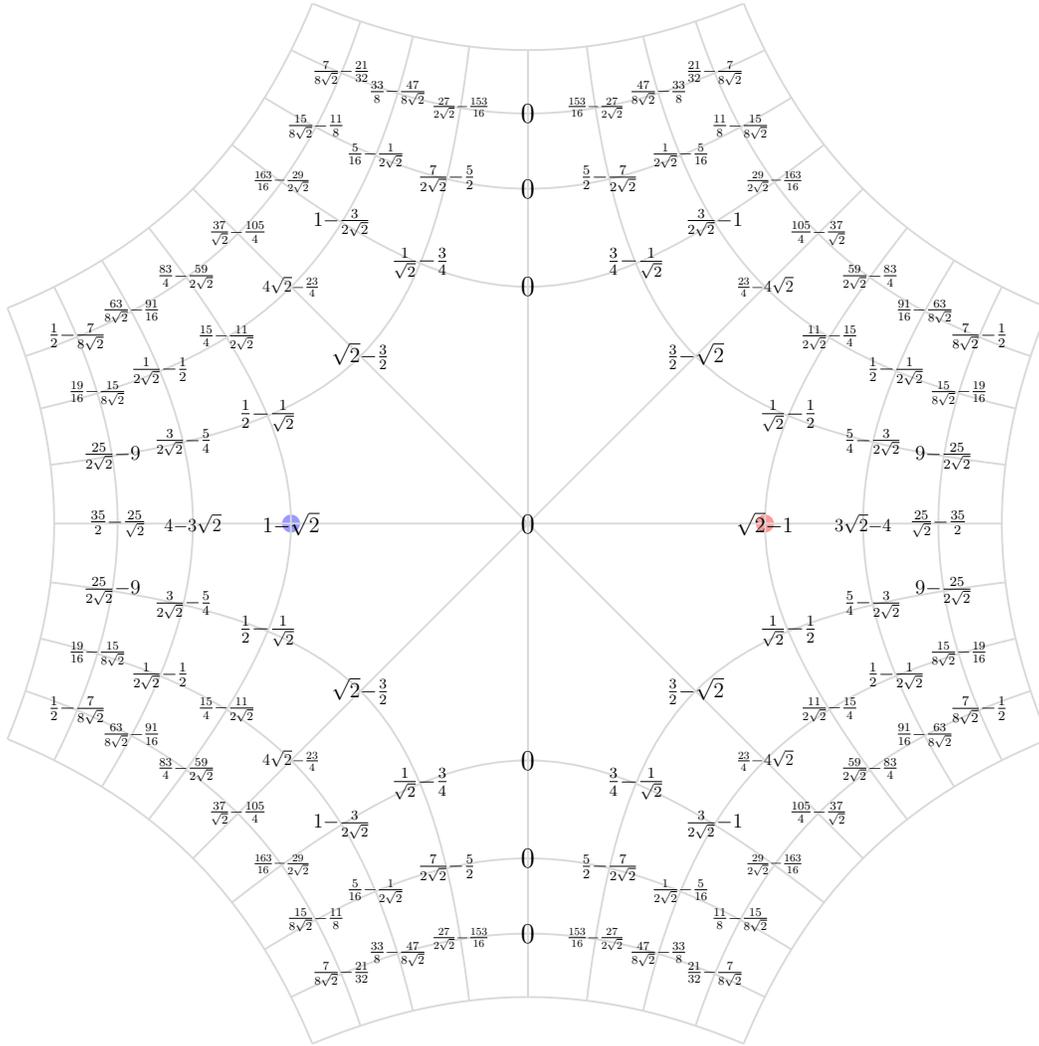

\begin{figure}[htbp]
\centerline{\begin{tikzpicture}[scale=0.9]
\newcommand{\bdclen}{2.19737}
\begin{axis}[xmin=-\bdclen,xmax=\bdclen,ymin=-\bdclen,ymax=\bdclen,x={(3.5cm,0)}, y={(0,3.5cm)}, axis lines=none]
\draw [white,fill=blue!40] (axis cs:-0.776887,-0.321797) circle [radius=4];
\draw [white,fill=red!40] (axis cs:0.776887,0.321797) circle [radius=4];
\foreach \th in {0,...,7} {
\foreach \y in {0.5,...,3.5} {\edef\temp{\noexpand\addplot [domain=-4:4,samples=50,gray!30!white,thick]({cos(atan(x/\y)/2+\th*45)*(x^2+\y*\y)^0.25},{sin(atan(x/\y)/2+\th*45)*(x^2+\y*\y)^0.25});}\temp} ;
};
\rnode{-0.321797,0.776887}{0}
\rnode{0.321797,-0.776887}{0}
\rnode{-0.321797,-0.776887}{\frac{1}{2}{-}\frac{1}{\sqrt{2}}}
\rnode{0.321797,0.776887}{\frac{1}{\sqrt{2}}{-}\frac{1}{2}}
\rnode{-0.776887,0.321797}{\frac{1}{2}{-}\frac{1}{\sqrt{2}}}
\rnode{0.776887,-0.321797}{\frac{1}{\sqrt{2}}{-}\frac{1}{2}}
\rnode{-0.776887,-0.321797}{{-}\frac{1}{2}}
\rnode{0.776887,0.321797}{\frac{1}{2}}
\rnode{0.201419,-1.24120}{\frac{1}{\sqrt{2}}{-}\frac{3}{4}}
\rnode{-0.201419,1.24120}{\frac{3}{4}{-}\frac{1}{\sqrt{2}}}
\rnode{-0.201419,-1.24120}{2{-}\frac{3}{\sqrt{2}}}
\rnode{0.201419,1.24120}{\frac{3}{\sqrt{2}}{-}2}
\rnode{-0.735234,1.02008}{\frac{1}{\sqrt{2}}{-}\frac{3}{4}}
\rnode{0.735234,-1.02008}{\frac{3}{4}{-}\frac{1}{\sqrt{2}}}
\rnode{-0.735234,-1.02008}{\frac{1}{2}{-}\frac{1}{\sqrt{2}}}
\rnode{0.735234,1.02008}{\frac{1}{\sqrt{2}}{-}\frac{1}{2}}
\rnode{-1.02008,0.735234}{2{-}\frac{3}{\sqrt{2}}}
\rnode{1.02008,-0.735234}{\frac{3}{\sqrt{2}}{-}2}
\rnode{1.02008,0.735234}{1{-}\frac{1}{\sqrt{2}}}
\rnode{-1.02008,-0.735234}{\frac{1}{\sqrt{2}}{-}1}
\rnode{-1.24120,0.201419}{\frac{1}{2}{-}\frac{1}{\sqrt{2}}}
\rnode{1.24120,-0.201419}{\frac{1}{\sqrt{2}}{-}\frac{1}{2}}
\rnode{1.24120,0.201419}{1{-}\frac{1}{\sqrt{2}}}
\rnode{-1.24120,-0.201419}{\frac{1}{\sqrt{2}}{-}1}
\rnode{-0.557369,1.34561}{0}
\rnode{0.557369,-1.34561}{0}
\rnode{0.557369,1.34561}{\frac{1}{2}{-}\frac{1}{2 \sqrt{2}}}
\rnode{-0.557369,-1.34561}{\frac{1}{2 \sqrt{2}}{-}\frac{1}{2}}
\rnode{1.34561,-0.557369}{\frac{1}{2}{-}\frac{1}{2 \sqrt{2}}}
\rnode{-1.34561,0.557369}{\frac{1}{2 \sqrt{2}}{-}\frac{1}{2}}
\rnode{-1.34561,-0.557369}{{-}\frac{1}{4}}
\rnode{1.34561,0.557369}{\frac{1}{4}}
\rnode{-0.157337,1.58895}{5{-}\frac{7}{\sqrt{2}}}
\rnode{0.157337,-1.58895}{\frac{7}{\sqrt{2}}{-}5}
\rnode{0.157337,1.58895}{\frac{25}{2 \sqrt{2}}{-}\frac{35}{4}}
\rnode{-0.157337,-1.58895}{\frac{35}{4}{-}\frac{25}{2 \sqrt{2}}}
\rnode{1.01230,-1.23481}{5{-}\frac{7}{\sqrt{2}}}
\rnode{-1.01230,1.23481}{\frac{7}{\sqrt{2}}{-}5}
\rnode{-1.01230,-1.23481}{3{-}\frac{9}{2 \sqrt{2}}}
\rnode{1.01230,1.23481}{\frac{9}{2 \sqrt{2}}{-}3}
\rnode{1.23481,-1.01230}{\frac{25}{2 \sqrt{2}}{-}\frac{35}{4}}
\rnode{-1.23481,1.01230}{\frac{35}{4}{-}\frac{25}{2 \sqrt{2}}}
\rnode{1.23481,1.01230}{\frac{15}{4}{-}\frac{5}{\sqrt{2}}}
\rnode{-1.23481,-1.01230}{\frac{5}{\sqrt{2}}{-}\frac{15}{4}}
\rnode{-1.58895,0.157337}{3{-}\frac{9}{2 \sqrt{2}}}
\rnode{1.58895,-0.157337}{\frac{9}{2 \sqrt{2}}{-}3}
\rnode{1.58895,0.157337}{\frac{15}{4}{-}\frac{5}{\sqrt{2}}}
\rnode{-1.58895,-0.157337}{\frac{5}{\sqrt{2}}{-}\frac{15}{4}}
\rnode{0.455783,-1.64552}{\frac{11}{16}{-}\frac{1}{\sqrt{2}}}
\rnode{-0.455783,1.64552}{\frac{1}{\sqrt{2}}{-}\frac{11}{16}}
\rnode{0.455783,1.64552}{4{-}\frac{11}{2 \sqrt{2}}}
\rnode{-0.455783,-1.64552}{\frac{11}{2 \sqrt{2}}{-}4}
\rnode{-0.841272,1.48585}{\frac{11}{16}{-}\frac{1}{\sqrt{2}}}
\rnode{0.841272,-1.48585}{\frac{1}{\sqrt{2}}{-}\frac{11}{16}}
\rnode{0.841272,1.48585}{\frac{1}{2}{-}\frac{1}{2 \sqrt{2}}}
\rnode{-0.841272,-1.48585}{\frac{1}{2 \sqrt{2}}{-}\frac{1}{2}}
\rnode{1.48585,-0.841272}{4{-}\frac{11}{2 \sqrt{2}}}
\rnode{-1.48585,0.841272}{\frac{11}{2 \sqrt{2}}{-}4}
\rnode{-1.48585,-0.841272}{\frac{1}{2}{-}\frac{1}{\sqrt{2}}}
\rnode{1.48585,0.841272}{\frac{1}{\sqrt{2}}{-}\frac{1}{2}}
\rnode{1.64552,-0.455783}{\frac{1}{2}{-}\frac{1}{2 \sqrt{2}}}
\rnode{-1.64552,0.455783}{\frac{1}{2 \sqrt{2}}{-}\frac{1}{2}}
\rnode{-1.64552,-0.455783}{\frac{1}{2}{-}\frac{1}{\sqrt{2}}}
\rnode{1.64552,0.455783}{\frac{1}{\sqrt{2}}{-}\frac{1}{2}}
\rnode{-0.133293,1.87557}{\frac{459}{16}{-}\frac{81}{2 \sqrt{2}}}
\rnode{0.133293,-1.87557}{\frac{81}{2 \sqrt{2}}{-}\frac{459}{16}}
\rnode{-0.133293,-1.87557}{42{-}\frac{119}{2 \sqrt{2}}}
\rnode{0.133293,1.87557}{\frac{119}{2 \sqrt{2}}{-}42}
\rnode{-0.719560,1.73717}{0}
\rnode{0.719560,-1.73717}{0}
\rnode{-0.719560,-1.73717}{\frac{1}{2}{-}\frac{7}{8 \sqrt{2}}}
\rnode{0.719560,1.73717}{\frac{7}{8 \sqrt{2}}{-}\frac{1}{2}}
\rnode{1.23198,-1.42048}{\frac{459}{16}{-}\frac{81}{2 \sqrt{2}}}
\rnode{-1.23198,1.42048}{\frac{81}{2 \sqrt{2}}{-}\frac{459}{16}}
\rnode{1.23198,1.42048}{\frac{45}{2 \sqrt{2}}{-}\frac{63}{4}}
\rnode{-1.23198,-1.42048}{\frac{63}{4}{-}\frac{45}{2 \sqrt{2}}}
\rnode{-1.42048,1.23198}{42{-}\frac{119}{2 \sqrt{2}}}
\rnode{1.42048,-1.23198}{\frac{119}{2 \sqrt{2}}{-}42}
\rnode{1.42048,1.23198}{\frac{35}{2}{-}\frac{49}{2 \sqrt{2}}}
\rnode{-1.42048,-1.23198}{\frac{49}{2 \sqrt{2}}{-}\frac{35}{2}}
\rnode{-1.73717,0.719560}{\frac{1}{2}{-}\frac{7}{8 \sqrt{2}}}
\rnode{1.73717,-0.719560}{\frac{7}{8 \sqrt{2}}{-}\frac{1}{2}}
\rnode{-1.73717,-0.719560}{{-}\frac{3}{16}}
\rnode{1.73717,0.719560}{\frac{3}{16}}
\rnode{1.87557,-0.133293}{\frac{45}{2 \sqrt{2}}{-}\frac{63}{4}}
\rnode{-1.87557,0.133293}{\frac{63}{4}{-}\frac{45}{2 \sqrt{2}}}
\rnode{1.87557,0.133293}{\frac{35}{2}{-}\frac{49}{2 \sqrt{2}}}
\rnode{-1.87557,-0.133293}{\frac{49}{2 \sqrt{2}}{-}\frac{35}{2}}
\rnode{-0.392356,1.91153}{\frac{11}{\sqrt{2}}{-}\frac{31}{4}}
\rnode{0.392356,-1.91153}{\frac{31}{4}{-}\frac{11}{\sqrt{2}}}
\rnode{-0.392356,-1.91153}{\frac{279}{8 \sqrt{2}}{-}\frac{99}{4}}
\rnode{0.392356,1.91153}{\frac{99}{4}{-}\frac{279}{8 \sqrt{2}}}
\rnode{1.07422,-1.62909}{\frac{11}{\sqrt{2}}{-}\frac{31}{4}}
\rnode{-1.07422,1.62909}{\frac{31}{4}{-}\frac{11}{\sqrt{2}}}
\rnode{1.07422,1.62909}{5{-}\frac{55}{8 \sqrt{2}}}
\rnode{-1.07422,-1.62909}{\frac{55}{8 \sqrt{2}}{-}5}
\rnode{-1.62909,1.07422}{\frac{279}{8 \sqrt{2}}{-}\frac{99}{4}}
\rnode{1.62909,-1.07422}{\frac{99}{4}{-}\frac{279}{8 \sqrt{2}}}
\rnode{-1.62909,-1.07422}{\frac{99}{16}{-}\frac{9}{\sqrt{2}}}
\rnode{1.62909,1.07422}{\frac{9}{\sqrt{2}}{-}\frac{99}{16}}
\rnode{1.91153,-0.392356}{5{-}\frac{55}{8 \sqrt{2}}}
\rnode{-1.91153,0.392356}{\frac{55}{8 \sqrt{2}}{-}5}
\rnode{-1.91153,-0.392356}{\frac{99}{16}{-}\frac{9}{\sqrt{2}}}
\rnode{1.91153,0.392356}{\frac{9}{\sqrt{2}}{-}\frac{99}{16}}
\rnode{0.632915,-1.97499}{\frac{1}{\sqrt{2}}{-}\frac{23}{32}}
\rnode{-0.632915,1.97499}{\frac{23}{32}{-}\frac{1}{\sqrt{2}}}
\rnode{-0.632915,-1.97499}{6{-}\frac{69}{8 \sqrt{2}}}
\rnode{0.632915,1.97499}{\frac{69}{8 \sqrt{2}}{-}6}
\rnode{-0.948990,1.84407}{\frac{1}{\sqrt{2}}{-}\frac{23}{32}}
\rnode{0.948990,-1.84407}{\frac{23}{32}{-}\frac{1}{\sqrt{2}}}
\rnode{-0.948990,-1.84407}{\frac{1}{2}{-}\frac{7}{8 \sqrt{2}}}
\rnode{0.948990,1.84407}{\frac{7}{8 \sqrt{2}}{-}\frac{1}{2}}
\rnode{-1.84407,0.948990}{6{-}\frac{69}{8 \sqrt{2}}}
\rnode{1.84407,-0.948990}{\frac{69}{8 \sqrt{2}}{-}6}
\rnode{-1.84407,-0.948990}{\frac{1}{\sqrt{2}}{-}\frac{7}{8}}
\rnode{1.84407,0.948990}{\frac{7}{8}{-}\frac{1}{\sqrt{2}}}
\rnode{-1.97499,0.632915}{\frac{1}{2}{-}\frac{7}{8 \sqrt{2}}}
\rnode{1.97499,-0.632915}{\frac{7}{8 \sqrt{2}}{-}\frac{1}{2}}
\rnode{-1.97499,-0.632915}{\frac{1}{\sqrt{2}}{-}\frac{7}{8}}
\rnode{1.97499,0.632915}{\frac{7}{8}{-}\frac{1}{\sqrt{2}}}
\rnode{-0.851395,2.05545}{0}
\rnode{0.851395,-2.05545}{0}
\rnode{0.851395,2.05545}{\frac{1}{2}{-}\frac{9}{16 \sqrt{2}}}
\rnode{-0.851395,-2.05545}{\frac{9}{16 \sqrt{2}}{-}\frac{1}{2}}
\rnode{2.05545,-0.851395}{\frac{1}{2}{-}\frac{9}{16 \sqrt{2}}}
\rnode{-2.05545,0.851395}{\frac{9}{16 \sqrt{2}}{-}\frac{1}{2}}
\rnode{-2.05545,-0.851395}{{-}\frac{5}{32}}
\rnode{2.05545,0.851395}{\frac{5}{32}}
\end{axis}
\end{tikzpicture}}
\caption{The antisymmetric Green's function $G_{\Xi}^A((1,1)^\beta,\cdot)$.  The source is at $(1,1)^\beta$ shown in red, and the sink is at $(1,1)^{\beta\sigma}$, shown in blue.}
\label{G11face}
\end{figure}

In this section it is convenient to have $D$ start at the origin; let $D_0 = \{(k,k)\;:\;k\le 0\}$ and
$G_{D_0}(v,w) = G_D(v-(1,1),w-(1,1))$ the shifted version of $G_D$.

Let $\Sigma$ be the double cover of the graph $\Z^2$,
branched over the origin. This is a graph with two vertices, edges and faces
over every vertex, edge and face of $\Z^2$ except for a single vertex over the
origin (of degree $8$).  See Figure~\ref{GSA10}.

Let $\Xi$ be the dual of $\Sigma$.  This is the double cover of the graph $(\Z+\frac12)^2$
branched over the face centered at the origin.  See Figure~\ref{G11face}.

Let $\pi:\Sigma\to\Z^2$ and $\pi:\Xi\to(\Z+\frac12)^2$ be the projection.
Each vertex $z\in\Z^2$ (except $z=0$) has two pre-images in $\Sigma$,
and each vertex $z\in(\Z+\frac12)^2$ has two pre-images in $\Xi$.
For $z\in\Z^2\cup(\Z+\frac12)^2$, we let $z^\beta$ denote the principal branch in $\Sigma$ or $\Xi$,
with a branch cut along but just below the negative diagonal, i.e., for adjacent vertices $z_1,z_2\in\Z^2$
or $z_1,z_2\in(\Z+\frac12)^2$,
the vertices $z_1^\beta,z_2^\beta\in\Sigma\cup\Xi$ are adjacent except for pairs of the form
$\{(-k,-k),(-k+1,-k)\}$ and $\{(-k,-k),(-k,-k-1)\}$ with $2k\geq 1$.

Let $\sigma:z\mapsto z^\sigma$ be the map exchanging sheets of the cover $\Sigma$ or $\Xi$,
and $\tau:z\mapsto z^\tau$ the reflection of $\Sigma$ or $\Xi$ in the pre-image of $D_0$.
These operations commute.

The Green's functions $G_\Sigma, G_\Xi$ on $\Sigma$ and $\Xi$ respectively
are defined as in \eqref{Greenlimitdef}, as a limit of the Green's functions for the corresponding double branched covers of $\G_n$.

Since
\begin{align*}
G_\Sigma(v,w)+G_\Sigma(v^{\sigma},w)&=G_{\Z^2}(v^\pi,w^\pi)\\
G_\Xi(v,w)+G_\Xi(v^{\sigma},w)&=G_{(\Z+\frac12)^2}(v^\pi,w^\pi)\,,
\end{align*}
the Green's function on either $\Sigma$ or $\Xi$ can be easily recovered from the
``antisymmetric Green's functions'', which are defined by
\begin{align*}
G^A_\Sigma(v,w) &= G_\Sigma(v,w)-G_\Sigma(v^{\sigma},w)\\
G^A_\Xi(v,w) &= G_\Xi(v,w)-G_\Xi(v^{\sigma},w)\,.
\end{align*}
The antisymmetric Green's functions on $\Sigma$ and $\Xi$ turn out to be nicer than the
usual Green's functions.  We can also interpret $G^A_\Sigma$ as being defined for $v,w\in\Z^2$
and $G^A_\Xi$ as being defined for $v,w\in(\Z+\frac12)^2$
by taking the principal branch.

\begin{figure}[b!]
\centerline{\begin{tikzpicture}[scale=0.9]
\newcommand{\vbr}{3.5}
\newcommand{\vbth}{80}
\fill [gray!20] ({-5*sin(22.5)},{5*cos(22.5)}) -- ({5*cos(22.5)},{5*cos(22.5)}) -- ({5*cos(22.5)},{-5*cos(22.5)}) -- ({5*sin(22.5)},{-5*cos(22.5)});
\draw ({-5*sin(22.5)},{5*cos(22.5)}) -- (0,0);
\draw [dotted] (0,0) -- ({5*sin(22.5)},{-5*cos(22.5)});
\draw [dashed] ({-5*cos(22.5)},{-5*sin(22.5)}) -- ({5*cos(22.5)},{5*sin(22.5)});
\node [rotate=-67.5,anchor=south] at ({-3*sin(22.5)},{3*cos(22.5)}) {preimage of $D_0$};
\node [rotate=-67.5,anchor=south] at ({3*sin(22.5)},{-3*cos(22.5)}) {other preimage of $D_0$};
\node [rotate=22.5,anchor=south] at ({2.6*cos(22.5)},{2.6*sin(22.5)}) {preimage of $-D_0$};
\node [rotate=22.5,anchor=south] at ({-2.6*cos(22.5)},{-2.6*sin(22.5)}) {other preimage of $-D_0$};
\node [rotate=-90,anchor=north] at ({5*cos(22.5)},2) {principal branch};
\filldraw (3,-.5) circle[radius=1pt];
\node [anchor=west] at (3,-.5) {$w$};
\filldraw ({\vbr*cos(\vbth)},{\vbr*sin(\vbth)}) circle[radius=1pt];
\node [anchor=west] at ({\vbr*cos(\vbth)},{\vbr*sin(\vbth)}) {$v$};
\filldraw ({\vbr*cos(225-\vbth)},{\vbr*sin(225-\vbth)}) circle[radius=1pt];
\node [anchor=west] at ({\vbr*cos(225-\vbth)},{\vbr*sin(225-\vbth)}) {$v^{\tau}$};
\filldraw ({-\vbr*cos(\vbth)},{-\vbr*sin(\vbth)}) circle[radius=1pt];
\node [anchor=west] at ({-\vbr*cos(\vbth)},{-\vbr*sin(\vbth)}) {$v^{\sigma}$};
\end{tikzpicture}}
\caption{The method of images in the branched double cover.  We pair $v$ with $v^{\tau}$ to make the preimage of $D_0$ zero and compute $G_{D_0}$.  We pair $v^{\sigma}$ with $v^{\tau}$ to make the preimage of $-D_0$ zero and compute $G_{-D_0}$, but there are two cases depending on whether $w$ is on the same side of the preimage of $-D_0$ as $v^{\tau}$ or $v^{\sigma}$.}
\label{fig:images}
\end{figure}
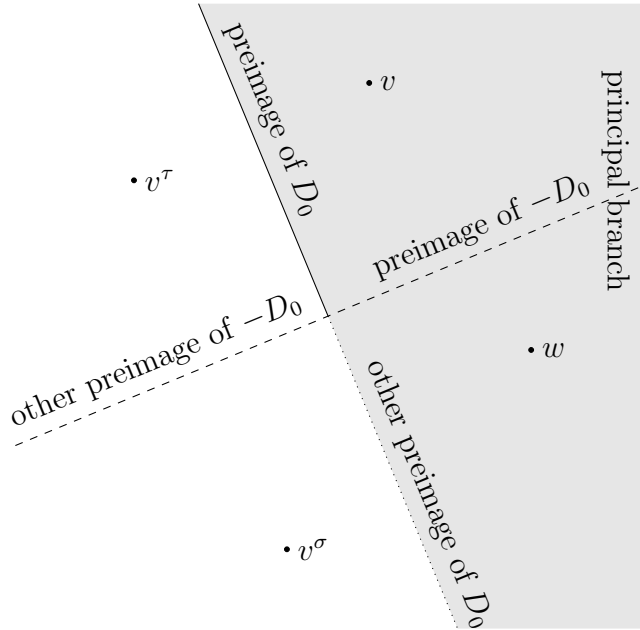

$G^A_\Xi$ can be directly expressed in terms of the Green's function with a zipper.  For $v,w\in\Xi$ we have
\begin{equation} \label{eq:GAP-GZ}
G^A_\Xi(v,w) = G_Z(v^\pi-(\tfrac12,\tfrac12),w^\pi-(\tfrac12,\tfrac12)) \times \begin{cases} +1& \text{$v$ and $w$ on the same branch} \\ -1& \text{$v$ and $w$ on opposite branches} \end{cases}
\end{equation}
$G^A_\Sigma$ can also be expressed in terms of the Green's function with a zipper, but now with Dirichlet boundary conditions at the origin:
For $v,w\in\Sigma$ we have
\begin{equation} \label{eq:GAS-GZ}
G^A_\Sigma(v,w) = \left[G_Z(v^\pi,w^\pi) - \frac{G_Z(v^\pi,0) G_Z(0,w^\pi)}{G_Z(0,0)} \right]\times \begin{cases} +1& \text{$v$ and $w$ on the same branch} \\ -1& \text{$v$ and $w$ on opposite branches} \end{cases}
\end{equation}

For $v,w\in\Z^2$, we can use the method of images in $\Sigma$ (Figure~\ref{fig:images}) to compute $G_{D_0}(v,w)$:
\begin{equation}\label{eq:GD0-GAS}
\begin{aligned}
G_{D_0}(v,w) &= G_\Sigma(v^\beta,w^\beta) - G_\Sigma(v^{\beta\tau},w^\beta) \\
&= \tfrac12 G^A_\Sigma(v^\beta,w^\beta) +\tfrac12 G_{\Z^2}(v,w) - \tfrac12 G^A_\Sigma(v^{\beta\tau},w^\beta) -\tfrac12 G_{\Z^2}(v^{\beta\tau\pi},w)
\end{aligned}
\end{equation}
From this we see that values of $G_{D_0}$ (and hence of $G_D$)
take values in $\Q\oplus \frac{1}{\sqrt{2}}\Q \oplus \frac{1}{\pi}\Q$.

\section{Temperleyan regions with a hole}\label{holesection}

We explain here how to compute the inverse Kasteleyn matrix
for Temperleyan regions with a hole in terms of Green's functions.
For background on the dimer model, the Kasteleyn matrix,
and Temperleyan regions see \cite{Kenyon.lecturenotes}.

\begin{figure}[b!]
\centerline{\includegraphics[width=2.8in]{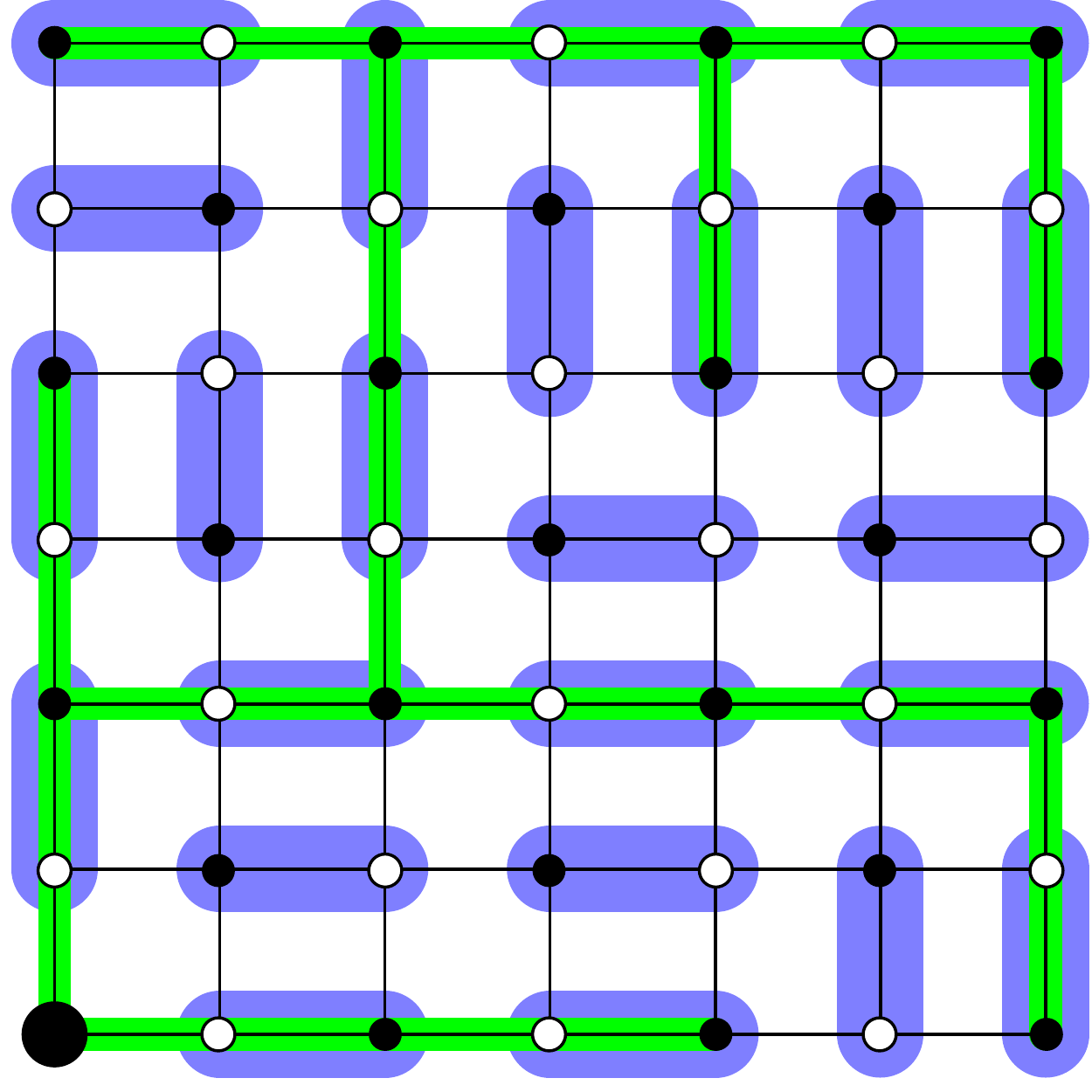}}
\caption{\label{Temptrick}Temperley's bijection. The tree in green (oriented towards the root vertex in the lower left corner)
gives rise to a dimer cover (blue) by laying dimers along the tree edges
when oriented towards the root.}
\end{figure}

We first recall the Temperley's bijection (see \cite{temperley:tree} for the square grid and \cite{KPW} for general planar graphs), between spanning trees of a planar
graph and dimer coverings of an associated graph, illustrated in Figure~\ref{Temptrick}.

\begin{figure}[t]
\centerline{\begin{tikzpicture}[scale=0.75]
\foreach \x in {-2,...,2} {
  \draw [gray,thick] (2*\x,-4) -- (2*\x,4);
  \draw [gray,thick] (-4,2*\x) -- (4,2*\x);
};
\foreach \x in {-2,...,2} {
\foreach \y in {-2,...,2} {
  \filldraw [blue!80!green] ({2*\x},{2*\y}) circle[radius=4pt];
}
};
\node at (0,-6) {$\G$};
\begin{scope}[shift={(12.5,0)}]
\foreach \x in {-2,...,2} {
  \draw [gray,thick] (2*\x,-4) -- (2*\x,4);
  \draw [gray,thick] (-4,2*\x) -- (4,2*\x);
};
\foreach \x in {-2,...,1} {
  \draw [gray,thick] (2*\x+1,-5) -- (2*\x+1,5);
  \draw [gray,thick] (-5,2*\x+1) -- (5,2*\x+1);
};
\foreach \x in {-2,...,2} {
\foreach \y in {-2,...,2} {
  \filldraw [blue!80!green] (2*\x,2*\y) circle[radius=4pt];
}
};
\foreach \x in {-2,...,1} {
\foreach \y in {-2,...,1} {
  \filldraw [red] (2*\x+1,2*\y+1) circle[radius=4pt];
}
};
\foreach \x in {-2,...,1} {
\foreach \y in {-2,...,2} {
  \filldraw [black] (2*\x+1,2*\y) circle[radius=2pt];
  \filldraw [black] (2*\y,2*\x+1) circle[radius=2pt];
}
};
\draw [red, line width=6pt] (-5,-5) -- (5,-5) -- (5,5) -- (-5,5) -- cycle;
\node at (0,-6) {$\G^+$};
\end{scope}
\end{tikzpicture}}
\caption{The graph $\G^+$ is formed from $\G$ (blue) and its dual $\G^\dagger$ (red).}
\label{fig:G+}
\end{figure}
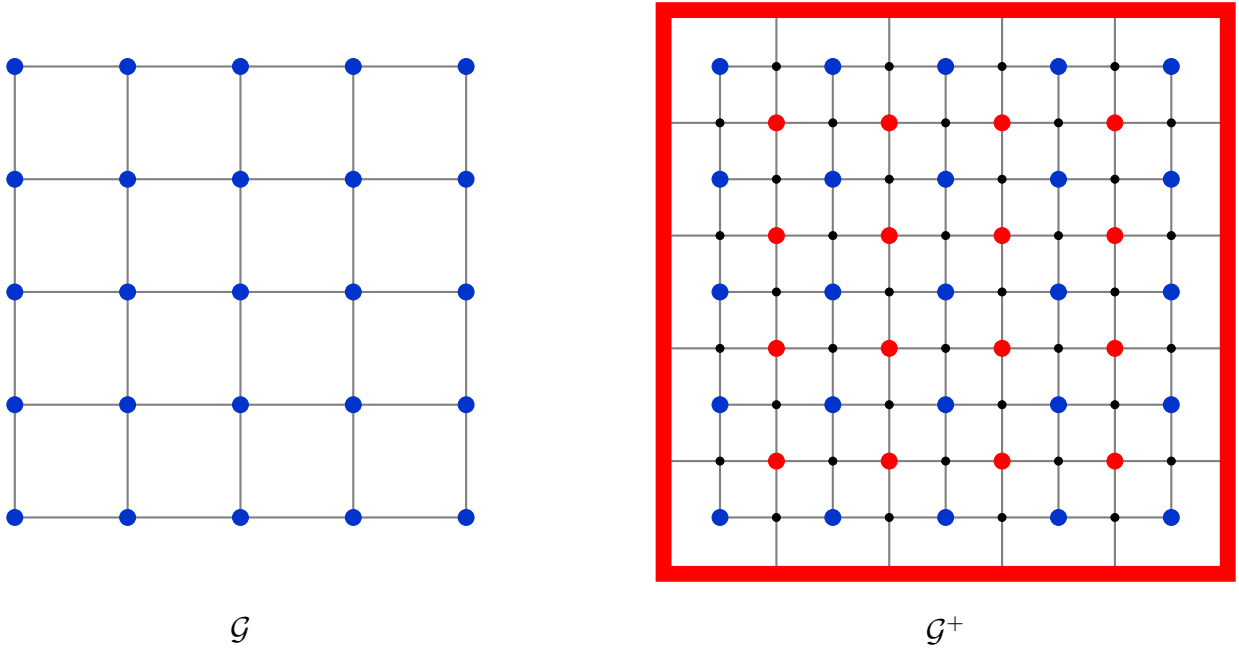

Let $\G$ be a finite
graph embedded in the plane with vertex set $V$, face set $F$, and edge set $E$.
Let $\G^+$ be the graph (embedded in the plane) whose
vertices are $V\cup F\cup E$, with an edge between $e\in E$ and $v\in V$ when
$v$ is an endpoint of $e$ in $\G$, with an edge between $e\in E$ and $f\in F$ when
$e$ bounds the face $f$ in $\G$, and no other edges (Figure~\ref{fig:G+}).
The graph $\G^+$ is bipartite, with parts $V\cup F$ and $E$.
By Euler's formula, $|V|+|F|=|E|+2$.
If $v_0$ and $f_0$ are a distinguished vertex and face of $\G$, and
$v_0$ is incident to $f_0$, then the spanning trees of $\G$ are in
bijective correspondence to the perfect matchings of
$\G^+\setminus\{v_0,f_0\}$.  If $v_0$ and $f_0$ are not incident, then
there is still an injective map from spanning trees of $\G$ to perfect
matchings of $\G^+\setminus\{v_0,f_0\}$, but it is not surjective in
general.  For further background see \cite{KPW}.
We will be interested in regions with holes, so we do not
generally require $v_0$ and $f_0$ to be incident to each other.

\newcommand{\gridh}{
\foreach \x in {1,...,4} {
  \draw [gray,thick] (\x,-4) -- (\x,4);
  \draw [gray,thick] (-4,\x) -- (4,\x);
  \draw [gray,thick] (-\x,-4) -- (-\x,4);
  \draw [gray,thick] (-4,-\x) -- (4,-\x);
};
\draw [gray,thick] (-4,0) -- (-1,0);
\draw [gray,thick] (+4,0) -- (+1,0);
\draw [gray,thick] (0,-4) -- (0,-1);
\draw [gray,thick] (0,+4) -- (0,+1);
}
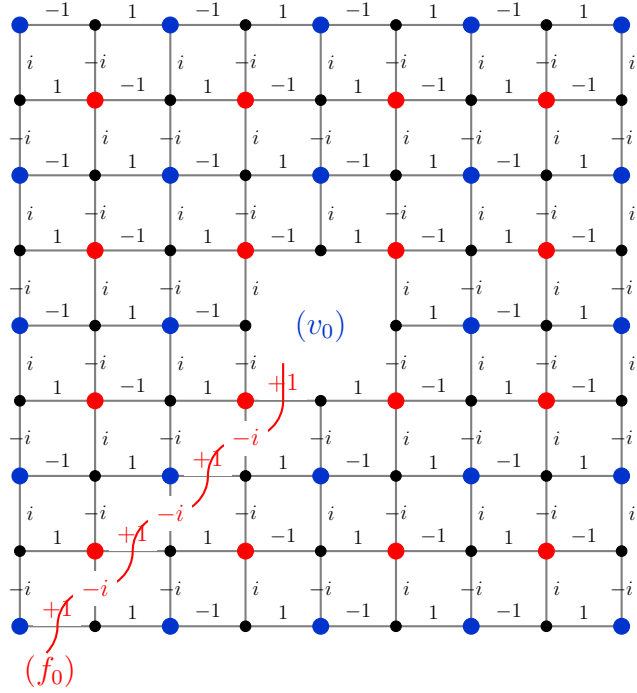
\begin{figure}[t]
\centerline{\begin{tikzpicture}[scale=1]
\gridh
\begin{scope}[every node/.style={scale=0.65}]
\foreach \x in {-4,...,4} {
\foreach \y in {-4,...,4} {
\pgfmathsetmacro\yxdiff{\y-\x}
\ifthenelse {1<0 \OR \isodd{\yxdiff}} {
  \ifthenelse {\y<4 \AND \NOT \(\x=0\AND \y=-1\)}  {\node [anchor=west] at (\x,\y+0.5) {$i$};}{}
  \ifthenelse {\y>-4\AND \NOT \(\x=0\AND\y=1\)} {\node at (\x,\y-0.5) {$-i$};}{}
  \ifthenelse {\x>-4\AND \NOT \(\y=0\AND\x=1\)} {\node [anchor=south] at (\x-0.5,\y) {$-1$};}{}
  \ifthenelse {\x<4 \AND \NOT \(\y=0\AND\x=-1\)} {\node [anchor=south] at (\x+0.5,\y) {$1$};}{}
} {}
}
};
\foreach \y in {-4,...,-1} {\node [red, scale=1.2, fill=white, anchor=south] at (\y+0.5,\y) {$+1$};};
\draw [red,thick] (-0.5,-0.5) -- (-0.5,-1);
\foreach \x in {-3,...,-1} {
  \draw [red,thick] (\x+0.5,\x) arc (0:-90:0.5) arc (90:180:0.5);
};
\draw [red,thick] (-3.5,-4) arc (0:-45:0.5);
\foreach \x in {-3,...,-1} {\node [red, scale=1.2, fill=white] at (\x,\x-0.5) {$-i$};};
\end{scope}
\foreach \x in {1,...,2} {
\foreach \y in {0,...,2} {
  \filldraw [blue!80!green] (2*\x,2*\y) circle[radius=3pt];
  \filldraw [blue!80!green] (-2*\y,2*\x) circle[radius=3pt];
  \filldraw [blue!80!green] (-2*\x,-2*\y) circle[radius=3pt];
  \filldraw [blue!80!green] (2*\y,-2*\x) circle[radius=3pt];
}
};
  \node [blue!80!green] at (0,0) {$(v_0)$};
\foreach \x in {-2,...,1} {
\foreach \y in {-2,...,1} {
  \filldraw [red] (2*\x+1,2*\y+1) circle[radius=3pt];
}
};
  \node [red] at (-3.6,-4.6) {$(f_0)$};
\foreach \x in {-2,...,1} {
\foreach \y in {-2,...,2} {
  \filldraw [black] (2*\x+1,2*\y) circle[radius=2pt];
  \filldraw [black] (2*\y,2*\x+1) circle[radius=2pt];
}
};
\end{tikzpicture}}
\caption{\label{Kastsignshole}The Kasteleyn signs of $\G^+\setminus\{v_0,f_0\}$ when $f_0$ and $v_0$ are not incident to each other.}
\end{figure}

We can make a bipartite Kasteleyn matrix for
$\G^+\setminus\{v_0,f_0\}$ as follows.  Each edge-type vertex $e^+$ of
$\G^+$ is incident to four edges in $\G^+$ which are cyclically
ordered since they are embedded in the plane.  For an arbitrary
complex number $w_e\in\C$ of unit modulus $|w_e|=1$, we assign the weights
$w_e$, $i w_e$, $-w_e$, $-i w_e$ in cyclic order to the edges incident to $e^+$.
(Weights of this type were used in \cite{Kenyon.confinv} for the square grid.)
Every face of $\G^+$ has four sides, and it is easily verified that each face
is ``Kasteleyn-flat''.  (In a signed bipartite graph, a face with $\delta$ edges is Kasteleyn-flat if the
product of signs of every other edge on the boundary equals the product of the other signs
times $-(-1)^{\delta/2}$.)  If $v_0$ and $f_0$ are incident, we can simply delete them,
and the resulting weighted graph remains Kasteleyn-flat, i.e., it gives a valid Kasteleyn matrix.
If $v_0$ and $f_0$ are not incident, we connect them by a path
that avoids vertices of $\G^+$; each time the path crosses an edge of $\G^+$, we change its sign. See Figure~\ref{Kastsignshole}.
Upon removing $v_0$ and $f_0$, the resulting weighted graph is Kasteleyn-flat, giving
a valid Kasteleyn matrix $K$.

We think of the matrix $K$ as mapping functions on $(V\cup F)\setminus\{v_0,f_0\}$ to functions on $E$
($K$ has rows indexed by $E$).
Let $\bar K$ be the conjugate of $K$, and let $K^*=\bar K^T$ be the conjugate transpose of $K$.

Let
$\Delta_{\G,v_0}^{v_0\leftrightsquigarrow f_0}$ be the Laplacian on
$\G$ with Dirichlet boundary at $v_0$ and a zipper on the path from $v_0$ to $f_0$, where edges on the zipper
have a weight of $-1$, and
similarly define
$\Delta_{\G^\dagger,f_0}^{v_0\leftrightsquigarrow f_0}$ to be the Dirichlet Laplacian on the dual $\G^\dagger$ of $\G$ with a zipper on the path from $v_0$ to $f_0$.

A key property of the above choice of weights is that $K^* K$, when restricted to functions on $V$, is just
$\Delta_{\G,v_0}^{v_0\leftrightsquigarrow f_0}$, and when $K^* K$ is restricted to functions on $F$, it is
$\Delta_{\G^\dagger,f_0}^{v_0\leftrightsquigarrow f_0}$.
In other words,
\[
 K^* K = K^T \bar K =  \Delta_{\G,v_0}^{v_0\leftrightsquigarrow f_0} \oplus \Delta_{\G^\dagger,f_0}^{v_0\leftrightsquigarrow f_0}\,.
\]

Let
\[
G_{\G,v_0}^{v_0\leftrightsquigarrow f_0} =
\left(\Delta_{\G,v_0}^{v_0\leftrightsquigarrow f_0}\right)^{-1}
\]
be the Dirichlet Green's function on $\G$ with boundary at $v_0$ and the zipper ${v_0\leftrightsquigarrow f_0}$.
Equivalently, this the antisymmetric Green's function on the branched double cover of $\G$, branched on the path $v_0\leftrightsquigarrow f_0$, with Dirichlet boundary at $v_0$.  Define $G_{\G,v_0}^{v_0\leftrightsquigarrow f_0}$ similarly.
Then
\begin{equation} \label{eq:K-inv-from-G}
K^{-1} = \left(\bar K G_{\G,v_0}^{v_0\leftrightsquigarrow f_0}\right)^T \oplus \left(\bar K G_{\G^\dagger,f_0}^{v_0\leftrightsquigarrow f_0}\right)^T
\end{equation}
In other words, $K^{-1}$ evaluated at $v\in V$ and edge $e\in E$ is (up to complex sign) the current flowing across edge $e$
in the branched double-cover of $\G$ (with Dirichlet boundary at $v_0$) when current is inserted at $v$ in one branch and extracted at $v$ in the other branch,
and $K^{-1}$ at $f\in F$ and $e\in E$ is interpreted similarly for the dual graph $\G^\dagger$.

When $v_0$ and $f_0$ are incident to each other, as in the standard Temperley bijection,
the zipper may be chosen to cross no edges at all, in which case \eqref{eq:K-inv-from-G} expresses
$K^{-1}$ in terms of the Dirichlet Green's function on $\G$ itself.  This relation was used in \cite{Kenyon.confinv}
to express $K^{-1}$ on $\Z^2$ (with Temperleyan boundary conditions)
in terms of the Green's function on $\Z^2$.

\section{The uniform spanning tree trunk}\label{trunksection}

\subsection{From trees to dimers}

Temperley's bijection between trees and dimers was extended in \cite[Lemma~17]{kenyon:det-lap} to a bijection
between dimer covers of regions in $\Z^2$ with a single hole, and spanning trees of associated regions in which a specified
branch goes through the corresponding edge. See Figure~\ref{fig:temperley-annulus} for the case we are interested in here:
Let $\H$ be the $(2n-1)\times(2n-1)$ grid graph with a single hole at the center.
Let $\G$ be the  $(n-1)\times n$ grid graph (with $n-1$ columns and $n$ rows) in which the left and right boundaries
have been wired, in the sense that we add two more vertices $L$ and $R$,
$L$ being connected to all vertices on $\G$'s left boundary
and $R$ being connected to all vertices on $\G$ right boundary.
Then dimer covers of $\H$ correspond
to spanning trees of $\G$ in which the unique tree path connecting $L$ and $R$ passes through the horizontal edge at the center.

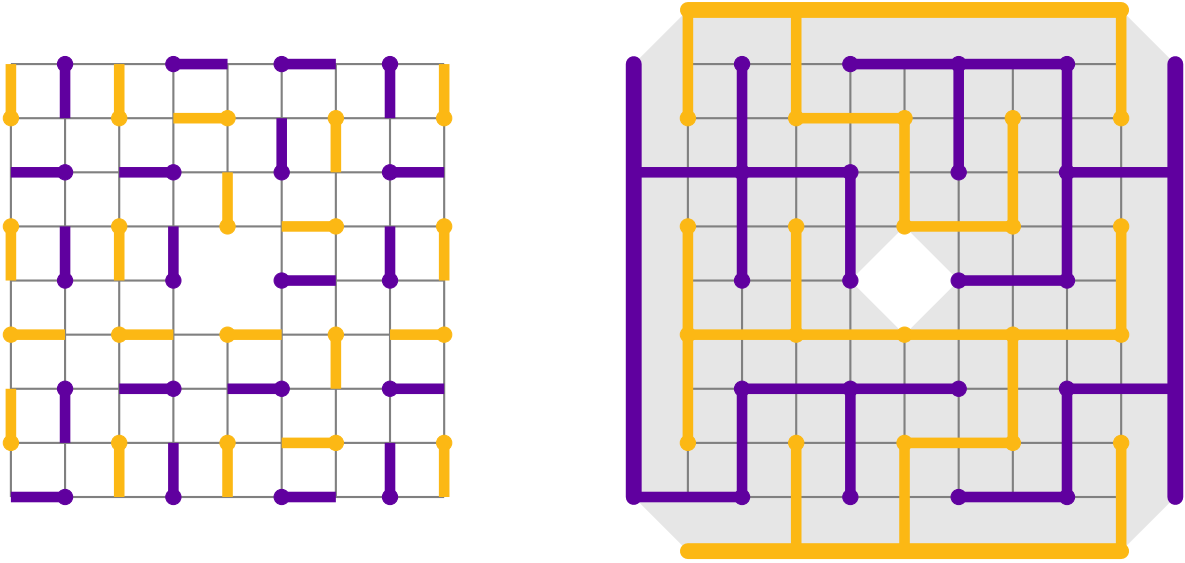
\begin{figure}[h]
\newcommand{\bdimer}{\draw[purple!50!blue, line width=4pt]}
\newcommand{\odimer}{\draw[yellow!40!orange, line width=4pt]}
\newcommand{\dimerLength}{1}
\newcommand{\bdu}[2]{\bdimer (#1,#2)--(#1,#2+\dimerLength);}
\newcommand{\bdd}[2]{\bdimer (#1,#2)--(#1,#2-\dimerLength);}
\newcommand{\bdr}[2]{\bdimer (#1,#2)--(#1+\dimerLength,#2);}
\newcommand{\bdl}[2]{\bdimer (#1,#2)--(#1-\dimerLength,#2);}
\newcommand{\odu}[2]{\odimer (#1,#2)--(#1,#2+\dimerLength);}
\newcommand{\odd}[2]{\odimer (#1,#2)--(#1,#2-\dimerLength);}
\newcommand{\odr}[2]{\odimer (#1,#2)--(#1+\dimerLength,#2);}
\newcommand{\odl}[2]{\odimer (#1,#2)--(#1-\dimerLength,#2);}
\newcommand{\dimers}{\bdr{1}{0} \bdu{3}{0} \bdr{3}{2} \bdu{1}{2} \bdr{1}{4} \bdd{3}{4} \bdl{1}{-2} \bdl{-1}{-2} \bdd{-3}{-2} \bdl{-3}{-4} \bdu{-1}{-4} \bdr{1}{-4} \bdu{3}{-4} \bdr{3}{-2} \bdu{-1}{0} \bdl{-1}{2} \bdl{-3}{2} \bdu{-3}{0} \bdd{-3}{4} \bdr{-1}{4} \odd{-2}{-3} \odu{-4}{3} \odd{2}{3} \odl{2}{1} \odu{0}{1} \odl{0}{3} \odu{-2}{3} \odd{-4}{1} \odr{-4}{-1} \odr{-2}{-1} \odr{0}{-1} \odd{2}{-1} \odl{2}{-3} \odd{0}{-3} \odd{4}{-3} \odd{4}{1} \odl{4}{-1} \odu{4}{3} \odu{-4}{-3} \odd{-2}{1}}
\newcommand{\vfs}{
\foreach \x in {-1,...,2} {
\foreach \y in {-2,...,2} {
  \filldraw [purple!50!blue] ({2*\x-1},{2*\y}) circle[radius=4pt];
  \filldraw [yellow!40!orange] ({2*\y},{2*\x-1}) circle[radius=4pt];
}
};
}
\newcommand{\oddxodd}{\begin{tikzpicture}[scale=0.75]
\gridh
\dimers
\vfs
\end{tikzpicture}}
\centerline{\begin{tikzpicture}[scale=0.72]
\gridh\dimers\vfs
\begin{scope}[shift={(12.5,0)}]
\renewcommand{\dimerLength}{2}
\fill[fill=gray!20!white]
(-5,-4) -- (-5,4) -- (-4,5) -- (4,5) -- (5,4) -- (5,-4) -- (4,-5) -- (-4,-5) -- cycle
(1,0) -- (0,1) -- (-1,0) -- (0,-1) -- cycle;
\gridh\dimers\vfs
\draw [purple!50!blue, line width=6pt, line cap=round] (-5,-4)--(-5,4);
\draw [purple!50!blue, line width=6pt, line cap=round](5,-4)--(5,4);
\draw [yellow!40!orange, line width=6pt, line cap=round] (-4,-5)--(4,-5);
\draw [yellow!40!orange, line width=6pt, line cap=round] (-4,5)--(4,5);
\end{scope}
\end{tikzpicture}}
\caption{Version of Temperley's bijection for an odd by odd square grid with a hole at the center.}
\label{fig:temperley-annulus}
\end{figure}

The dimer statistics on $\H$ are determined by the inverse Kasteleyn matrix \cite{Kenyon.localstats}, which is notoriously sensitive to boundary
conditions \cite{Kenyon.lecturenotes}.  But here we can use a special feature of this set-up, that
the graph $\H$ of the dimer system in Figure~\ref{fig:temperley-annulus}
is identical to the graph in Figure~\ref{Kastsignshole}.
Using formula \eqref{eq:K-inv-from-G} we can compute the inverse Kasteleyn matrix of
$\H$ in terms of the Green's functions for the
graph and dual graph in Figures~\ref{fig:G+} and~\ref{Kastsignshole} with a zipper connecting the center to the boundary.
These finite-graph Green's functions with a zipper converge to
the Green's function on $\Z^2$ with a zipper as long as the distance between
the boundary of the region and the origin tends to infinity.
As a consequence the inverse Kasteleyn matrix for $\H$ converges near the origin as $n\to\infty$, and therefore the local statistics near the origin
for the ``UST trunk'' on $\G$ converge.

Here by ``local statistics'' we mean the probabilities of cylinder events.
A somewhat stronger statement holds, since the dimer local statistics determine the local statistics of the \emph{directed\/}
spanning tree (directed, say, from $L$ to $R$) and directed dual tree.

\begin{theorem} \label{thm:trunk-measure-dimers}
Let $\G_{M,N}$ be the $M\times N$ grid, containing a horizontal edge $e_0$.
Wire the left and right sides of $\G_{M,N}$ to vertices $L$ and $R$ respectively.
Let $T$ be a uniform spanning tree, conditioned on the path within the tree from $L$ to $R$ to pass
through $e_0$.  From every vertex there is a unique path in $T$ leading to one of $\{L,R\}$ that avoids $e_0$;
orient the tree
edge in this direction, and similarly orient the dual tree edges towards the free boundary components while avoiding the dual of $e_0$.
Then the measure on the directed tree and dual tree edges converges as
the distance from $e_0$ to the boundaries of the box go to infinity.
The limiting measure
is determinantal with kernel given by \eqref{eq:K-inv-from-G}  with $G_{\G,v_0}^{v_0\leftrightsquigarrow f_0}, G_{\G^\dagger,f_0}^{v_0\leftrightsquigarrow f_0}$ there replaced by $G_\Sigma^A$ and $G_\Xi^A$.
\end{theorem}

The existence of the limiting measure also follows from work of Lawler \cite{lawler:pr-edge}, and is made more explicit in \cite{lawler:trunk-measure}.

\subsection{Example calculations}
We illustrate with an example calculation.  We can for instance
compute the probability that the trunk turns left at $(1,0)$ and has
two subtrees attached to it.  This event is equivalent to the
existence of dimers at $(1,0)(1,1)$, $(3,0)(2,0)$, and
$(1,-2)(1,-1)$.
The probability of this event is given by the following $3\times 3$ subdeterminant of $K^{-1}$,
where the rows are indexed by the even-index vertices within the three dimers
and the columns are indexed by the odd-index vertices:
\[
\det \begin{bmatrix}
K^{-1}_{(1,1),(1,0)} & K^{-1}_{(1,1),(3,0)} & K^{-1}_{(1,1),(1,-2)} \\
K^{-1}_{(2,0),(1,0)} & K^{-1}_{(2,0),(3,0)} & K^{-1}_{(2,0),(1,-2)} \\
K^{-1}_{(1,-1),(1,0)} & K^{-1}_{(1,-1),(3,0)} & K^{-1}_{(1,-1),(1,-2)}
\end{bmatrix}
\times K_{(1,0),(1,1)} K_{(3,0),(2,0)} K_{(1,-2),(1,-1)}
\]
Since $(1,1)$ and $(1,-1)$ are of type (odd, odd), their rows are obtained from differences of $G_\Xi^A$,
and since $(2,0)$ is of type (even, even), its row is obtained from differences in $G_\Sigma^A$.
For instance,
the $(1,1),(1,0)$ matrix entry is
\[ K^{-1}_{(1,1),(1,0)} = -i G_\Xi^A\big(\tfrac{(1,1)}{2},\tfrac{(1,1)}{2}\big) +i G_\Xi^A\big(\tfrac{(1,1)}{2},\tfrac{(1,-1)}{2}\big) = -i\left(\frac{1}{2}\right) + i \left(-\frac12+\frac{1}{\sqrt{2}}\right) = -i + \frac{i}{\sqrt{2}},\]
and
the $(2,0),(1,-2)$ entry is
\[ K^{-1}_{(2,0),(1,-2)} = G_\Sigma^A\big(\tfrac{(2,0)}{2},\tfrac{(2,-2)}{2}\big) - G_\Sigma^A\big(\tfrac{(2,0)}{2},\tfrac{(0,-2)}{2}\big) = \left(-\frac{1}{2}+\frac{1}{\sqrt{2}}\right) - \left(\frac{3}{2}-\sqrt{2}\right) = -2 + \frac{3}{\sqrt{2}}. \]
The rest of the matrix is computed in the same fashion, and is
\[
\begin{gathered}\bordermatrix[{[]}]{
      & (1,0) &(3,0)& (1,-2)\\
(1,1) & \frac{i}{\sqrt{2}}-i & i \sqrt{2}-\frac{3 i}{2} & i \sqrt{2}-\frac{3 i}{2} \\
(2,0) & \sqrt{2}-1 & 2 \sqrt{2}-3 & \frac{3}{\sqrt{2}}-2 \\
(1,-1)& i-\frac{i}{\sqrt{2}} & \frac{3 i}{2}-i \sqrt{2} & \frac{i}{2}-\frac{i}{\sqrt{2}} \\
}
\end{gathered}
\]
and upon taking the determinant, we find that the probability is
\[
\frac{5}{2} - \frac{7}{2 \sqrt{2}}\,.
\]

Similarly, the probability that the trunk continues straight at $(1,0)$ and has
two subtrees attached to it is given by
\[
\det \begin{bmatrix}
K^{-1}_{(2,0),(1,0)} & K^{-1}_{(2,0),(1,2)} & K^{-1}_{(2,0),(1,-2)} \\
K^{-1}_{(1,1),(1,0)} & K^{-1}_{(1,1),(1,2)} & K^{-1}_{(1,1),(1,-2)} \\
K^{-1}_{(1,-1),(1,0)} & K^{-1}_{(1,-1),(1,2)} & K^{-1}_{(1,-1),(1,-2)}
\end{bmatrix}
\times K_{(1,0),(2,0)} K_{(1,2),(1,1)} K_{(1,-2),(1,-1)} = \frac{5}{\sqrt{2}} - \frac{7}{2}\,.
\]

By combining the last two probabilities, we see that the probability that vertex $(1,0)$ (i.e., a typical vertex on the trunk) has degree 4 is $3/2 - \sqrt{2}$.
The probability that vertex $(1,0)$ has degree $2$ or $3$ can be computed in the same manner using a finite sum of dimer cylinder events,
and these probabilities are $1/2$ and $\sqrt{2}-1$ respectively.
Another interesting event is the probability that the trunk continues straight at vertex $(1,0)$, which is $\sqrt{2}-1$.

\subsection{Geometric distribution for straight runs}

Suppose that the spanning tree trunk contains a straight sequence of $k$ edges.
The probability that the trunk continues straight for a $k+1$-st edge is $\sqrt{2}-1$.
More formally,

\begin{theorem} \label{thm:square-geometric}
In the spanning tree trunk measure for the square lattice,
the probability that the trunk contains
the path $(1,0)(3,0)\cdots(2k+1,0)$ is $(\sqrt{2}-1)^k$.
\end{theorem}

\begin{figure}[b!]
\centerline{\begin{tikzpicture}[scale=0.7]
\foreach \x in {-3,...,2} {\draw [gray,thick] (2*\x,-4) -- (2*\x,4);};
\foreach \y in {-2,...,2} {\draw [gray,thick] (-6,2*\y) -- (4,2*\y);};
\foreach \x in {-3,...,2} {
\foreach \y in {-2,...,2} {
  \filldraw [blue!80!green] ({2*\x},{2*\y}) circle[radius=4pt];
}
};
\draw [blue!80!green, line width=6pt] (-2,0) -- (0,0);
\begin{scope}[shift={(-12,0)}]
\foreach \x in {1,...,4} {
  \draw [gray,thick] (\x,-4) -- (\x,4);
  \draw [gray,thick] (-6,\x) -- (4,\x);
  \draw [gray,thick] (-\x-2,-4) -- (-\x-2,4);
  \draw [gray,thick] (-6,-\x) -- (4,-\x);
};
\draw [gray,thick] (-6,0) -- (-3,0);
\draw [gray,thick] (+4,0) -- (+1,0);
\foreach \x in {-2,...,0} {
\draw [gray,thick] (\x,-4) -- (\x,-1);
\draw [gray,thick] (\x,+4) -- (\x,+1);
}
\begin{scope}[every node/.style={scale=0.55}]
\foreach \x in {-6,...,4} {
\foreach \y in {-4,...,4} {
\pgfmathsetmacro\yxdiff{\y-\x}
\ifthenelse {1<0 \OR \isodd{\yxdiff}} {
 \ifthenelse {\y<4 \AND \NOT \(\x>-3\AND \x<1 \AND \y>-2 \AND \y<1\)}  {\node [anchor=west] at (\x,\y+0.5) {$i$};}{}
 \ifthenelse {\y>-4\AND \NOT \(\x>-3\AND \x<1 \AND\y>-1 \AND \y<2\)} {\node at (\x,\y-0.5) {$-i$};}{}
 \ifthenelse {\x>-6\AND \NOT \(\y=0\AND \x>-2\AND \x<2\)} {\node [anchor=south] at (\x-0.5,\y) {$-1$};}{}
 \ifthenelse {\x<4 \AND \NOT \(\y=0\AND \x>-4\AND \x<0\)} {\node [anchor=south] at (\x+0.5,\y) {$1$};}{}
} {}
}
};
\foreach \y in {-4,...,-1} {\node [red, scale=1.0, fill=white, anchor=south] at (\y+0.5,\y) {$+1$};};
\draw [red,thick] (-0.5,-0.5) -- (-0.5,-1);
\foreach \x in {-3,...,-1} {
  \draw [red,thick] (\x+0.5,\x) arc (0:-90:0.5) arc (90:180:0.5);
};
\draw [red,thick] (-3.5,-4) arc (0:-45:0.5);
\foreach \x in {-3,...,-1} {\node [red, scale=1.0, fill=white] at (\x,\x-0.5) {$-i$};};
\end{scope}
\foreach \x in {-3,...,2} {
\foreach \y in {1,...,2} {
  \filldraw [blue!80!green] (2*\x,2*\y) circle[radius=3pt];
  \filldraw [blue!80!green] (2*\x,-2*\y) circle[radius=3pt];
}
};
\foreach \x in {1,...,2} {\filldraw [blue!80!green] (2*\x,0) circle[radius=3pt];};
\foreach \x in {-3,...,-2} {\filldraw [blue!80!green] (2*\x,0) circle[radius=3pt];};
  \node [blue!80!green, scale=0.85] at (-1,0) {$(v_0)$};
\foreach \x in {-3,...,1} {
\foreach \y in {-2,...,1} {
  \filldraw [red] (2*\x+1,2*\y+1) circle[radius=3pt];
}
};
  \node [red,scale=0.85] at (-3.6,-4.6) {$(f_0)$};
\foreach \x in {-6,...,4} {
\foreach \y in {-4,...,4} {
\pgfmathsetmacro\yxdiff{\y-\x}
  \ifthenelse {\isodd{\yxdiff} \AND \NOT \( \y=0 \AND \x=-1\)} {\filldraw [black] (\x,\y) circle[radius=2pt];} {}
}
};
\end{scope}
\end{tikzpicture}}
\caption{\label{longhole}On the left, the Kasteleyn matrix associated to a UST trunk of length $2$. We relate $K^{-1}$
to the Green's function with Dirichlet boundary conditions on the central black segment of the figure on the right.}
\end{figure}
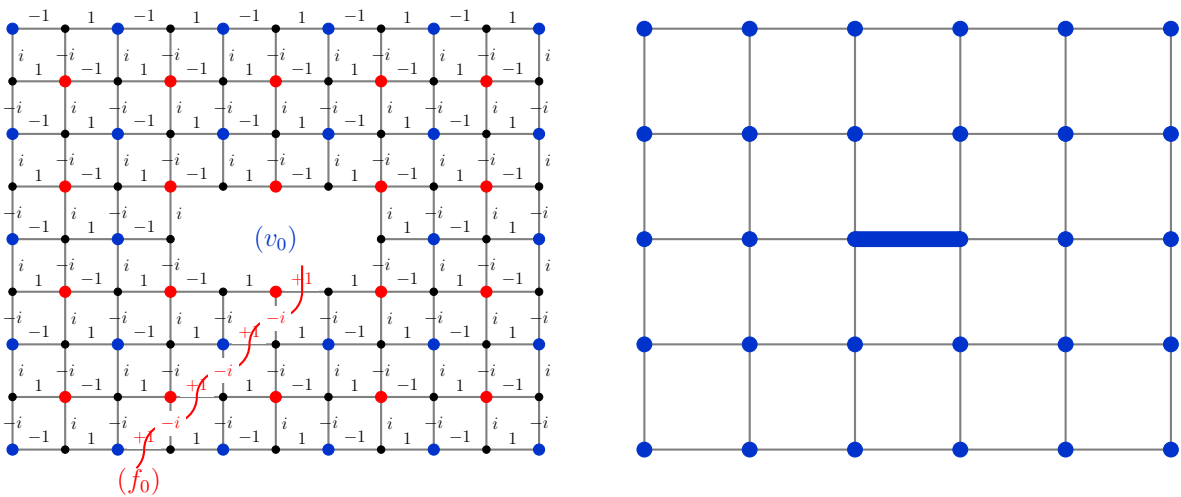

\begin{proof}
We use induction on $k$, with the case $k=0$ being trivial.
In the Temperleyan region associated with the box, we remove the $2k+1$ vertices $(0,0),(1,0),\ldots,(2k,0)$ which correspond to the edges and interior vertices of
a path of length $k+1$, and let $R$ denote the resulting region.
The conditional probability that $P$ also contains the next edge $(2k+1,0)(2k+3,0)$ is the probability that the dimer $(2k+2,0)(2k+3,0)$ is present
in a random dimer covering of~$R$.
To compute this probability, we invert the Kasteleyn matrix with a zipper (illustrated in Figure~\ref{longhole}).  This in turn can be expressed in terms of the Green's
function $G$ on type-$(0,0)$ vertices within the box, with a zipper from the hole to the outer boundary, with Dirichlet boundary at $(0,0),(2,0),\ldots,(2k,0)$,
and in particular it is $G_{(2k+2,0),(2k+2,0)}$.  We may deform the zipper so that it lies just below the negative $x$-axis, and see that $G_{(2k+2,0),v}$ is just the
Green's function with Dirichlet boundary at $\{(x,0)\;:\;x \leq 2k\}$, evaluated at $(2k+2,0)$ and $v$.  With $v=(2k+2,0)$, in the limit where
the boundary of the box tends to infinity, this is $\sqrt{2}-1$.
\end{proof}

There is another argument for the geometric distribution for straight runs in the trunk.
Lawler proved a combinatorial identity \cite[Thm.~3.1]{lawler:pr-edge} which he used to estimate the probability
that an edge lies on the uniform spanning tree path connecting two sides of a box (see also \cite{BLV} for more precise results.)
The focus of these works was scaling limit properties rather than the local measure around the trunk,
but the proof of \cite[Thm.~3.1]{lawler:pr-edge} can be adapted to show that the local measure exists.
The method of proof does not lend itself to explicit computations of probabilities in the trunk measure,
with one exception:
the argument in the proof can be extended to show that
the probability that the trunk continues straight $k$ times is the $k$th power of $G_{(0,0),(0,0)}$ where $G$
the Green's function on $\Z^2$ with Dirichlet boundary conditions at $\{(x,0)\;:\;x<0\}$.
We already have a short proof of this for the square lattice, but the proof of \cite[Thm.~3.1]{lawler:pr-edge} can also be adapted to
work for the triangular lattice, since it has a reflection symmetry and other properties required by the proof.
Rather than give a complete argument here for the geometric runs property for the triangular lattice, which would reproduce large portions of the
proof of \cite[Thm.~3.1]{lawler:pr-edge}, we direct the reader to read \cite[Proof of Thm.~3.1]{lawler:pr-edge} with this new claim in mind.
We will see in Section~\ref{triangularsection} that the relevant constant is $2-\sqrt{3}$.

\section{Near a triple point of the UST}\label{triplesection}

Consider the UST on an $n\times n$ square grid, and consider the
event $A=A(v_t,v_1,v_2,v_3)$ that a nonboundary vertex $v_t$ has three disjoint
branches to three boundary vertices $v_1,v_2,v_3$.
We call such a point $v_t$ a \emph{triple point}.
We orient the edges of the tree towards $v_1$; the edges of the dual tree are naturally oriented towards the outer face.
See Figure~\ref{Temptrick}, with $v_1$ the lower left corner, $v_2$ the upper right corner, and $v_3$ the lower right corner.

We show that the oriented edges and dual edges of this UST conditioned on event $A$ form a determinantal process, with kernel
obtained from the inverse Kasteleyn matrix for an associated graph. The \emph{unoriented\/} tree edges also form a determinantal
process, with a kernel obtained from this one.

In the limit $n\to\infty$ with the distance from $v_t$ to the boundary also going to $\infty$,
the inverse Kasteleyn matrix converges near $v_t$ and we find a limiting determinantal
measure on the UST on $\Z^2$ conditioned to have a tripod point at a given point, for example the origin.

\subsection{Temperley's bijection}

As before we take $\G$ to be an $n\times n$ grid and $\G^+$ to be a $(2n-1)\times(2n-1)$ grid obtained by taking the superposition of $\G$ and its planar dual.
We make another bipartite graph $\H_{NE}$ from $\G^+$  by removing black vertices $f_0,v_1,v_2,v_3$
(where $f_0$ corresponds the outer face of $\G$, and $v_1,v_2,v_3$ are on the boundary of $\G$) and
removing white vertices $w_1,w_2$ corresponding to edges $e_1,e_2$ going $E$ and $N$ respectively from $v_t$.
Let $A^{NE}\subset A$ be the event that the tree has a triple point at $v_t$ and two of the three paths from $v_t$ start in the $N$ and $E$
directions.  Similarly define $A^{NW},A^{SW},A^{SE}$.
Let $A^{NEW}\subset A$ be the event that the tree has a triple point at $v_t$ and the three paths start in the $N$, $E$, and $W$ directions,
and similarly define $A^{SEW},A^{NES},A^{NWS}$.  Let $Z^{NE}$, $Z^{NEW}$, etc., denote the number of spanning trees of the above types.

\begin{figure}[htbp]
\centerline{\includegraphics[width=2in]{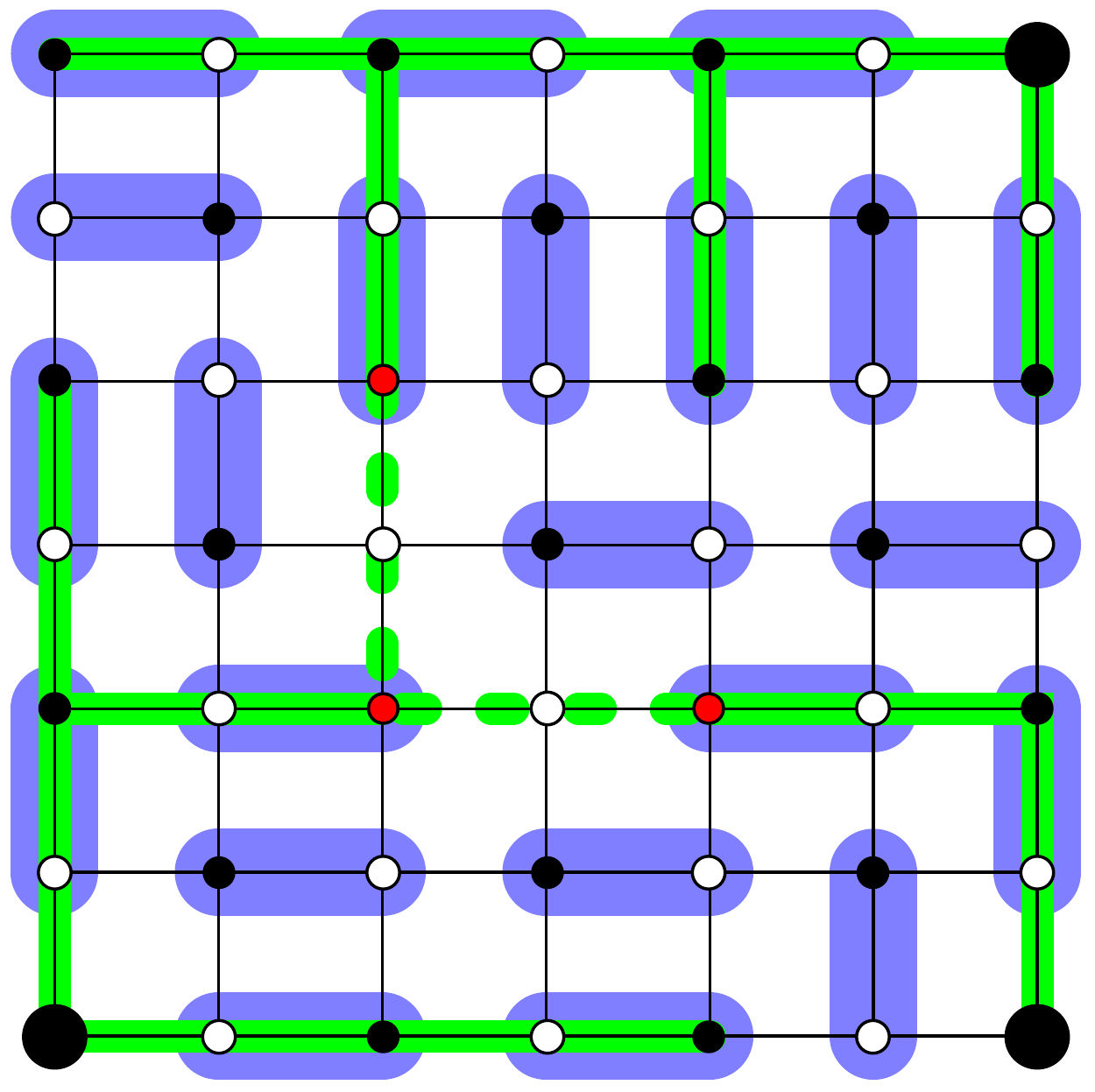}\hspace{-2in}\hspace{-11.0pt}\raisebox{-7.2pt}[0pt][0pt]{\begin{tikzpicture}[scale=0.76]
\node [anchor=north east] at (0,0) {\small$v_1$};
\node [anchor=north west] at (6,0) {\small$v_2$};
\node [anchor=south west] at (6,6) {\small$v_3$};
\contourlength{1.6pt}
\node [scale=0.9] at (2,2) {\contour{black}{\small $v_t$}};
\node [scale=0.9] at (4,2) {\contour{black}{\small $b_1$}};
\node [scale=0.9] at (2,4) {\contour{black}{\small $b_2$}};
\node [scale=0.9] at (3,2) {\contour{black}{\small $w_1$}};
\node [scale=0.9] at (2,3) {\contour{black}{\small $w_2$}};
\contourlength{1.2pt}
\node [scale=0.9] at (2,2) {\contour{gray!50}{\small $v_t$}};
\node [scale=0.9] at (4,2) {\contour{gray!50}{\small $b_1$}};
\node [scale=0.9] at (2,4) {\contour{gray!50}{\small $b_2$}};
\node [scale=0.9] at (3,2) {\contour{white}{\small $w_1$}};
\node [scale=0.9] at (2,3) {\contour{white}{\small $w_2$}};
\end{tikzpicture}}}
\caption{\label{tripodtemp}The tripod version of Temperley's bijection,
with the spanning tree in green and dimer configuration in blue.
Here the boundary vertices $v_1,v_2,v_3$ are at the SW, SE, NE corners,
and the tripod vertex $v_t$ with its east and north neighbors are labeled.}
\end{figure}
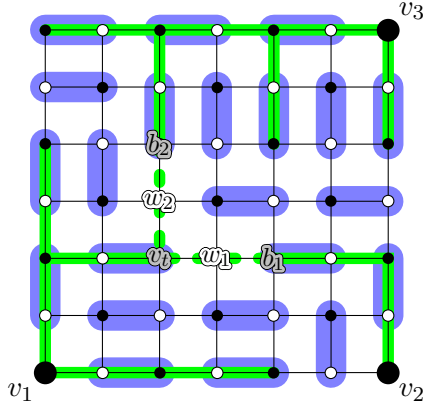

\begin{lemma}
Temperley's bijection extends to a bijection between dimer covers of $\H_{NE}$ and spanning trees of $\G$ in $A^{NE}$.
\end{lemma}

\begin{proof} See Figure~\ref{tripodtemp}.
Let $b_1,b_2,b_3,b_4$ be the E,N,W,S (respectively) neighbors of $v_t\in\G$. When we apply Temperley's bijection to a dimer cover of $\H_{NE}$, the tree branches starting at
$v_t,b_1,b_2$ will necessarily be disjoint and land on $v_1,v_2,v_3$.
Indeed, if any two of these branches meet, the region enclosed between them will have an odd number of vertices of $\H_{NE}$,
contradicting the existence of a dimer cover.

Conversely, given a spanning tree of $\G$ in $A^{NE}$, root the tree at $v_1$ and orient all edges towards $v_1$. Now
change the orientation on the paths between $v_2$ and $v_t$ and between $v_3$ and $v_t$.
When we the convert the tree to a dimer cover as per Temperley's bijection, the dimer cover will leave $w_1,w_2$ uncovered.
\end{proof}

We also make use of another version of Temperley's bijection.
From $\G^+$ remove $f_0,v_1,v_2,v_3$
and contract $w_1,w_2,w_3,w_4,v_t$ to a single white vertex $w_0$, i.e.,
so $w_0$ is attached to all of the neighbors of $w_1,\dots,w_4$ with multiplicity,
and denote the resulting graph $\H_{tr}$ (tr is short for ``tripod'').
We make a Kasteleyn matrix $K_{tr}$ for $\H_{tr}$
by keeping the Kasteleyn signs from the normal Kasteleyn matrix $K_{NE}$ for $\H_{NE}$, and choosing the
signs for $K_{tr}(w_0,b)$ as shown in Figure~\ref{fig:K_tr}.
(Since an even number of vertices are removed near the point of interest in the dimer graphs $\H_{NE}$ and $\H_{tr}$,
there is no zipper of $-1$'s in the Kasteleyn matrix as there was in the trunk calculations.  Consequently it
is the Green's function on $\Z^2$ rather than the branched double-cover of $\Z^2$ that will be relevant
for the tripod calculations.)
\begin{lemma}[\cite{kenyon:jmp}]
Dimer covers of the graph $\H_{tr}$ correspond through Temperley's bijection to spanning trees of
$\G$ in $A(v_t,v_1,v_2,v_3)$.  As a consequence, the edge process in $A(v_t,v_1,v_2,v_3)$ is determinantal
(with kernel determined by the inverse Kasteleyn matrix of $\H_{tr}$).
\end{lemma}
\begin{proof}
There are 12 edges (with multiplicity) connecting vertex $w_0$ to the rest of the graph, and each
of these edges was originally connected to one of $w_1,\dots,w_4$ before the contraction.
Say that $w_0$ is paired via an edge originally connecting to $w_i$.  Then in the tree given by
Temperley's bijection, $w_0$ plays the role of $w_i$, and the tree contains a tripod at $v_t$ using the edges $\{w_1,\dots,w_4\}\smallsetminus\{w_i\}$.
\end{proof}

\begin{figure}[t]
\centerline{\begin{tikzpicture}
\newcommand{\htrlim}{3.4}
\begin{scope}
\begin{scope}[every path/.style={gray,thick}]
\foreach \y in {-3,...,3} \draw (-\htrlim,\y) -- (\htrlim,\y);
\foreach \x in {-3,...,3} \draw (\x,-\htrlim) -- (\x,\htrlim);
\end{scope}
\begin{scope}[every node/.style={scale=0.6,fill=white,inner sep=1.2pt}]
\foreach \x in {-3,...,3} {
\foreach \y in {-3,...,3} {
\pgfmathsetmacro\yxdiff{\y-\x}
\ifthenelse {\isodd{\yxdiff}} {
\draw [fill=white] (\x,\y) circle [radius=2pt];
 \ifthenelse {\y<3}  {\node at (\x,\y+0.5) {$i$};}{}
 \ifthenelse {\y>-3} {\node at (\x,\y-0.5) {$-i$};}{}
 \ifthenelse {\x>-3} {\node at (\x-0.5,\y) {$-1$};}{}
 \ifthenelse {\x<3}  {\node at (\x+0.5,\y) {$1$};}{}
} {
\draw [fill=black] (\x,\y) circle [radius=2pt];
}
};
};
\end{scope}
\begin{scope}[every node/.style={anchor = south west, inner sep=1.5pt}]
\node at (0,0) {$v_t$};
\node at (1,0) {$w_1$};
\node at (0,1) {$w_2$};
\node at (-1,0) {$w_3$};
\node at (0,-1) {$w_4$};
\node at (+2,0) {$b_1$};
\node at (-2,0) {$b_3$};
\node at (0,+2) {$b_2$};
\node at (0,-2) {$b_4$};
\end{scope}
\end{scope}
\node at (4.5,0) {$\Longrightarrow$};
\begin{scope}[shift={(9,0)}]
\begin{scope}[every path/.style={gray,thick}]
\foreach \y in {-3,-2,0,2,3} \draw (-\htrlim,\y) -- (\htrlim,\y);
\foreach \x in {-3,-2,0,2,3} \draw (\x,-\htrlim) -- (\x,\htrlim);
\foreach \s in {-1,1} {
\draw (-\htrlim,\s)--(-1,\s);
\draw (+\htrlim,\s)--(+1,\s);
\draw (\s,-\htrlim)--(\s,-1);
\draw (\s,+\htrlim)--(\s,+1);
};
\foreach \t in {0,90,180,270} {
\draw (0,0) arc (\t:{\t+90}:1);
\draw (0,0) arc (\t:{\t-90}:1);
};
\end{scope}
\begin{scope}[every node/.style={scale=0.6,fill=white,inner sep=1.2pt}]
\foreach \x in {-3,...,3} {
\foreach \y in {-3,...,3} {
\pgfmathsetmacro\yxdiff{\y-\x}
\ifthenelse {\(\y<2 \AND \y>-2 \AND \x=0\) \OR \(\x<2 \AND \x>-2 \AND \y=0\)} {} {
\ifthenelse {\isodd{\yxdiff}} {
\draw [fill=white] (\x,\y) circle [radius=2pt];
 \ifthenelse {\y<3}  {\node at (\x,\y+0.5) {$i$};}{}
 \ifthenelse {\y>-3} {\node at (\x,\y-0.5) {$-i$};}{}
 \ifthenelse {\x>-3} {\node at (\x-0.5,\y) {$-1$};}{}
 \ifthenelse {\x<3}  {\node at (\x+0.5,\y) {$1$};}{}
} {
\draw [fill=black] (\x,\y) circle [radius=2pt];
}}
};
};
\draw [fill=white] (0,0) circle [radius=2pt];
\node at (+1.2,0) {$1$};
\node at (-1.2,0) {$1$};
\node at (0,+1.2) {$-1$};
\node at (0,-1.2) {$-1$};
\node at (+.9,+0.5) {$i$};
\node at (+0.5,+.9) {$i$};
\node at (-.9,-0.5) {$i$};
\node at (-0.5,-.9) {$i$};
\node at (-.9,+0.5) {$-i$};
\node at (-0.5,+.9) {$-i$};
\node at (+.9,-0.5) {$-i$};
\node at (+0.5,-.9) {$-i$};
\end{scope}
\begin{scope}[every node/.style={anchor = south west, inner sep=1.5pt}]
\node at (0,0) {$w_0$};
\node at (2,0) {$b_1$};
\node at (3,0) {$w_E$};
\node at (2,1) {$w_N$};
\node at (2,-1) {$w_S$};
\end{scope}
\end{scope}
\end{tikzpicture}}
\caption{Local graph transformation in which $v_t$ and its neighbors
$w_1,w_2,w_3,w_4$ are contracted to a white vertex $w_0$.  The Kasteleyn
weights of the resulting graph stay the same as the original weights, except for edges incident to $w_0$, which are updated as shown.}
\label{fig:K_tr}
\end{figure}
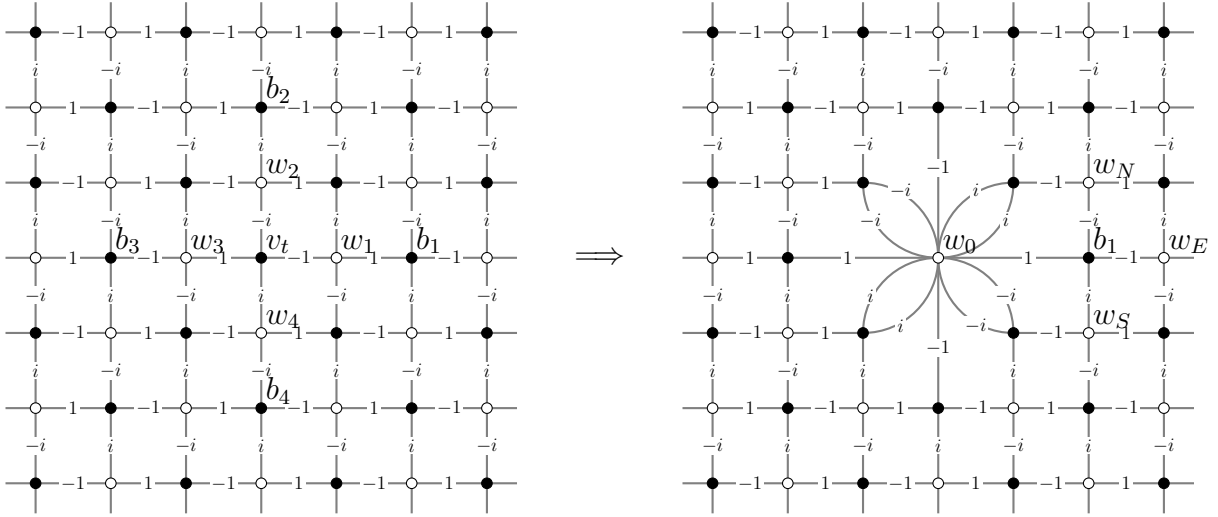

\subsection{Computation}
The Kasteleyn matrix $K$ on $\G^+\setminus\{f_0,v_1\}$ is related to the Green's function on $\G$ as discussed in section~\ref{holesection}.
To obtain the Kasteleyn matrix $K_{NE}$ for $\H_{NE}$ from the Kasteleyn matrix $K$ for $\G^+\setminus\{f_0,v_1\}$, we simply
delete rows $w_1,w_2$ and columns $v_2,v_3$.

Now
\[K_{NE}^{-1}(b,w)  = \frac{\det[K-\{\text{rows $w,w_1,w_2$, columns $b,v_2,v_3$}\}]}
{\det[K-\{\text{rows $w_1,w_2$, columns $v_2,v_3$}\}]},\]
or, using the Jacobi relation between minors of a matrix and minors of its inverse,

\begin{lemma} \label{lem:Kinv_NE}
\be\label{4over3}
K^{-1}_{NE}(b,w) = \frac{\det\begin{pmatrix}K^{-1}(b,w)&K^{-1}(b,w_1)&K^{-1}(b,w_2)\\
K^{-1}(v_2,w)&K^{-1}(v_2,w_1)&K^{-1}(v_2,w_2)\\
K^{-1}(v_3,w)&K^{-1}(v_3,w_1)&K^{-1}(v_3,w_2)
\end{pmatrix}}{\det\begin{pmatrix}K^{-1}(v_2,w_1)&K^{-1}(v_2,w_2)\\
K^{-1}(v_3,w_1)&K^{-1}(v_3,w_2)\end{pmatrix}}.\ee
The denominator in \eqref{4over3} is the probability that a random tree in the graph has a tripod at $v_t$ connecting $v_1,v_2,v_3$
and in which two of the three initial directions are north and east, i.e., it is $\Pr[A^{NE}]$.
\end{lemma}

\subsection{Limit}
For simplicity we will take a specific limit as $n\to\infty$; for the general case, one can use the estimates
for $K^{-1}$ for a general Temperleyan region in \cite{Kenyon.confinv}.
As $n$ increases we center the $n\times n$ box $\G_n$ on the middle of its lower boundary, so that $\G_n$ converges to the upper half plane.
We also assume that the size of $\G_n$ grows faster than the distances between the $v_i$ and $v_t$, and that these distances grow faster
than the distances $|b-v_t|$ and $|w-v_t|$:
let $m=\sqrt{n}$, and suppose
$v_1,v_2,v_3$ are on the lower boundary of $\G$
and converge when rescaled by $1/m$ to distinct points $z_1,z_2,z_3\in\R$.
Suppose further that $v_t$ when rescaled by $1/m$ converges to a point $u\in\C$ with $\Im(u)>0$.
(So $b,w,w_1,w_2$ when rescaled by $1/m$ all converge to $u$.)
Then for $n$ large, the Green's function $G(b,v_t)$ approximates the Neumann Green's function
for the upper half plane, and in particular
\[G(v_i,v_t) = \frac{\log{m}}{\pi} + \frac1{2\pi}\log[(z_i-u)(z_i-\bar u)] + O(1/m^2)\]
\cite{Stohr}.
Upon taking discrete derivatives we find
\[K^{-1}(v_j,w_1) = -\frac1{\pi m}\left(\frac1{z_j-u} + \frac{1}{z_j-\bar u}\right)+O(1/m^2)\]
and
\[K^{-1}(v_j,w_2) = -\frac{i}{\pi m}\left(\frac1{z_j-u} - \frac{1}{z_j-\bar u}\right)+O(1/m^2).\]
So the denominator of \eqref{4over3} is
\be\label{nonzerodet}\frac{i}{\pi^2 m^2}\det\begin{pmatrix}
\frac1{z_2-u} + \frac{1}{z_2-\bar u}&\frac1{z_2-u} - \frac{1}{z_2-\bar u}\\
\frac1{z_3-u} + \frac{1}{z_3-\bar u}&\frac1{z_3-u} - \frac{1}{z_3-\bar u}
\end{pmatrix} +O(1/m^3) =
\frac{i}{\pi^2 m^2}\frac{2(z_2-z_3)(u-\bar u)}{|(z_2-u)(z_3-u)|^2} + O(1/m^3).\ee

If $w$ is horizontal, then
the last two rows of the matrix in the numerator of \eqref{4over3}, when multiplied by $m$,
have the form
\[\begin{pmatrix}A+O(|w-w_1|/m)&A&C\\B+O(|w-w_1|/m)&B&D\end{pmatrix}\]
where $AD-BC$ is bounded away from $0$ by equation \eqref{nonzerodet}.
Under this condition on $AD-BC$, we have the identity
\[\frac{\det\begin{pmatrix}a&b&c\\A+O(|w-w_1|/m)&A&C\\B+O(|w-w_1|/m)&B&D\end{pmatrix}}{\det\begin{pmatrix}A&C\\B&D\end{pmatrix}} = a-b+O(c|w-w_1|/m) + O(b|w-w_1|/m) = a-b+o(1),\]
so \eqref{4over3} tends to $K^{-1}(b,w) -K^{-1}(b,w_1)$ as $n\to\infty$.

Similarly if $w$ is vertical then
the ratio \eqref{4over3} takes the form
\[\frac{\det\begin{pmatrix}a&b&c\\C+O(|w-w_2|/m)&A&C\\D+O(|w-w_2|/m)&B&D\end{pmatrix}}{\det\begin{pmatrix}A&C\\B&D\end{pmatrix}} = a-c+o(1).\]

Thus
\begin{equation} \label{eq:K_NE^{-1}}
K^{-1}_{NE}(b,w) \to \begin{cases}K^{-1}(b,w) -K^{-1}(b,w_1) & \text{if $w$ is a horizontal edge,} \\ K^{-1}(b,w) -K^{-1}(b,w_2) & \text{if $w$ is a vertical edge.} \end{cases}
\end{equation}

Similar considerations apply when the initial directions of two of the tripod edges from $v_t$
are any two adjacent directions among $N,E,S,W$.

If we wish to consider all possible sets of directions we need $K^{-1}_{tr}$.
To compute $K^{-1}_{tr}(b,w)$ we can use Cramer's rule and express it in terms of the values that we computed already.
The denominator is $\det K_{tr} = Z^{NEW}+Z^{SEW}+Z^{NES}+Z^{NWS}$, which we also denote as $Z_{tr}$.  The numerator can be expressed as a sum of four determinants (according to how $w_0$ is matched)
in which vertices $b$ and $w$ are removed.  We can denote the first such determinant by $Z^{NEW}_{b,w}$, which is a signed sum over dimer configurations with monomers
at $b$ and $w$.
Let $\sigma=\pm 1$ depending on whether or not row $w$ and column $b$ have the same parity position in $K_{tr}$.
(Here we assume that the white vertices $w_0,w_1,w_2,w_3,w_4$ are listed last and that black vertex $v_t$ is listed last.)
Then for $w\neq w_0$,
\begin{align*}
K^{-1}_{tr}(b,w)
 &= \sigma \frac{Z^{NEW}_{b,w}+Z^{SEW}_{b,w}+Z^{NES}_{b,w}+Z^{NWS}_{b,w}}{Z_{tr}} \\
 &= \sigma \frac{Z^{NE}_{b,w}+Z^{NW}_{b,w}+Z^{SW}_{b,w}+Z^{SE}_{b,w}}{2 Z_{tr}} \\
 &= \frac{Z^{NE} K^{-1}_{NE}(b,w) + Z^{NW} K^{-1}_{NW}(b,w) + Z^{SW} K^{-1}_{SW}(b,w) + Z^{SE} K^{-1}_{SE}(b,w)}{2 Z_{tr}}\,.
\end{align*}
Recall that the denominator in \eqref{4over3} gives $\Pr(A^{NE})=Z^{NE}/\text{(total number of trees)}$.
We obtain the same limiting expression for each of $\Pr(A^{NE})$, $\Pr(A^{NW})$, $\Pr(A^{SW})$, $\Pr(A^{SE})$,
so the coefficients $Z^{NE}$, $Z^{NW}$, $Z^{SW}$, $Z^{SE}$ are equal up to factors of $1+o(1)$.  Since $Z^{NE}+Z^{NW}+Z^{SW}+Z^{SE} = 2Z_{tr}$,
\[ K^{-1}_{tr}(b,w) = \frac{K^{-1}_{NE}(b,w) + K^{-1}_{NW}(b,w) + K^{-1}_{SW}(b,w) + K^{-1}_{SE}(b,w)}{4} + o(1)\,, \]
and hence in the limit
\begin{equation}
K^{-1}_{tr}(b,w) \to K^{-1}(b,w) -
 \begin{cases}\frac12(K^{-1}(b,w_1)+K^{-1}(b,w_3))&\text{if $w$ is horizontal}\\
\frac12(K^{-1}(b,w_2)+K^{-1}(b,w_4))&\text{if $w$ is vertical.}\end{cases}
\end{equation}

We can solve for the remaining case $w=w_0$ using the fact that $K_{tr} K^{-1}_{tr}$ is the identity.
Let $b_1$ be the black vertex east of $w_1$ (before $w_1$ was contracted) and let $w_S$, $w_E$, and $w_N$
be the white neighbors to the south, east, and north of $b_1$ (see Figure~\ref{fig:K_tr}).  Then
\begin{align}
K^{-1}_{tr}(b,w_0) &= K^{-1}_{tr}(b,w_E) + i K^{-1}_{tr}(b,w_N) - i K^{-1}_{tr}(b,w_S) + \delta_{b,b_1} \notag\\
&\to K^{-1}(b,w_E) + i K^{-1}(b,w_N) - i K^{-1}(b,w_S) - \tfrac12 (K^{-1}(b,w_1)+K^{-1}(b,w_3)) + \delta_{b,b_1} \notag\\
&= K^{-1}(b,w_1) - \tfrac12 (K^{-1}(b,w_1)+K^{-1}(b,w_3)) \notag\\
&= \tfrac{1}{2} (K^{-1}(b,w_1) - K^{-1}(b,w_3)) = \tfrac{1}{2} i (K^{-1}(b,w_4) - K^{-1}(b,w_2))\,.
\end{align}

For large boxes ($n\to\infty$) the values of $K^{-1}(b,w)$ and $K^{-1}(b,w_i)$ converge to the differences of
Green's functions on $\Z^2$ \cite{Kenyon.confinv} (see equation~\eqref{eq:K-inv-from-G} and
Figure~\ref{Green}).

\subsection{Examples}
\begin{figure}[b!]
\centerline{\includegraphics[width=4in]{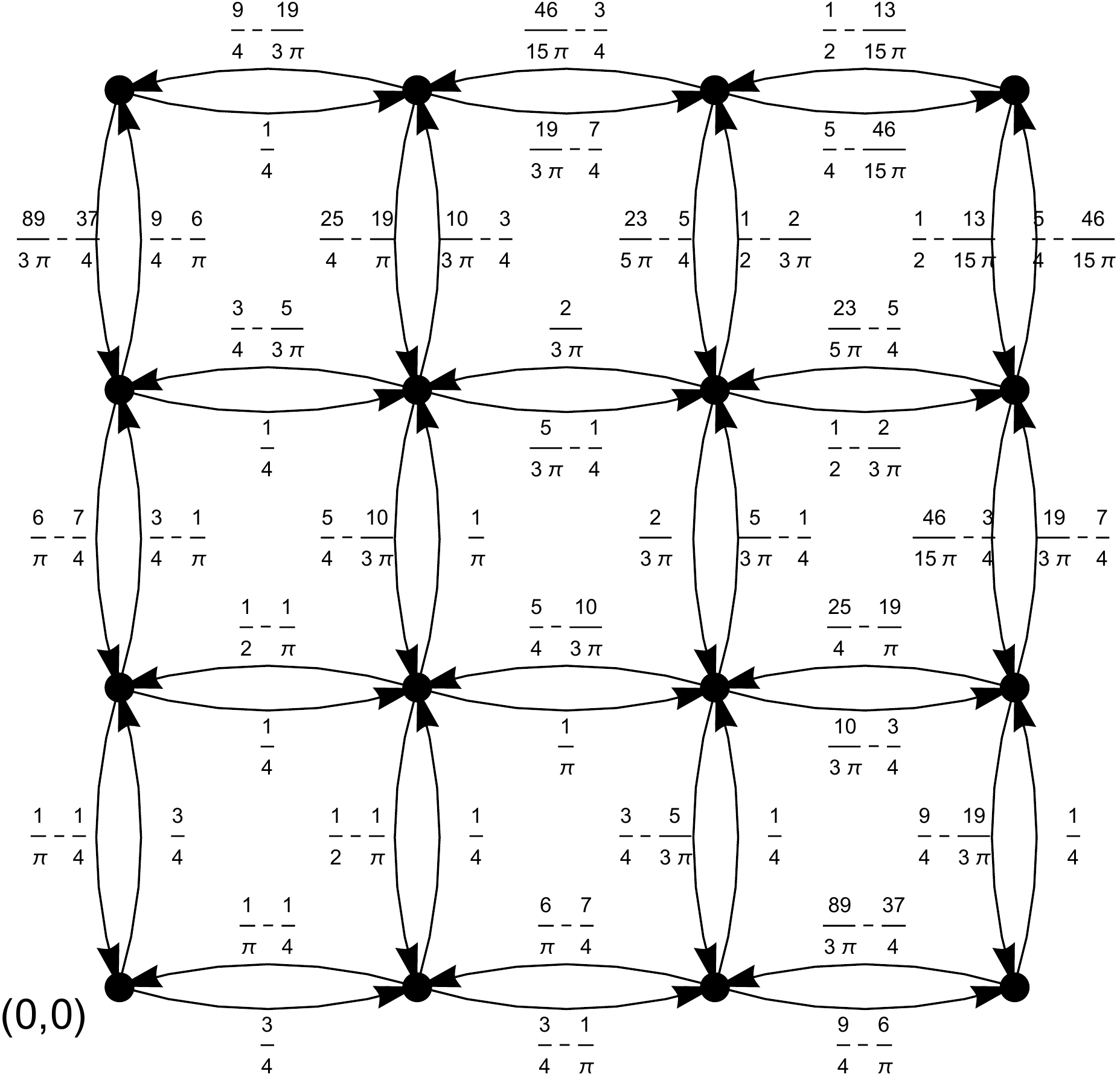}}
\caption{\label{tripodprobs}Directed edge probabilities for the UST on $\Z^2$ conditioned on having a tripod point at the origin, where each edge is directed towards $\infty$, and the three branches of the tripod are directed away from the tripod point.  These probabilities are computed using $K_{tr}^{-1}$.  In this figure, the origin is at the lower left.}
\end{figure}

For a spanning tree with a tripod vertex at $v_t=(0,0)$, the probability that it has (directed) edge $e=(-1,0)(0,0)$ given that it is partially directed NE
is $K_{NE}(w,b) K_{NE}^{-1}(b,w)$, where $b=(-2,0), w=(-1,0)$.  Here $K_{NE}(w,b)=1$, and $K_{NE}^{-1}(b,w)$ is given
by the horizontal case of \eqref{eq:K_NE^{-1}} (recall that $w_1=(1,0)$).  The inverse Kasteleyn matrix values converge to
\begin{align*}
 K^{-1}(b,w) &\to -G_{\Z^2}((1,0))+G_{\Z^2}((0,0)) = \tfrac{1}{4} \\
 K^{-1}(b,w_1) &\to -G_{\Z^2}((2,0))+G_{\Z^2}((1,0)) = (1-\tfrac{2}{\pi}) - (\tfrac{1}{4})\,,
\end{align*}
so the conditional probability that the tree has edge $e$ is $\frac{2}{\pi}-\frac{1}{2}$.
The probability that the tripod vertex $(0,0)$ has (directed) edge $e=(-1,0)(0,0)$, not conditioning on the NE direction,
is then $\frac{1}{\pi}-\frac14.$

The probability that a tripod vertex at $(0,0)$ has degree $4$ is $4$ times the probability of directed edge $e$, given a tripod at $(0,0)$,
which is $\frac{4}{\pi}-1$.
Thus the expected degree of the tripod vertex is thus $2+\frac{4}{\pi}$.

See Figure~\ref{tripodprobs} for other directed edge values. Directed dual edge probabilities can be similarly computed.

\section{Triangular lattice}\label{triangularsection}

Here we do similar calculations for the triangular lattice.
(Additional ideas are required to get the full spanning tree trunk measure; however we
do prove the geometric runs property.)
While the values of the Green's function on the triangular
lattice are in $\Q\oplus\Q \sqrt{3}/\pi$,
the values of the antisymmetric Green's function for the triangular lattice
on the branched double-cover of the plane turn out to be in $\Q\oplus\Q\sqrt{3}$
(see Figure~\ref{fig:Green-triangular}).

\begin{figure}[t]
\renewcommand{\rsbox}[1]{\restrictboxwidth{$#1$}{34pt}}
\centerline{\begin{tikzpicture}
\begin{axis}[xmin=-4.7,xmax=4.7,ymin=-.5,ymax=4.2,x={(1.5cm,0)}, y={(0,1.5cm)}, axis lines=none]
\foreach \j in {-9,...,9} {\edef\temp{\noexpand\addplot[domain=0:4.2,gray!30!white,thick]({\j-x},1.732*x);\noexpand\addplot[domain=0:4.2,gray!30!white,thick]({\j+x},1.732*x);}\temp} ;
\foreach \j in {0,...,5} {\edef\temp{\noexpand\addplot[domain=-4.7:4.7,gray!30!white,thick](x,{1.732/2*\j});}\temp};
\rnode{0,0}{2{-}\sqrt{3}}
\rnode{-1.00000,0}{0}
\rnode{-0.5,0.866025}{\frac{3 \sqrt{3}}{2}{-}\frac{5}{2}}
\rnode{0.5,0.866025}{1{-}\frac{\sqrt{3}}{2}}
\rnode{1.00000,0}{14{-}8 \sqrt{3}}
\rnode{-1.5,0.866025}{\frac{7}{2}{-}2 \sqrt{3}}
\rnode{0,1.73205}{9 \sqrt{3}{-}\frac{31}{2}}
\rnode{1.5,0.866025}{\frac{73 \sqrt{3}}{4}{-}\frac{63}{2}}
\rnode{-2.00000,0}{0}
\rnode{-1.00000,1.73205}{\frac{7 \sqrt{3}}{2}{-}6}
\rnode{1.00000,1.73205}{\frac{79}{2}{-}\frac{91 \sqrt{3}}{4}}
\rnode{2.00000,0}{143{-}\frac{165 \sqrt{3}}{2}}
\rnode{-2.5,0.866025}{\frac{7}{4}{-}\sqrt{3}}
\rnode{-2.00000,1.73205}{\frac{111}{4}{-}16 \sqrt{3}}
\rnode{-0.5,2.59808}{\frac{85}{2}{-}\frac{49 \sqrt{3}}{2}}
\rnode{0.5,2.59808}{\frac{387 \sqrt{3}}{4}{-}\frac{335}{2}}
\rnode{2.00000,1.73205}{9 \sqrt{3}{-}\frac{31}{2}}
\rnode{2.5,0.866025}{\frac{857 \sqrt{3}}{4}{-}371}
\rnode{-3.00000,0}{0}
\rnode{-1.5,2.59808}{53 \sqrt{3}{-}\frac{367}{4}}
\rnode{1.5,2.59808}{466{-}269 \sqrt{3}}
\rnode{3.00000,0}{1649{-}952 \sqrt{3}}
\rnode{-3.00000,1.73205}{\frac{67 \sqrt{3}}{2}{-}58}
\rnode{0,3.46410}{398{-}\frac{919 \sqrt{3}}{4}}
\rnode{3.00000,1.73205}{\frac{6443}{8}{-}\frac{7439 \sqrt{3}}{16}}
\rnode{-3.5,0.866025}{\frac{149}{4}{-}\frac{43 \sqrt{3}}{2}}
\rnode{-2.5,2.59808}{317{-}183 \sqrt{3}}
\rnode{-1.00000,3.46410}{\frac{551}{4}{-}\frac{159 \sqrt{3}}{2}}
\rnode{1.00000,3.46410}{\frac{2235 \sqrt{3}}{2}{-}\frac{3871}{2}}
\rnode{2.5,2.59808}{\frac{9251 \sqrt{3}}{16}{-}\frac{8011}{8}}
\rnode{3.5,0.866025}{\frac{44123 \sqrt{3}}{16}{-}\frac{38211}{8}}
\rnode{-4.00000,0}{0}
\rnode{-2.00000,3.46410}{\frac{1235 \sqrt{3}}{2}{-}\frac{2139}{2}}
\rnode{2.00000,3.46410}{\frac{47011}{8}{-}\frac{54283 \sqrt{3}}{16}}
\rnode{4.00000,0}{\frac{80183}{4}{-}\frac{92587 \sqrt{3}}{8}}
\rnode{-4.00000,1.73205}{\frac{33 \sqrt{3}}{2}{-}\frac{457}{16}}
\rnode{-3.5,2.59808}{406 \sqrt{3}{-}\frac{11251}{16}}
\rnode{3.5,2.59808}{\frac{3191}{8}{-}\frac{921 \sqrt{3}}{4}}
\rnode{4.00000,1.73205}{\frac{84781}{8}{-}\frac{12237 \sqrt{3}}{2}}
\rnode{-4.5,0.866025}{\frac{4933}{16}{-}178 \sqrt{3}}
\rnode{-3.00000,3.46410}{\frac{60747}{16}{-}2192 \sqrt{3}}
\rnode{3.00000,3.46410}{7817 \sqrt{3}{-}\frac{108315}{8}}
\rnode{-4.5,2.59808}{\frac{2977}{2}{-}\frac{6875 \sqrt{3}}{8}}
\rnode{-4.00000,3.46410}{\frac{44963 \sqrt{3}}{8}{-}\frac{38939}{4}}
\rnode{4.00000,3.46410}{\frac{453335}{16}{-}\frac{523465 \sqrt{3}}{32}}
\old{
\rnode{-4,0}{0}
\rnode{-3,0}{0}
\rnode{-2,0}{0}
\rnode{-1,0}{0}
\rnode{0,0}{2-\sqrt{3}}
\rnode{1,0}{14-8\sqrt{3}}
\rnode{2,0}{143-\frac{165\sqrt{3}}{2}}
\rnode{-2.5,.866}{\frac74-\sqrt{3}}
\rnode{-1.5,.866}{\frac72-2\sqrt3}
\rnode{-.5,.866}{-\frac52+\frac{3\sqrt3}{2}}
\rnode{.5,.866}{1-\frac{\sqrt3}{2}}
\rnode{1.5,.866}{-\frac{63}{2}+\frac{73\sqrt3}{4}}
}
\end{axis}
\end{tikzpicture}}
\caption{Triangular lattice Green's function $G((0,0),\cdot)$ branched around an edge incident to the origin.}
\label{fig:Green-triangular}
\end{figure}

\subsection{Green's function}

For the triangular lattice we let $(x,y)$ denote the vertex at position $x + y e^{2\pi i/3}$.
The Green's function $G((x,y))$ on the triangular grid has the formula
\[ -G((x,y))= \frac1{4\pi^2}\oint\!\!\!\oint \frac{1-z^xw^y}{6-z-1/z-w-1/w-zw-1/zw} \,\frac{dw}{iw}\,\frac{dz}{iz}.\]

We evaluate this along the $x$ axis as follows. First perform the contour integral over $w$ to obtain
\[G((x,0)) = \frac1{2\pi i}\oint \frac{1-z^x}{(1-z)\sqrt{1-14z+z^2}}\,dz.\]
Here the choice of branch of the square root is determined so that $G((x,0))\le 0$.

Taking differences we have
\begin{align*}
 G((x,0))-G((x+1,0)) &= \frac1{2\pi i}\oint \frac{z^x}{\sqrt{1-14z+z^2}}\,dz \\
\intertext{and substituting $z=e^{i\theta}$}
&= \frac1{2\pi i}\int_0^{2\pi} \frac{e^{i(x+1/2)\theta}}{\sqrt{14-2\cos\theta}}\,d\theta
\intertext{and using the $\theta$-$(2\pi-\theta)$ symmetry the imaginary part cancels}
&= \frac1{2\pi}\int_0^{2\pi} \frac{\sin[(x+1/2)\theta]}{\sqrt{14-2\cos\theta}}\,d\theta
\end{align*}

We collect these terms into a generating function
\begin{align*}
 \delta_+(u) &= \sum_{x=0}^\infty u^{2x+1} [G((x,0))-G((x+1,0))]
\end{align*}
and since
\[\sum_{x\ge0} u^{2x+1} \sin((x+\tfrac12)\theta)= u\frac{(1+u^2)\sin(\theta/2)}{1+u^4-2u^2\cos\theta},\]
we arrive at
\[\delta_+(u) = \sum_{x\ge0} (G((x,0))-G(x+1,0))u^{x+1/2}
 = -\frac1{2\pi}\int_0^{2\pi} \frac{u(1+u^2)\sin(\theta/2)}{(1+u^4-2u^2\cos\theta)\sqrt{14-2\cos\theta}} \,d\theta
\]
which has an explicit integral
\[\delta_+(u) = u \frac{\cot^{-1}\left(\frac{\sqrt{3}(1+u^2)}{\sqrt{1-14u^2+u^4}}\right)}{\pi \sqrt{1-14u^2+u^4}}.\]

The first few values are
\[\delta_+(u) =  \frac{1}6 u + \left(\frac76-\frac{2\sqrt{3}}{\pi}\right) u^{3} + \left(\frac{73}6-\frac{22\sqrt{3}}{\pi}\right) u^{5} +\cdots.\]

Note that $\delta_+(z)$ converges for $|z|\leq 1$ except at $z=\pm 1$,
and that since $\cot^{-1}$ is odd, we make same choice of sign both square roots.
The choice of $\cot^{-1}$ is the one which gives $\pi/6$ when $u=0$, so when $u\to\pm i$, the $\cot^{-1}$ converges to $\pi/2$,
and $\delta_+(u)$ converges to $\pm i/8$.
When $u=e^{i\theta}$ we have
\[\frac{\sqrt{3}(1+u^2)}{\sqrt{u^4-14u^2+1}} =
-i \frac{2\sqrt{3}\cos(\theta)}{\sqrt{14-2\cos(2\theta)}}\]
which is pure imaginary with modulus less than $1$.
Using $\cot^{-1}(-ix) = \frac{i}{2}\log(\frac{x+1}{x-1})$, we have
\begin{align*}
\delta_+(u) &= \frac{i\log(\frac{x+1}{x-1})}{2\pi\sqrt{u^2-14+1/u^2}} \\
 &= \frac{\pm \pi i + \log(\frac{1+x}{1-x})}{2\pi\sqrt{14-2\cos(2\theta)}}
\end{align*}
where the sign is given by the sign of $\Im u$.

For $u\in S^1$, adding the negative powers of $u$ corresponds to taking $2i\times$ the imaginary part.
Thus
\[
\delta(e^{i\theta}) = \delta_-(e^{i\theta})+\delta_+(e^{i\theta}) = \frac{i}{\sqrt{14-2\cos\theta}} \times \sign\sin\theta
\]

\subsection{Triangular slit plane}

Let $D=\{(k,0)\;:\;k<0\}$ be the vertices on the negative real axis.
As in the case of the square grid we compute the Green's function $G_D$ for the triangular
grid with Dirichlet boundary on $D$.
As in that case, we need to find $C(u)$ analytic outside the disk so that
$C(u) \delta(u)\times u$ is analytic in the unit disk.
Notice that $14-u^2-1/u^2$ has roots at $\pm2\pm\sqrt{3}$, so we let $\alpha=2-\sqrt{3}$, and factor
\[
14-u^2-1/u^2 = (1-\alpha^2 u^2)(1-\alpha^2 / u^2)/\alpha^2
\]
which allows us to guess the generating functions
\[C_*(u) = \sqrt{(1-1/u^2)(1-\alpha^2/u^2)}\quad\quad\text{and}\quad\quad D_*(u) = -\frac{\alpha}{u}\sqrt{\frac{1-u^2}{1-\alpha^2 u^2}}\,.\]
Since $|\alpha|<1$, $u D_*(u)$ is analytic inside the unit disk and $C_*(u)$ is analytic outside the unit disk.
The power series coefficients for $\sqrt{1-u^2}$ are all negative except for the constant term, and they sum to $0$,
so the coefficients are absolutely summable.  The coefficients for $1/\sqrt{1-\alpha^2 u^2}$ are also absolutely summable,
since the function is analytic in a larger disk.  Thus the series for $D_*(u)$ converges absolutely when $|u|=1$.
Similarly, the series for $C_*(u)$ converges absolutely when $|u|=1$.

Taking the ratio
\[
\frac{D_*(u)}{C_*(u)} = \pm \frac{i}{(1-\alpha^2/u^2)(1-\alpha^2 u^2)/\alpha^2} = \pm \frac{i}{\sqrt{14-u^2-1/u^2}}
\]
gives $\pm\delta(u)$ when $|u|=1$.  By moving $u$ from $\infty$ to $\pm i$ on the imaginary axis we see $\arg C_*(\pm i) = 0$,
and similarly $\arg D_*(\pm i) = \pm \pi/2$, so in fact
\[
C_*(u) \delta(u) = D_*(u)
\]
As in the square lattice case, we deduce that $C(u)=C_*(u)$ and $D(u)=D_*(u)$.
Thus the voltage generating function $V(u) = -u/(1-u^2) D_*(u)$ is given by
\[
V(u) = \alpha\frac{1}{\sqrt{(1-u^2)(1-\alpha^2 u^2)}}
 = (2-\sqrt{3}) + (14-8\sqrt{3})u^2 + \left(143-\frac{165\sqrt{3}}2\right)u^4+ \cdots.\]

\subsection{Branching around a face versus an edge}

We obtained the values of the Green's function along the real axis.
We can use harmonicity to obtain the remaining values if there is another
line along which we know already know the values.
One way to obtain the values along another line is to deform the zipper
so that to obtain the Green's function $\tilde G$ for the double-cover branched around a face
instead of an edge, and then use the symmetry of $\tilde G$.

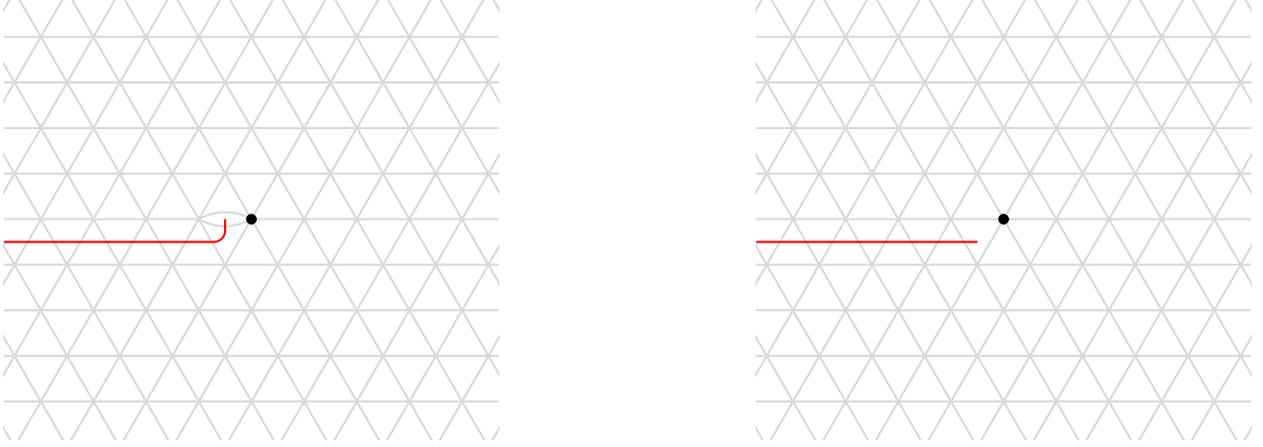
\begin{figure}[t]
\newcommand{\rt}{1.7320508075688772935}
\centerline{\begin{tikzpicture}
\begin{scope}
\begin{axis}[xmin=-4.7,xmax=4.7,ymin=-4.2,ymax=4.2,x={(.7cm,0)}, y={(0,.7cm)}, axis lines=none]
\foreach \j in {-9,...,9} {\edef\temp{\noexpand\addplot[domain=-4.2:4.2,gray!30!white,thick]({\j-x},\rt*x);\noexpand\addplot[domain=-4.2:4.2,gray!30!white,thick]({\j+x},\rt*x);}\temp} ;
\foreach \j in {-5,...,-1} {\edef\temp{\noexpand\addplot[domain=-4.7:4.7,gray!30!white,thick](x,{\rt/2*\j});}\temp};
\foreach \j in {1,...,5} {\edef\temp{\noexpand\addplot[domain=-4.7:4.7,gray!30!white,thick](x,{\rt/2*\j});}\temp};
\addplot[domain=-4.7:-1,gray!30!white,thick](x,0);
\addplot[domain=0:4.7,gray!30!white,thick](x,0);
\end{axis}
\begin{scope}[scale=.7,shift={(4.7,4.2)}]
\draw [gray!30!white,thick](0,0) arc (60:120:1);
\draw [gray!30!white,thick](0,0) arc (-60:-120:1);
\draw [red,thick] (-0.5,0) -- (-.5,{-\rt/8}) arc (0:-90:{\rt/8}) -- (-4.7,{-\rt/4});
\draw [fill=black] (0,0) circle (.5ex);
\end{scope}
\end{scope}
\begin{scope}[shift={(10,0)}]
\begin{axis}[xmin=-4.7,xmax=4.7,ymin=-4.2,ymax=4.2,x={(.7cm,0)}, y={(0,.7cm)}, axis lines=none]
\foreach \j in {-9,...,9} {\edef\temp{\noexpand\addplot[domain=-4.2:4.2,gray!30!white,thick]({\j-x},\rt*x);\noexpand\addplot[domain=-4.2:4.2,gray!30!white,thick]({\j+x},\rt*x);}\temp} ;
\foreach \j in {-5,...,5} {\edef\temp{\noexpand\addplot[domain=-4.7:4.7,gray!30!white,thick](x,{\rt/2*\j});}\temp};
\end{axis}
\begin{scope}[scale=.7,shift={(4.7,4.2)}]
\draw [red,thick] (-0.5,{-\rt/4})-- (-4.7,{-\rt/4});
\draw [fill=black] (0,0) circle (.5ex);
\end{scope}
\end{scope}
\end{tikzpicture}}
\caption{Triangular lattice with the zipper originating in the middle of a split edge from $0$ to $-1$ (left), and with the zipper modified so that it originates from the triangle below that edge.}
\label{fig:Green-triangular-edge}
\end{figure}

Suppose we split the edge from $(-1,0)(0,0)$ into two edges
with half the conductance, with a zipper originating in the face between these
two edges, and proceeding left along the negative real axis but below the vertices
on the axis.  Let $G$ be the resulting Green's function.
If we deform the zipper to run above the axis, then the values at the
on-axis vertices get negated, and then symmetry implies that the voltages along
the negative real axis are $0$.  Thus $G$ coincides with the function that
we were working on computing in the previous section.

Here for notational convenience we identify a vertex $v$ with its position $x + y e^{2\pi i/3}$ (with $x,y\in\Z$)
in the complex plane using the natural embedding.

We move the endpoint of the zipper to the triangle below the edge(s) $(-1,0)(0,0)$.
\[
  \tilde G(u,v) = G(u,v) + a G(p,v) + b G(q,v)
\]
where $p=(0,0)$ and $q=(-1,0)$.  To find $a$ and $b$, we
use the formulas from earlier with $u=p$, but since we are moving the zipper across
a split edge with only half the conductance, the $2$'s in the formula are dropped.
We have already $G_{p,p}=2-\sqrt{3}$ and $G_{p,q}=0$.  By deforming the
(original edge-branched) zipper to run below the positive real axis, we see that
\[
G_{q,v}=G_{p,-\bar v-1} \times \begin{cases}1 & \Im v\geq 0 \\ -1 & \Im v < 0\,.\end{cases}
\]
In particular $G_{q,q}=G_{p,p}$ (and of course $G_{q,p}=G_{p,q}$).
Solving
\begin{align*}
a + b G_{q,q} &= 0 \\
b + (1+a) G_{p,p} &= 0
\end{align*}
we find $a = -1/2 + 1/\sqrt{3}$ and $b=-1/(2\sqrt{3})$, so
\[
\tilde G(0,v) = \left(\frac{1}{2}+\frac{1}{\sqrt{3}}\right)G_D(0,v) - \frac{1}{2\sqrt{3}} G_D(0,-\bar v-1)\times\begin{cases}1 & \Im v\geq 0 \\ -1 & \Im v<0\end{cases}
\]

For purposes of symmetry, we can deform the zipper $30^\circ$ counterclockwise, so that $0$ and the center of the triangle containing the zipper's start
are on the line of the zipper.

Then by construction and by symmetry,
\begin{equation} \label{eq:G-triangular-branched-face}
\begin{aligned}
\sum_{x\geq 0} \tilde G(x,0) u^x = \sum_{y\geq 0} \tilde G(0,y) u^y &= \frac{1}{\sqrt{12 (1-u)(1-\alpha u)}} \\
\sum_{x<0} \tilde G(x,0) u^x = \sum_{y<0} \tilde G(0,y) u^y &= \frac{1/\sqrt{3}-1/2}{u\sqrt{(1-1/u)(1-\alpha/u)}}
\end{aligned}
\end{equation}
Observe also that for $k\geq 1$, $\tilde G(-k,k) = \tilde G(-k,k-1)$.

Equation~\eqref{eq:G-triangular-branched-face} gives two lines along which we know $\tilde G(0,v)$.
Given $\tilde G$ on these two lines, the Laplacian relation determines $\tilde G(0,v)$ everywhere else (see Figure~\ref{fig:Green-triangular-face}).
We can then deform the zipper again to obtain the Green's function $G(0,v)$ with the zipper originating in the middle of an edge (Figure~\ref{fig:Green-triangular-edge}),
or move the zipper to evaluate either $G(u,v)$ or $\tilde G(u,v)$ at arbitrary pairs of vertices.

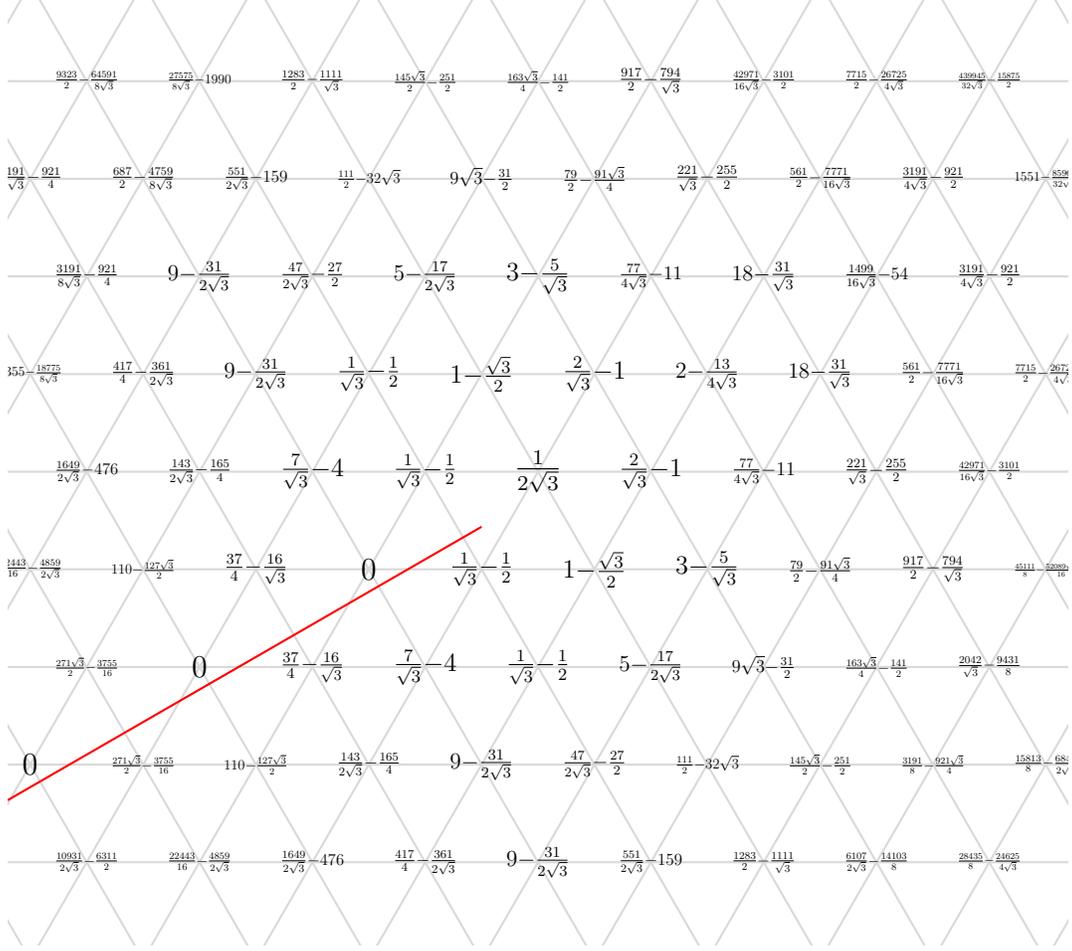
\begin{figure}
\centerline{\begin{tikzpicture}
\begin{axis}[xmin=-4.7,xmax=4.7,ymin=-4.2,ymax=4.2,x={(1.5cm,0)}, y={(0,1.5cm)}, axis lines=none]
\foreach \j in {-9,...,9} {\edef\temp{\noexpand\addplot[domain=-4.2:4.2,gray!30!white,thick]({\j-x},1.732*x);\noexpand\addplot[domain=-4.2:4.2,gray!30!white,thick]({\j+x},1.732*x);}\temp} ;
\foreach \j in {-5,...,5} {\edef\temp{\noexpand\addplot[domain=-4.7:4.7,gray!30!white,thick](x,{1.732/2*\j});}\temp};
\addplot[domain=-0.5:-4.7,red,thick](x,{x/1.732-0.2});
\rnode{0,0}{\frac{1}{2 \sqrt{3}}}
\rnode{-1.00000,0}{\frac{1}{\sqrt{3}}{-}\frac{1}{2}}
\rnode{-0.5,-0.866025}{\frac{1}{\sqrt{3}}{-}\frac{1}{2}}
\rnode{-0.5,0.866025}{1{-}\frac{\sqrt{3}}{2}}
\rnode{0.5,-0.866025}{1{-}\frac{\sqrt{3}}{2}}
\rnode{0.5,0.866025}{\frac{2}{\sqrt{3}}{-}1}
\rnode{1.00000,0}{\frac{2}{\sqrt{3}}{-}1}
\rnode{-1.5,-0.866025}{0}
\rnode{-1.5,0.866025}{\frac{1}{\sqrt{3}}{-}\frac{1}{2}}
\rnode{0,-1.73205}{\frac{1}{\sqrt{3}}{-}\frac{1}{2}}
\rnode{0,1.73205}{3{-}\frac{5}{\sqrt{3}}}
\rnode{1.5,-0.866025}{3{-}\frac{5}{\sqrt{3}}}
\rnode{1.5,0.866025}{2{-}\frac{13}{4 \sqrt{3}}}
\rnode{-2.00000,0}{\frac{7}{\sqrt{3}}{-}4}
\rnode{-1.00000,-1.73205}{\frac{7}{\sqrt{3}}{-}4}
\rnode{-1.00000,1.73205}{5{-}\frac{17}{2 \sqrt{3}}}
\rnode{1.00000,-1.73205}{5{-}\frac{17}{2 \sqrt{3}}}
\rnode{1.00000,1.73205}{\frac{77}{4 \sqrt{3}}{-}11}
\rnode{2.00000,0}{\frac{77}{4 \sqrt{3}}{-}11}
\rnode{-2.5,-0.866025}{\frac{37}{4}{-}\frac{16}{\sqrt{3}}}
\rnode{-2.5,0.866025}{9{-}\frac{31}{2 \sqrt{3}}}
\rnode{-2.00000,-1.73205}{\frac{37}{4}{-}\frac{16}{\sqrt{3}}}
\rnode{-2.00000,1.73205}{\frac{47}{2 \sqrt{3}}{-}\frac{27}{2}}
\rnode{-0.5,-2.59808}{9{-}\frac{31}{2 \sqrt{3}}}
\rnode{-0.5,2.59808}{9 \sqrt{3}{-}\frac{31}{2}}
\rnode{0.5,-2.59808}{\frac{47}{2 \sqrt{3}}{-}\frac{27}{2}}
\rnode{0.5,2.59808}{\frac{79}{2}{-}\frac{91 \sqrt{3}}{4}}
\rnode{2.00000,-1.73205}{9 \sqrt{3}{-}\frac{31}{2}}
\rnode{2.00000,1.73205}{18{-}\frac{31}{\sqrt{3}}}
\rnode{2.5,-0.866025}{\frac{79}{2}{-}\frac{91 \sqrt{3}}{4}}
\rnode{2.5,0.866025}{18{-}\frac{31}{\sqrt{3}}}
\rnode{-3.00000,0}{\frac{143}{2 \sqrt{3}}{-}\frac{165}{4}}
\rnode{-1.5,-2.59808}{\frac{143}{2 \sqrt{3}}{-}\frac{165}{4}}
\rnode{-1.5,2.59808}{\frac{111}{2}{-}32 \sqrt{3}}
\rnode{1.5,-2.59808}{\frac{111}{2}{-}32 \sqrt{3}}
\rnode{1.5,2.59808}{\frac{221}{\sqrt{3}}{-}\frac{255}{2}}
\rnode{3.00000,0}{\frac{221}{\sqrt{3}}{-}\frac{255}{2}}
\rnode{-3.00000,-1.73205}{0}
\rnode{-3.00000,1.73205}{9{-}\frac{31}{2 \sqrt{3}}}
\rnode{0,-3.46410}{9{-}\frac{31}{2 \sqrt{3}}}
\rnode{0,3.46410}{\frac{163 \sqrt{3}}{4}{-}\frac{141}{2}}
\rnode{3.00000,-1.73205}{\frac{163 \sqrt{3}}{4}{-}\frac{141}{2}}
\rnode{3.00000,1.73205}{\frac{1499}{16 \sqrt{3}}{-}54}
\rnode{-3.5,-0.866025}{110{-}\frac{127 \sqrt{3}}{2}}
\rnode{-3.5,0.866025}{\frac{417}{4}{-}\frac{361}{2 \sqrt{3}}}
\rnode{-2.5,-2.59808}{110{-}\frac{127 \sqrt{3}}{2}}
\rnode{-2.5,2.59808}{\frac{551}{2 \sqrt{3}}{-}159}
\rnode{-1.00000,-3.46410}{\frac{417}{4}{-}\frac{361}{2 \sqrt{3}}}
\rnode{-1.00000,3.46410}{\frac{145 \sqrt{3}}{2}{-}\frac{251}{2}}
\rnode{1.00000,-3.46410}{\frac{551}{2 \sqrt{3}}{-}159}
\rnode{1.00000,3.46410}{\frac{917}{2}{-}\frac{794}{\sqrt{3}}}
\rnode{2.5,-2.59808}{\frac{145 \sqrt{3}}{2}{-}\frac{251}{2}}
\rnode{2.5,2.59808}{\frac{561}{2}{-}\frac{7771}{16 \sqrt{3}}}
\rnode{3.5,-0.866025}{\frac{917}{2}{-}\frac{794}{\sqrt{3}}}
\rnode{3.5,0.866025}{\frac{561}{2}{-}\frac{7771}{16 \sqrt{3}}}
\rnode{-4.00000,0}{\frac{1649}{2 \sqrt{3}}{-}476}
\rnode{-2.00000,-3.46410}{\frac{1649}{2 \sqrt{3}}{-}476}
\rnode{-2.00000,3.46410}{\frac{1283}{2}{-}\frac{1111}{\sqrt{3}}}
\rnode{2.00000,-3.46410}{\frac{1283}{2}{-}\frac{1111}{\sqrt{3}}}
\rnode{2.00000,3.46410}{\frac{42971}{16 \sqrt{3}}{-}\frac{3101}{2}}
\rnode{4.00000,0}{\frac{42971}{16 \sqrt{3}}{-}\frac{3101}{2}}
\rnode{-4.00000,-1.73205}{\frac{271 \sqrt{3}}{2}{-}\frac{3755}{16}}
\rnode{-4.00000,1.73205}{\frac{3191}{8 \sqrt{3}}{-}\frac{921}{4}}
\rnode{-3.5,-2.59808}{\frac{271 \sqrt{3}}{2}{-}\frac{3755}{16}}
\rnode{-3.5,2.59808}{\frac{687}{2}{-}\frac{4759}{8 \sqrt{3}}}
\rnode{3.5,-2.59808}{\frac{3191}{8}{-}\frac{921 \sqrt{3}}{4}}
\rnode{3.5,2.59808}{\frac{3191}{4 \sqrt{3}}{-}\frac{921}{2}}
\rnode{4.00000,-1.73205}{\frac{2042}{\sqrt{3}}{-}\frac{9431}{8}}
\rnode{4.00000,1.73205}{\frac{3191}{4 \sqrt{3}}{-}\frac{921}{2}}
\rnode{-4.5,-0.866025}{\frac{22443}{16}{-}\frac{4859}{2 \sqrt{3}}}
\rnode{-4.5,0.866025}{1355{-}\frac{18775}{8 \sqrt{3}}}
\rnode{-3.00000,-3.46410}{\frac{22443}{16}{-}\frac{4859}{2 \sqrt{3}}}
\rnode{-3.00000,3.46410}{\frac{27575}{8 \sqrt{3}}{-}1990}
\rnode{3.00000,-3.46410}{\frac{6107}{2 \sqrt{3}}{-}\frac{14103}{8}}
\rnode{3.00000,3.46410}{\frac{7715}{2}{-}\frac{26725}{4 \sqrt{3}}}
\rnode{4.5,-0.866025}{\frac{45111}{8}{-}\frac{52089 \sqrt{3}}{16}}
\rnode{4.5,0.866025}{\frac{7715}{2}{-}\frac{26725}{4 \sqrt{3}}}
\rnode{5.00000,0}{\frac{134483}{4 \sqrt{3}}{-}\frac{155287}{8}}
\rnode{-4.5,-2.59808}{0}
\rnode{-4.5,2.59808}{\frac{3191}{8 \sqrt{3}}{-}\frac{921}{4}}
\rnode{4.5,-2.59808}{\frac{15813}{8}{-}\frac{6847}{2 \sqrt{3}}}
\rnode{4.5,2.59808}{1551{-}\frac{85961}{32 \sqrt{3}}}
\rnode{-4.00000,-3.46410}{\frac{10931}{2 \sqrt{3}}{-}\frac{6311}{2}}
\rnode{-4.00000,3.46410}{\frac{9323}{2}{-}\frac{64591}{8 \sqrt{3}}}
\rnode{4.00000,-3.46410}{\frac{28435}{8}{-}\frac{24625}{4 \sqrt{3}}}
\rnode{4.00000,3.46410}{\frac{439945}{32 \sqrt{3}}{-}\frac{15875}{2}}
\rnode{5.00000,-1.73205}{\frac{428875}{16 \sqrt{3}}{-}\frac{123805}{8}}
\rnode{5.00000,1.73205}{\frac{439945}{32 \sqrt{3}}{-}\frac{15875}{2}}
\rnode{5.00000,-3.46410}{\frac{13051 \sqrt{3}}{2}{-}\frac{180839}{16}}
\rnode{5.00000,3.46410}{13051{-}\frac{180839}{8 \sqrt{3}}}
\end{axis}
\end{tikzpicture}}
\caption{Green's function $\tilde G(0,v)$ on the triangular lattice with a zipper originating in a triangle incident to the origin.}
\label{fig:Green-triangular-face}
\end{figure}

\phantomsection
\pdfbookmark[1]{References}{bib}
\bibliographystyle{hmralphaabbrv}
\bibliography{zipper}

\end{document}